 \documentclass[11pt]{article}
\usepackage{amsthm}
\usepackage{amsfonts}
\usepackage{xr}
\usepackage{graphicx}
\usepackage{amsmath,bm}
\usepackage{epsfig,subfigure}
\usepackage{xcolor}
\usepackage{tikz}
\usetikzlibrary{arrows.meta,calc,decorations.markings}
\usepackage{amssymb}
\usepackage{theoremref}
\usepackage{hyperref}
\hypersetup{hidelinks}
\usepackage{mathrsfs}

\newcommand{\N}{\mathbb{N}} 
\newcommand{\Z}{\mathbb{Z}} 
\newcommand{\R}{\mathbb{R}} 
\newcommand{\T}{\mathbb{T}} 

\newcommand{\Dcal}{\mathcal{D}}

\newcommand{\Lcal}{\mathcal{L}}
\newcommand{\Mcal}{\mathcal{M}}

\def\epsilon{\varepsilon}

\def\tilde{\widetilde}

\def\XXint#1#2#3{{\setbox0=\hbox{$#1{#2#3}{\int}$ }
		\vcenter{\hbox{$#2#3$ }}\kern-.6\wd0}}

\newcommand{\be}{\begin{equation}}
\newcommand{\ee}{\end{equation}}
\newcommand{\baa}{\begin{array}}
\newcommand{\eaa}{\end{array}}
\newcommand{\ba}{\begin{eqnarray}}
\newcommand{\ea}{\end{eqnarray}}

\newtheorem{theorem}{Theorem}[section]     
\newtheorem{assumption}{Assumption}[section]
\newtheorem{corollary}{Corollary}[section]
\newtheorem{proposition}{Proposition}[section]
\newtheorem{lemma}{Lemma}[section]

\newtheorem{example}{Example}[section]

\newtheorem{remark}{Remark}[section]

\numberwithin{equation}{section}
\newcommand{\supp}{\operatorname{supp}}

\definecolor{Ink}{HTML}{17384D}
\definecolor{Flow}{HTML}{126B82}
\definecolor{Sink}{HTML}{236DB0}
\definecolor{Source}{HTML}{BE523B}
\definecolor{Deg}{HTML}{B38025}
\definecolor{Fixed}{HTML}{86538F}
\definecolor{Muted}{HTML}{697B85}
\definecolor{Soft}{HTML}{EDF4F6}
\newcommand{\arrarc}[2]{\draw[Flow,line width=1.15pt,-{Stealth[length=2.1mm]}] (#1:1.22) arc[start angle=#1,end angle=#2,radius=1.22];}
\newcommand{\degpoint}[1]{\draw[Deg,fill=white,line width=1.3pt] (#1:1.22) circle (2.6pt);}
\newcommand{\sourcepoint}[1]{\draw[Source,fill=white,line width=1.3pt] (#1:1.22) circle (2.6pt);}
\newcommand{\sinkpoint}[1]{\fill[Sink] (#1:1.22) circle (2.8pt);}
\newcommand{\CircleFlow}[1]{%
	\begin{tikzpicture}[scale=1.05]
		\path[use as bounding box] (-1.75,-1.70) rectangle (1.90,1.65);
		\draw[Muted!45,line width=.7pt] (0,0) circle(1.22);
		\ifcase#1\or
		\arrarc{10}{85}\arrarc{100}{175}\arrarc{190}{265}\arrarc{280}{355}
		\node at (0,0) {$q>0$};
		\or
		\arrarc{15}{100}\arrarc{125}{230}\arrarc{255}{345}\degpoint{0}
		\node[right] at (1.30,0) {$0$};
		\node at (0,0) {$q\geq0$};
		\or
		\arrarc{15}{80}\arrarc{100}{165}\arrarc{195}{260}\arrarc{280}{345}
		\degpoint{0}\degpoint{180}
		\node[right] at (1.30,0) {$0$};
		\node[left] at (-1.30,0) {$\frac12$};
		\node at (0,0) {$q\geq0$};
		\or
		\arrarc{15}{80}\arrarc{100}{165}\arrarc{345}{280}\arrarc{260}{195}
		\sourcepoint{0}\sinkpoint{180}
		\node[right] at (1.30,0) {$0$};
		\node[left] at (-1.30,0) {$\frac12$};
		\node at (0,.35) {$+$};\node at (0,-.35) {$-$};
		\or
		\draw[Fixed,line width=4pt] (45:1.22) arc[start angle=45,end angle=135,radius=1.22];
		\arrarc{150}{225}\arrarc{250}{330}\arrarc{0}{30}
		\node[Fixed] at (0,1.55) {$K$};
		\node at (0,0) {$q=0\text{ on }K$};
		\node at (0,-.43) {$q>0\text{ off }K$};
		\or
		\draw[Fixed,line width=3.2pt] (0.0000:1.22) arc[start angle=0.0000,end angle=13.3333,radius=1.22];
		\draw[Fixed,line width=3.2pt] (26.6667:1.22) arc[start angle=26.6667,end angle=40.0000,radius=1.22];
		\draw[Fixed,line width=3.2pt] (80.0000:1.22) arc[start angle=80.0000,end angle=93.3333,radius=1.22];
		\draw[Fixed,line width=3.2pt] (106.6667:1.22) arc[start angle=106.6667,end angle=120.0000,radius=1.22];
		\draw[Fixed,line width=3.2pt] (240.0000:1.22) arc[start angle=240.0000,end angle=253.3333,radius=1.22];
		\draw[Fixed,line width=3.2pt] (266.6667:1.22) arc[start angle=266.6667,end angle=280.0000,radius=1.22];
		\draw[Fixed,line width=3.2pt] (320.0000:1.22) arc[start angle=320.0000,end angle=333.3333,radius=1.22];
		\draw[Fixed,line width=3.2pt] (346.6667:1.22) arc[start angle=346.6667,end angle=360.0000,radius=1.22];
		\arrarc{135}{225}\arrarc{45}{75}\arrarc{285}{315}
		\node[Fixed] at (0,.25) {$K$};
		\node at (0,-.30) {$q>0\text{ off }K$};
		\fi
\end{tikzpicture}}
\newcommand{\ThreeGraph}[1]{%
	\begin{tikzpicture}[x=1.15cm,y=1cm]
		\path[use as bounding box] (-.5,-.85) rectangle (5.5,1.0);
		\ifnum#1=1
		\foreach \x in {0,2.5,5}{\draw[Deg,fill=white,line width=1.2pt] (\x,0) circle(3pt);}
		\draw[Flow,thick,-Stealth] (.18,0)--(2.32,0);
		\draw[Flow,thick,-Stealth] (2.68,0)--(4.82,0);
		\else
		\foreach \x in {0,5}{\draw[Source,fill=white,line width=1.2pt] (\x,0) circle(3pt);}
		\fill[Sink] (2.5,0) circle(3pt);
		\draw[Flow,thick,-Stealth] (.18,0)--(2.32,0);
		\draw[Flow,thick,-Stealth] (4.82,0)--(2.68,0);
		\fi
		\node[below=7pt] at (0,0) {$Z_0$};
		\node[below=7pt] at (2.5,0) {$Z_{L/2}$};
		\node[below=7pt] at (5,0) {$Z_L$};
		\node[above] at (1.25,.15) {$u_1(t,\cdot)$};
		\node[above] at (3.75,.15) {$u_2(t,\cdot)$};
\end{tikzpicture}}

\newcommand{\RealizationPanel}[4]{%
	\par\medskip\noindent
	\begin{minipage}{\linewidth}
		\normalfont
		\noindent\textbf{#1}\par\smallskip
		\noindent
		\begin{minipage}[c]{.25\linewidth}
			\centering\CircleFlow{#2}
		\end{minipage}\hfill
		\begin{minipage}[c]{.71\linewidth}
			\centering
			#3\par\smallskip
			\raggedright
			#4
		\end{minipage}
	\end{minipage}\par\medskip
}

\begin{document}
\date{}
\title{\bf{Bistable fronts in periodic media: phase reduction and realization of circle flows}}
\author{Weiwei Ding\footnote{School of Mathematical Sciences, South China Normal University, Guangzhou 510631, China (dingweiwei@m.scnu.edu.cn).} \qquad Xing Liang\footnote{School of Mathematical Sciences, University of Science and Technology of China, Hefei, Anhui, 230026, China (xliang@ustc.edu.cn).}  
\qquad Linfeng Xu\footnote{School of Mathematical Sciences, University of Science and Technology of China, Hefei, Anhui, 230026, China (xlf\_hzfy@mail.ustc.edu.cn)} 
}

\maketitle
\begin{abstract}
	We establish a phase reduction and a converse realization theorem for front dynamics in one-dimensional periodic reaction–diffusion–advection equations. For sufficiently large periods under local bistability, and sufficiently small periods when the homogenized equation admits a stationary front, we construct strictly ordered invariant front families that are normally exponentially attracting and locally complete for entire orbits. Modulo spatial-period translations, the dynamics on each family reduces exactly to a circle flow: rotations correspond to propagating pulsating fronts, equilibria to stationary fronts, and connecting orbits to transition fronts between stationary fronts. Conversely, every smooth circle flow can be realized, up to time-preserving topological conjugacy, as the quotient dynamics of such a family for a periodic, locally bistable equation with constant diffusion and no advection. Realizations with a Cantor set of stationary phases and a connecting orbit in each complementary interval illustrate the richness of front dynamics in these media. We also obtain first-order phase-velocity expansions in both period limits, quantitative propagation laws near zero frozen speed, and global exponential orbital stability of propagating pulsating fronts.
	\end{abstract}
\noindent\textbf{Keywords:} bistable front; periodic medium; invariant manifold;
circle flow; topological realization; homogenization; propagation failure.
\par\smallskip
\noindent\textbf{Mathematics Subject Classification:} 35K57, 35C07, 35B27, 37L25.
\medskip

\section{Introduction}
\label{sec:intro}
In this paper,
we investigate how spatial periodicity can cause a bistable interface to propagate through the medium,
remain pinned, 
or move between distinct stationary fronts.
We consider the one-dimensional reaction--diffusion--advection equation
with periodic coefficients and uniformly stable equilibria $1$ and $0$,
\begin{equation}\label{intro:pde}
	u_t=F_L(u):=\partial_x\big(a(x/L)u_x\big)+b(x/L)u_x+f(x/L,u),
	\qquad x\in\R,\quad L>0.
\end{equation}
Under the regime-specific hypotheses stated below,
we describe these behaviors within a unified phase-reduction framework for transition fronts connecting $1$ and $0$ in the large- and small-period regimes.
Following \cite{BH2012,DHZ2015}, 
a \emph{transition front} connecting $1$ to $0$ is a classical entire solution $0<u<1$ for which there is an interface position $\xi(t)$ such that
\[
u(t,\xi(t)+z)\to1\quad(z\to-\infty),\qquad
u(t,\xi(t)+z)\to0\quad(z\to+\infty),
\]
uniformly in $t\in\R$. A \emph{stationary front} is a time-independent transition front. 
A \emph{propagating pulsating front} is a transition front of the form
\[
u(t,x)=\Phi(x-c_Lt,x/L),\qquad
\Phi(\xi,y+1)=\Phi(\xi,y),\qquad
\Phi(-\infty,y)=1,\quad\Phi(+\infty,y)=0,
\]
where $c_L\ne0$, the limits are uniform in $y$, and $0<\Phi<1$. 
Throughout the paper, 
``pulsating front'' is used in the broad sense,
including both stationary fronts and propagating pulsating fronts.
A bounded function is called a \emph{front-like profile} if its limits at $-\infty$ and $+\infty$ are $1$ and $0$, respectively.

The study of front propagation in reaction--diffusion equations began with the works of Fisher  \cite{Fisher1937} and with Kolmogorov, Petrovskii and Piskunov \cite{KPP1937} on monostable problem.
Many subsequent works developed the theory of front propagation in homogeneous and heterogeneous media \cite{AronsonWeinberger1978,FreidlinGartner1979,BerestyckiHamel2002}.
For a homogeneous scalar bistable reaction, 
the traveling front has a selected speed rather than a continuum of admissible speeds. 
Fife and McLeod \cite{FifeMcLeod1977} studied the unique traveling front and convergence of suitable solutions to its translates.

In periodic media, front existence requires additional hypotheses.
Xin \cite{Xin1991} proved existence of bistable pulsating fronts for coefficients close to constants and, in one dimension, their stability.
Fang and Zhao \cite{FangZhao} developed an abstract existence theory for bistable traveling fronts of monotone semiflows, including applications to periodic habitats.

For small periods, Heinze \cite{Heinze2001} constructed propagating pulsating fronts in periodically perforated domains
from a nondegenerate traveling front of the homogenized equation.
For one-dimensional periodic reaction--diffusion equations without drift,
Ding, Hamel, and Zhao \cite{DHZ} proved existence of propagating pulsating fronts for sufficiently small periods
when the homogenized front has nonzero speed.
Ding and Xu \cite{DingXu2026} established small-period existence and homogenization results
for multidimensional equations with periodic advection and general diffusion,
also assuming a homogenized front with nonzero speed.

In the large-period regime, 
we assume \emph{local bistability}:
for each spatial phase $Y$, 
the frozen equation
\[
v_t=a(Y)v_{xx}+b(Y)v_x+f(Y,v)
\]
has the bistable structure with stable equilibria $0$ and $1$.
Its decreasing traveling wave $v(t,x)=\phi(Y,x-c(Y)t)$ defines the \emph{frozen profile} $\phi(Y,\cdot)$ and the \emph{frozen speed} $c(Y)$,
with $\phi(Y,0)=1/2$.
These frozen fronts provide a local reference for front motion in the periodic medium.
With $b=0$ and a general positive periodic diffusion coefficient $a$, Ding, Hamel, and Zhao \cite{DHZ} proved existence of propagating pulsating fronts
when the reaction integrals $\int_0^1 f(y,u)\,du$ have a common strict sign, 
under further assumptions.
Ding, Hamel, and Liang \cite{DHL2025} established the harmonic-mean speed limit and profile convergence under their respective hypotheses.
When the reaction integrals change sign, Ding, Hamel, and Zhao \cite{DHZ2015}
constructed multiple stationary fronts and transition fronts connecting them.

Ducrot \cite{Ducrot2016} obtained, under suitable hypotheses, either propagating pulsating fronts or stationary fronts in every direction
for a two-dimensional periodic cubic equation without drift, in both period limits.
Local bistability alone,
however,
does not guarantee transition-front existence.
Zlato\v{s} \cite{Zlatos2017} constructed spatially periodic bistable reactions with no transition front at all.

Reference front profiles also provide a basis for reduced descriptions of front motion.
Scheel and Tikhomirov \cite{ST2017} formulated an invariant-manifold framework
for positional dynamics and depinning in heterogeneous media.
In their discussion of slowly varying media,
they proposed a leading-order position law governed by the speed of the homogeneous front obtained by freezing the medium at the current front position,
while noting that a rigorous justification of this description was not known to them.
More recently, 
starting from homogeneous Allen--Cahn front profiles,
Bastiaansen, Doelman, and Kaper \cite{BastiaansenDoelmanKaper2026} derived leading-order position equations incorporating weak spatial heterogeneity and front interactions.

Our large-period theorem makes the frozen-speed description discussed by Scheel and Tikhomirov  \cite{ST2017} rigorous in the periodic setting.
Under the corresponding hypotheses of Subsection~\ref{intro:assumptions},
we construct an invariant front family whose exact position velocity has the frozen front speed as its leading term,
even when that speed vanishes or changes sign.
In the opposite small-period regime, we construct an analogous family around a stationary front of the homogenized equation.

Both constructions start from an approximately invariant family of front-like profiles.
Separating interface motion from stable transverse modes is a classical approach in metastability,
as in Carr and Pego \cite{CarrPego1989}.
For the passage from approximate to exact invariance,
Bates, Lu, and Zeng~\cite{BLZ2008} developed an abstract persistence theory for approximately invariant, 
approximately normally hyperbolic manifolds, 
with graph transforms as a central tool. 
In the present work, 
the frozen profiles at large periods and the homogenization-corrected profiles at small periods provide the approximate families.
Uniform phase--normal estimates in the respective singular limits then yield a contracting time-map graph transform and hence a true invariant family.

In either regime, 
write $\mathcal M_L=\{Z_p:p\in\R\}$, 
where $p$ is the physical front position.
The exact dynamics on this family are $u(t,x)=Z_{p(t)}(x)$ with
\[
 \dot p=v_L(p),\qquad v_L(p+L)=v_L(p).
\]

The families are strictly ordered by position and attract nearby solutions exponentially
in the normal coordinates of \eqref{lp:tube} and \eqref{sp:chart}.
They are locally complete:
every entire orbit that remains in the corresponding phase--normal neighborhood belongs to the family.

Each zero of $v_L$ gives a stationary front. If the zero set is nonempty, each component of its complement gives a transition front connecting the stationary fronts at its boundary phases. When $v_L$ has no zero, the front propagates through successive spatial periods as a pulsating front, with signed mean speed

\begin{equation}\label{intro:passage}
	c_L=\frac{L}{\displaystyle\int_0^L v_L(p)^{-1}\,dp}.
\end{equation}
For periods for which the corresponding invariant family has been constructed,
the phase classification together with \cite[Theorem~1.2]{DingXu2026}
gives the exact propagation criterion:
a propagating pulsating front exists if and only if $v_L$ has no zero,
and is then unique up to time translation.

The two constructions yield velocity estimates at different scales.
On the large-period family,
$p=X$ and
$v_L(X)=s_\epsilon(\epsilon X)$, 
where $\epsilon=L^{-1}$ and
\begin{equation}\label{intro:reduction}
 s_\epsilon=c+\epsilon d+O(\epsilon^2)
 \quad\text{in }C^1(\T), \qquad \T = \R / \Z.
\end{equation}
The coefficient $d$ is the explicit adjoint projection defined in
\eqref{lp:d}. 
Where $c$ touches zero without changing sign, the correction $d$ can decide whether stationary or propagating fronts occur.
Example~\ref{lp:touching-example} shows that identical frozen profiles and speeds do not determine the front dynamics:
one of the two invariant families contains stationary and connecting fronts,
while the other propagates.
The correction also determines the speed of slow propagation near degenerate zeros of $c$.
In particular, suppose $c\geq0$ has a unique zero $Y_0$,
with $c(Y_0+r)=Ar^2+o(r^2)$, $A>0$, and $d(Y_0)=d_0>0$.
Then the corrected velocity is positive for all sufficiently large $L$,
and the mean speed satisfies $c_L\sim\sqrt{A d_0}/(\pi\sqrt L)$.

In the small-period regime, $p=\alpha$ and the $L$-periodic velocity satisfies
$v_L(\alpha)=s_L(\alpha)=\kappa L+O(L^2)$ uniformly in $\alpha$,
with $\kappa$ given in Proposition~\ref{sp:first-order-speed}.
Thus $\kappa\ne0$ yields propagation for all sufficiently small $L$,
although the homogenized front is stationary.

For a propagating front, a squeezing argument also yields global exponential orbital stability for front-like initial data.
Solutions converge exponentially to a time translate of the front.
Related global stability results for bistable pulsating fronts were obtained by Ding, Hamel, and Zhao \cite{DHZ} in spatially periodic reaction--diffusion media,  
and by Chen, Liang, and Tsai \cite{CLT2021} for pulsating fronts in a dissipative medium with periodically located Delta sources.
Their comparison and squeezing methods can be adapted to the present scalar periodic reaction--diffusion--advection equation.
Moreover, above results imply uniform boundedness of the width of propagating pulsating fronts in large periods. This conclusion is proved in \cite{DHL2025} by a zero-number discussion.

Attraction to an invariant family need not extend to all front-like initial data:
Section~\ref{sec:stability-uniqueness} constructs a stationary front with three spatial transitions outside the large-period family,
whose solution stays a positive distance from that family in the uniform norm.

This phase reduction also raises an inverse question:
can any smooth circle flows occur as front dynamics?
We give a positive answer in this paper.
More precisely,
we realize every smooth circle flow, 
up to time-preserving topological conjugacy,
as the quotient dynamics of a normally attracting invariant front family for a suitably chosen locally bistable periodic equation with constant diffusion and no advection.
Proposition~\ref{real:period-sensitivity} shows that normal attraction can persist while the internal pinning structure changes:
a Cantor set of stationary phases disappears under every sufficiently small nonzero change of the spatial period,
with the reaction fixed, and the quotient flow becomes a rotation.

\subsection{Common and regime-specific hypotheses}\label{intro:assumptions}
Throughout this paper, 
we assume that the following \textbf{common conditions} \eqref{coefficient1}-\eqref{intro:extended-stable-states} hold.

Write $\T=\R/\Z$ and identify functions on $\T$ with periodic functions on $\R$.
We assume that $a,b\in C^3(\T)$ and $f\in C^4(\T\times[0,1])$ are periodic in the spatial argument,
and
\begin{equation}\label{coefficient1}
0<a_*\leq a(y)\leq a^*,\qquad f(y,0)=f(y,1)=0.
\end{equation}
Assume there are $\gamma_0>0$ and $0<\delta_0<1/2$ such that
\begin{equation}\label{intro:stable-states}
 f_u(y,u)\leq-\gamma_0
 \quad\text{for }y\in\T,
 \quad u\in[0,\delta_0]\cup[1-\delta_0,1].
\end{equation}
Thus $0$ and $1$ are uniformly stable equilibria for the reaction term.
For the semiflow estimates we choose a $C^4$ extension of $f$ to $\T\times\R$, 
with at most linear growth and
\begin{equation}\label{intro:extended-stable-states}
 f_u(y,u)\leq-\gamma_0
 \quad\text{whenever }u\leq\delta_0\text{ or }u\geq1-\delta_0.
\end{equation}

\begin{assumption}[Large-period regime]\label{intro:large-assumption}
For each $y\in\T$ there is $\vartheta(y)\in(0,1)$ such that
\begin{equation}\label{intro:pointwise-bistable}
 \begin{gathered}
 f(y,\vartheta(y))=0,\\
 f(y,u)<0\quad(0<u<\vartheta(y)),\qquad
 f(y,u)>0\quad(\vartheta(y)<u<1).
 \end{gathered}
\end{equation}
\end{assumption}

In this paper, 
``locally bistable'' refers to the combination of the common linear stability condition \eqref{intro:stable-states} and the pointwise bistable sign structure \eqref{intro:pointwise-bistable}. 
It is a local condition on the frozen reactions and does not assert a global classification of all periodic stationary states.

For the small-period regime, define the effective coefficients
\begin{equation}\label{sp:effective}
 a_h=\left(\int_0^1a^{-1}\right)^{-1},\qquad
 b_h=a_h\int_0^1\frac{b}{a},\qquad
 \bar f(u)=\int_0^1f(y,u)\,dy.
\end{equation}
\begin{assumption}[Small-period regime]\label{intro:small-assumption}
The homogenized equation has a strictly decreasing stationary front $U$,
normalized by $U(0)=1/2$:
\begin{equation}\label{sp:standing}
 a_hU''+b_hU'+\bar f(U)=0,\qquad
 U(-\infty)=1,\quad U(+\infty)=0.
\end{equation}
\end{assumption}

\begin{remark}The two additional assumptions are used only in their respective period regimes;
when both hold, 
both constructions apply to the same equation.

Under Assumption~\ref{intro:small-assumption}, 
we restrict our attention to the case in which the homogenized equation admits a stationary front $U$. 
Its strict monotonicity, 
together with \eqref{intro:stable-states},
implies that $U$ is spectrally stable modulo translations.
The complementary case, 
in which it admits a propagating front, 
was studied in \cite{DingXu2026}.
\end{remark}

\subsection{Main results}\label{intro:main-results}
For the large-period statement, 
let $(\phi(Y,\cdot),c(Y))$ be the normalized frozen profile--speed pair at $Y\in\T$:
\begin{equation}\label{lp:frozen}
	\begin{cases}
		a(Y)\phi_{zz}+(b(Y)+c(Y))\phi_z+f(Y,\phi)=0,\\
		\phi(Y,-\infty)=1,\quad\phi(Y,+\infty)=0,\quad\phi(Y,0)=1/2.
	\end{cases}
\end{equation}
Assumption~\ref{intro:large-assumption} and the scalar homogeneous front theorem give this unique pair;
see \cite{FifeMcLeod1977}. 

Fix a smooth function $\rho$ equal to $|z|$ for $|z|\geq1$ and define
\[
H=L^2_\eta(\R),\qquad W=H^1_\eta(\R),\qquad D=H^2_\eta(\R),
\qquad \|v\|_{H^j_\eta}=\|e^{\eta\rho}v\|_{H^j}.
\]
The parameter $\eta>0$ is chosen sufficiently small. 
These spaces apply to \emph{perturbations} of a front profile.
The underlying affine phase space is
\[
\mathcal X_\eta=\phi(0,\cdot)+W.
\]
Every finite translate of every frozen profile belongs to this space.

\begin{theorem}[Large-period invariant family]\label{lp:main}
Assume that the common conditions and
Assumption~\ref{intro:large-assumption} hold.
Then 
there exist constants $L_0,C>0$ and $\eta>0$ such that,
for every $L\geq L_0$, equation~\eqref{intro:pde}
admits a one-dimensional totally invariant manifold
$\Mcal_\epsilon$, where $\epsilon=L^{-1}$.

More precisely, 
there exist functions $h_\epsilon(Y,z)$ and $s_\epsilon(Y)$, 
both $1$-periodic in $Y$, 
such that,
with
\[
    U_\epsilon(Y,z)=\phi(Y,z)+h_\epsilon(Y,z),
\]
the manifold has the parametrization
\begin{equation}\label{lp:manifold}
    Z_{\epsilon,X}(x)=U_\epsilon(\epsilon X,x-X),
    \qquad
    \Mcal_\epsilon=\{Z_{\epsilon,X}:X\in\R\}.
\end{equation}
Moreover, the functions $h_\epsilon,s_\epsilon$ and the manifold
$\Mcal_\epsilon$ satisfy the following properties.	
\begin{enumerate}
	\item The functions $h_\epsilon$ and $s_\epsilon$ satisfy
	\begin{equation}\label{lp:main-est}
		\|h_\epsilon\|_{C^1(\T;W)}
		+\|h_\epsilon\|_{C(\T;D)}
		+\|s_\epsilon-c\|_{C^1(\T)}\leq C\epsilon.
	\end{equation}
	The parametrization $X\mapsto Z_{\epsilon,X}$ is a $C^1$ embedding
	in $\mathcal X_\eta$ and satisfies
	\begin{equation}\label{lp:covariance}
		Z_{\epsilon,X+L}(x+L)=Z_{\epsilon,X}(x).
	\end{equation}
	The profiles are strictly increasing in the position phase:
	\begin{equation}\label{lp:phase-order}
		\begin{aligned}
			\partial_X Z_{\epsilon,X}(x)&>0\qquad(X,x\in\R),\\
			X_1<X_2&\quad\Longrightarrow\quad
			Z_{\epsilon,X_1}(x)<Z_{\epsilon,X_2}(x)\qquad(x\in\R).
		\end{aligned}
	\end{equation}
	\item Every solution on $\Mcal_\epsilon$ has the form
	\begin{equation}\label{lp:reduced}
		X'=s_\epsilon(\epsilon X),\qquad
		u(t,x)=Z_{\epsilon,X(t)}(x).
	\end{equation}
	Zeros of $s_\epsilon$ give stationary fronts, 
	and complementary intervals give connecting orbits. 
	If $s_\epsilon$ has no zero, 
	$\Mcal_\epsilon$ consists of the time slices of a single propagating pulsating front, 
	with signed speed given by \eqref{intro:passage} for $v_L(X)=s_\epsilon(X/L)$.
	\item The profiles satisfy
	\begin{equation}\label{lp:range}
		0<U_\epsilon(Y,z)<1,\qquad
		U_\epsilon(Y,-\infty)=1,\quad U_\epsilon(Y,+\infty)=0
	\end{equation}
	uniformly in $Y$, with uniform exponential convergence as
	$z\to\pm\infty$.
	\item $\Mcal_\epsilon$ attracts nearby solutions exponentially.
	Write $X(t)$ and $w(t)$ for the position and transverse perturbation in the chart \eqref{lp:tube}.
	There are $r,r_1,C_1,\nu>0$, independent of $L$,
	with $r_1<r$, such that solutions initially in $\mathcal U_{\epsilon,r_1}$ defined by \eqref{lp:tube} exist for all positive times, 
	remain in $\mathcal U_{\epsilon,r}$, 
	and satisfy
	\begin{equation}\label{lp:attraction}
		\|w(t)-h_\epsilon(\epsilon X(t))\|_W
		\leq C_1e^{-\nu t}
		\|w(0)-h_\epsilon(\epsilon X(0))\|_W.
	\end{equation}
	Every complete orbit which remains in $\mathcal U_{\epsilon,r}$ for all $t\in\R$ belongs to $\Mcal_\epsilon$. 
	\end{enumerate}
\end{theorem}

For the small-period statement, 
write $\mathsf X=L^2(\R)$, 
$V=H^1(\R)$, 
and $V'=H^{-1}(\R)$.
Let $\chi$ be the zero-mean periodic solution of
\begin{equation}\label{sp:chi}
	(a(1+\chi'))'= 0,
\end{equation}
and define
\begin{equation}\label{sp:approx}
	U_\alpha(x)=U(x-\alpha),\qquad
	V_{L,\alpha}=U_\alpha+L\chi(x/L)U_\alpha'.
\end{equation}
The affine phase space at phase $\alpha$ is
\[
\mathcal X_{0,\alpha}=U_\alpha+\mathsf X,
 \qquad
 \operatorname{dist}_{2}(u,U_\alpha)=\|u-U_\alpha\|_2.
\]

\begin{theorem}[Small-period invariant family]\label{sp:main}
	Under the common conditions and Assumption~\ref{intro:small-assumption},
	there exist $L_0,C>0$ such that, for
	$0<L<L_0$, equation \eqref{intro:pde} admits a one-dimensional totally invariant manifold
	\[
	\mathcal M_L=\{Z_{L,\alpha}=V_{L,\alpha}+g_L(\alpha):\alpha\in\R\}.
	\]
	It has the following properties.
	\begin{enumerate}
\item The map $g_L:\R\to\mathsf X$ is Lipschitz and
		\begin{equation}\label{sp:graphbound}
			\sup_\alpha\|g_L(\alpha)\|_2+
\operatorname{Lip}_{\mathsf X}(g_L)\leq CL.
		\end{equation}
		The parametrization is injective and satisfies
		\begin{equation}\label{sp:covariance}
			Z_{L,\alpha+L}(x+L)=Z_{L,\alpha}(x).
		\end{equation}
		The profiles are strictly increasing in the position phase:
		\begin{equation}\label{sp:phase-order}
			\alpha_1<\alpha_2\quad\Longrightarrow\quad
			Z_{L,\alpha_1}(x)<Z_{L,\alpha_2}(x)\qquad(x\in\R).
		\end{equation}
		\item There is an $L$-periodic Lipschitz function $s_L$ with $\|s_L\|_\infty+\operatorname{Lip}(s_L)\leq CL$ 
		such that every solution on this curve is
		\begin{equation}\label{sp:reduction}
			u(t,x)=Z_{L,\alpha(t)}(x),\qquad \dot\alpha=s_L(\alpha).
		\end{equation}
		The zeros and complementary intervals of $s_L$ give the stationary and connecting fronts described by Theorem~\ref{phase:classification}. 
		If $s_L$ has no zero, 
		the curve consists of the time slices of a single propagating pulsating front, 
		with signed speed given by \eqref{intro:passage} for $v_L=s_L$.
		\item The solutions on $\Mcal_L$ are defined for all $t\in\R$,
		are classical, take values strictly between $0$ and $1$, 
		and converge to $1$ and $0$ at the corresponding spatial infinities,
		uniformly in the phase:
		\begin{equation}\label{sp:tails}
			\lim_{R\to\infty}\sup_\alpha\left\{
			\sup_{x-\alpha\geq R}|Z_{L,\alpha}(x)|+
			\sup_{x-\alpha\leq-R}|1-Z_{L,\alpha}(x)|\right\}=0.
		\end{equation}
		\item The curve attracts sufficiently close data exponentially in the normal coordinate defined in \eqref{sp:chart}. 
		There is a radius $r>0$, independent of $L$, such that every complete orbit remaining in $\mathcal U_{L,r}$ for all $t\in\R$ lies on $\mathcal M_L$.
		If $s_L\equiv0$, 
each $Z_{L,\alpha}$ is Lyapunov stable in the affine topology of
 $\mathcal X_{0,\alpha}$, 
		and every sufficiently close datum converges exponentially in $L^2$ to a possibly different,
		nearby member of $\mathcal M_L$.
	\end{enumerate}
\end{theorem}

\begin{remark}
Theorem~\ref{sp:main} does not require the homogenized
reaction to satisfy the usual bistable sign condition.
Under the common hypotheses,
Assumption~\ref{intro:small-assumption} only requires
the homogenized equation to admit a strictly decreasing
stationary front connecting $1$ to $0$.
\end{remark}

The two preceding theorems reduce front dynamics to scalar periodic flows, 
which can be considered as circle flows equivalently. 
The next theorem gives the converse: every smooth circle flow can be realized by a locally bistable equation.

For the equation constructed in the next theorem, 
let $\phi_f(0,\cdot)$ denote its normalized frozen-front profile at phase $0$,
and choose the exponential weight $\eta>0$ as in Theorem~\ref{lp:main}.
We use the affine weighted phase space
\[
 \mathcal X_{\eta,f}^{\rm aff}
 =\phi_f(0,\cdot)+H^1_\eta(\R),
 \qquad
 \|u-v\|_{\mathcal X_{\eta,f}^{\rm aff}}
 :=\|u-v\|_{H^1_\eta}.
\]
Here a $C^1$ embedded family means that the map $p\mapsto Z_p$ is $C^1$ and injective, 
has nonzero tangent at every $p$, 
and is a homeomorphism from $\R$ onto its image with the subspace topology of $\mathcal X_{\eta,f}^{\rm aff}$.

\begin{theorem}[Realization of arbitrary smooth circle flows]\label{real:circle-flow}
Let $q\in C^\infty(\T)$, 
and let $\Phi_t^q$ be the flow generated by
\[
 \dot\theta=q(\theta),\qquad \theta\in\T.
\]
Then there exist a spatial period $L>0$ and a smooth reaction $f:\T\times[0,1]\to\R$,
satisfying the common conditions and Assumption~\ref{intro:large-assumption}, 
such that the locally bistable equation
\begin{equation}\label{real:pde}
 u_t=u_{xx}+f(x/L,u),\qquad x\in\R,
\end{equation}
admits a $C^1$ embedded one-dimensional invariant family $\mathcal M=\{Z_p:p\in\R\}$ of front-like profiles.
This family is strictly ordered by $p$, satisfies $Z_{p+L}(x+L)=Z_p(x)$, 
and attracts sufficiently close solutions exponentially in phase--normal coordinates as in Theorem~\ref{lp:main}. 
Every complete trajectory on $\mathcal M$ is a transition front connecting $1$ to $0$.

Furthermore,
let $\widehat{\mathcal M}=\mathcal M/\!\sim$, 
where $w\sim\widetilde w$ if $\widetilde w(x)=w(x-kL)$ for some $k\in\Z$,
and let $\overline S_t$ be the induced flow on $\widehat{\mathcal M}$. 
There is an orientation-preserving homeomorphism $\mathcal H:\T\to\widehat{\mathcal M}$ such that
\begin{equation}\label{real:conjugacy}
 \mathcal H\bigl(\Phi_t^q(\theta)\bigr)
 =\overline S_t\bigl(\mathcal H(\theta)\bigr)
 \qquad(\theta\in\T,\ t\in\R).
\end{equation}
The quotient $\widehat{\mathcal M}$ carries the quotient topology induced
by $\mathcal X_{\eta,f}^{\rm aff}$ and the spatial-translation relation
$\sim$.
\end{theorem}

The spatial period and the reaction are chosen as part of the realization,
which applies to every smooth circle flow, 
including those with degenerate or nonisolated equilibria.
Here \emph{topological conjugacy} means the time-preserving relation
in \eqref{real:conjugacy}: the same time parameter $t$ appears in both flows.
The conjugating map is a homeomorphism and may be only continuous in phase.

\begin{example}[Realized circle flows and front dynamics]\label{real:examples}
	\normalfont
	The reaction and the spatial period in Theorem~\ref{real:circle-flow} may be chosen separately for each prescribed smooth vector field $q$.
	The six diagrams below illustrate the resulting dynamics, 
	in order.
	In each diagram, 
	 the left panel shows the prescribed flow on $\T$ and the right panel shows its lift to the ordered front family.
	The correspondence is a time-preserving topological conjugacy; 
	it does not assert equality between $q$ and the exact normalized phase velocity.
	For flows with equilibria, 
	the construction in Section~\ref{sec:realization} fixes the singular phases: 
	a zero $\theta$ of $q$ corresponds to the stationary front $Z_{L\theta}$, 
	with phases interpreted through their lifts.
	
	\RealizationPanel{(a) Rotation: a propagating pulsating front.}{1}{%
		\(q(\theta)\equiv1\)\par\medskip
		\begin{tikzpicture}[x=1.3cm,font=\small]
			\draw[Flow,thick,-Stealth] (0,0)--(5.6,0);
			\foreach \x/\s in {.4/p,2.6/{p+L},4.8/{p+2L}}{
				\draw[Muted] (\x,-.08)--(\x,.08);
				\node[below] at (\x,-.1) {$Z_{\s}$};}
			\node[above] at (2.8,.18) {$u(t+T,x+L)=u(t,x)$};
		\end{tikzpicture}
		\[
		T=\int_0^L\frac{dp}{v_L(p)},\qquad c_L=\frac{L}{T}.
		\]
	}{%
		Here $v_L>0$ is the lifted position-phase velocity.
		All profiles lie on a single right-moving pulsating-front orbit.
		Time preservation gives $T=1$ for $q\equiv1$.
	}
	
	\RealizationPanel{(b) One stationary class: a circle homoclinic orbit.}{2}{%
		\(q(\theta)=1-\cos(2\pi\theta)\)\par\medskip
		\begin{tikzpicture}[x=1.3cm,font=\small]
			\foreach \x in {0,5}{\draw[Deg,fill=white,thick] (\x,0) circle(3pt);}
			\draw[Flow,thick,-Stealth] (.2,0)--(4.8,0);
			\node[below=7pt] at (0,0) {$Z_0$};
			\node[below=7pt] at (5,0) {$Z_L$};
			\node[above] at (2.5,.15) {$t:-\infty\longrightarrow+\infty$};
		\end{tikzpicture}
		\[
		Z_L(x)=Z_0(x-L),\qquad [Z_L]=[Z_0].
		\]
	}{%
		The unique nonstationary circle orbit is homoclinic to the stationary class.
		Its lift is a transition front connecting the distinct stationary profiles
		$Z_0$ and $Z_L$; its spatial-period translates give all the other lifts.
	}
	
	\RealizationPanel{(c) Two stationary classes: connections in the same direction.}{3}{%
		\(q(\theta)=\sin^2(2\pi\theta)\)\par\smallskip
		\ThreeGraph{1}
	}{%
		The two complementary arcs lift to
		$Z_0\longrightarrow Z_{L/2}\longrightarrow Z_L$.
		These arrows represent two distinct entire orbits, both moving to the right.
		Neither orbit crosses the intermediate stationary front in finite time.
		Each stationary phase attracts from one side and repels from the other.
	}
	
	\RealizationPanel{(d) One attracting and one repelling class: opposite directions.}{4}{%
		\(q(\theta)=\sin(2\pi\theta)\)\par\smallskip
		\ThreeGraph{2}
	}{%
		Both circle orbits run from the repelling class to the attracting class.
		Their lifts are $Z_0\longrightarrow Z_{L/2}\longleftarrow Z_L$,
		so the two interfaces move in opposite directions.
		Attraction and repulsion here refer to phase dynamics along the family;
		normal perturbations decay exponentially.
	}
	
	\RealizationPanel{(e) A closed arc of stationary phases: a continuum of fronts.}{5}{%
		\(K=[a,b],\quad 0<a<b<1,\quad q^{-1}(0)=K,\quad q\geq0\)\par\medskip
		\begin{tikzpicture}[x=1.3cm,font=\small]
			\draw[Muted!40] (0,0)--(5.8,0);
			\draw[Fixed,line width=4pt] (0,0)--(1.25,0);
			\draw[Fixed,line width=4pt] (4.3,0)--(5.55,0);
			\draw[Flow,thick,-Stealth] (1.45,0)--(4.1,0);
			\node[below=8pt] at (0,0) {$Z_{La}$};
			\node[below=8pt] at (1.25,0) {$Z_{Lb}$};
			\node[below=8pt] at (4.3,0) {$Z_{L(a+1)}$};
			\node[below=8pt] at (5.55,0) {$Z_{L(b+1)}$};
			\node[above] at (2.8,.15) {$u(t,\cdot)$};
		\end{tikzpicture}\par\medskip
	}{%
		Purple intervals represent the continuous, strictly ordered stationary
		family $\{Z_{L\theta}:\theta\in[a,b]\}$ and its translates.
		The complementary circle arc gives a connecting front from
		$Z_{Lb}$ to $Z_{L(a+1)}$.
	}
	
	\RealizationPanel{(f) A Cantor set of stationary phases.}{6}{%
		\(K\subset\T\text{ is Cantor},\quad q\in C^\infty(\T),\quad
		q^{-1}(0)=K,\quad q\geq0\)\par\medskip
		\begin{tikzpicture}[x=1.3cm,font=\small]
			\draw[Muted!35] (0,0)--(5.4,0);
			\draw[Fixed,line width=3.3pt] (0.0000,0)--(0.2000,0);
			\draw[Fixed,line width=3.3pt] (0.4000,0)--(0.6000,0);
			\draw[Fixed,line width=3.3pt] (1.2000,0)--(1.4000,0);
			\draw[Fixed,line width=3.3pt] (1.6000,0)--(1.8000,0);
			\draw[Fixed,line width=3.3pt] (3.6000,0)--(3.8000,0);
			\draw[Fixed,line width=3.3pt] (4.0000,0)--(4.2000,0);
			\draw[Fixed,line width=3.3pt] (4.8000,0)--(5.0000,0);
			\draw[Fixed,line width=3.3pt] (5.2000,0)--(5.4000,0);
			\draw[Flow,thick,-Stealth] (1.95,0)--(3.45,0);
			\draw[Flow,thick,-Stealth] (.67,0)--(1.13,0);
			\draw[Flow,thick,-Stealth] (4.27,0)--(4.73,0);
			\node[below=8pt] at (0,0) {$Z_0$};
			\node[below=8pt] at (5.4,0) {$Z_L$};
		\end{tikzpicture}\par\medskip
	}{%
		Choose the phase origin in $K$.
		There are uncountably many stationary classes, with no isolated points.
		Each lifted gap $(a_j,b_j)$ gives one connection
		$Z_{La_j}\longrightarrow Z_{Lb_j}$, up to time translation.
		Thus the connections are countable modulo time and spatial-period
		translations. The purple pieces show only a finite stage of the Cantor
		construction; the limiting set contains no interval.
	}
	
	In cases (b)--(f), each nonconstant lifted orbit has bounded position and
	connects stationary fronts. Consequently,
	\[
	\lim_{|t|\to\infty}\frac{p(t)-p(0)}{t}=0.
	\]
	Such a temporal connection is a transition front, but it is not a
	propagating pulsating front. All its spatial-period translates are also
	realized.
\end{example}

\begin{proposition}[Sensitivity of Cantor pinning to the spatial period]\label{real:period-sensitivity}
Let $q\in C^\infty(\T)$ be nonnegative with Cantor zero set $K$,
and fix the reaction $f$ constructed from $q$ in Section~\ref{sec:realization} at a sufficiently large period $L_0$.
For all $L$ sufficiently close to $L_0$, equation~\eqref{real:pde}
admits a normally attracting invariant front family $\mathcal M_L$.
At $L=L_0$, its stationary set modulo spatial-period translations is homeomorphic to $K$.
At every nearby $L\ne L_0$, its quotient flow is topologically conjugate to a rotation,
and $\mathcal M_L$ consists of the time slices of a single propagating pulsating front.
\end{proposition}

\subsection{Organization and conventions}
Section~\ref{sec:existence} constructs the large- and small-period invariant front families and establishes their transverse estimates.
Section~\ref{sec:dynamics} classifies the induced phase dynamics and proves strict ordering by phase, 
then develops the large-period speed correction, 
propagation-threshold laws, 
and the small-period first-order velocity formula.
Section~\ref{sec:stability-uniqueness} discusses local stability and proves global stability of propagating fronts,
then gives front-like initial data whose solutions do not approach the large-period invariant family.
Finally, Section~\ref{sec:realization} proves the realization theorem for arbitrary smooth circle flows
and Proposition~\ref{real:period-sensitivity} on sensitivity to the spatial period.
Throughout, 
$C$ denotes a positive constant whose value may change between occurrences; 
its uniformity with respect to the period is stated where used. 
All unqualified pairings are the usual distributional or $L^2$ duality pairings.

\section{Existence of invariant front families}\label{sec:existence}

We construct an invariant family in each period regime under the corresponding assumptions of Subsection~\ref{intro:assumptions}.
We choose the extension in \eqref{intro:extended-stable-states} with bounded derivatives of orders one through four.

\subsection{Large-period construction}\label{sec:large-complete}

\subsubsection{Moving coordinates and the phase--normal equations}

Throughout this subsection we assume the common conditions and Assumption~\ref{intro:large-assumption}, 
and put $\epsilon=L^{-1}$.
For a front position $X(t)$, 
introduce
\begin{equation}\label{lp:moving-coordinates}
 z=x-X(t),\qquad Y(t)=\epsilon X(t)\pmod 1.
\end{equation}
The medium is therefore evaluated at $x/L=Y+\epsilon z$.
In the moving coordinate the spatial operator is
\begin{equation}\label{lp:moving-operator}
 \Dcal_{\epsilon,Y}(v)
 =\partial_z\bigl(a(Y+\epsilon z)v_z\bigr)
  +b(Y+\epsilon z)v_z+f(Y+\epsilon z,v).
\end{equation}
Let $(\phi(Y,\cdot),c(Y))$ be the normalized frozen profile--speed pair in \eqref{lp:frozen}, 
and set $q_Y=\phi_z(Y,\cdot)$.
We write $\phi_Y=\partial_Y\phi$ and $\dot\ell_Y=\partial_Y\ell_Y$, 
reserving subscripts $z$ for spatial derivatives.
Define the normalized adjoint functional and the phase projections by
\begin{equation}\label{lp:projection}
 \begin{split}
  \psi_Y(z)&=
  \frac{e^{(b(Y)+c(Y))z/a(Y)}q_Y(z)}
  {\displaystyle\int_\R e^{(b(Y)+c(Y))r/a(Y)}q_Y(r)^2\,dr},\\
  \ell_Yv&=\int_\R\psi_Yv,\qquad
  P_Y=q_Y\ell_Y,\qquad Q_Y=I-P_Y,\qquad E_Y=\ker\ell_Y.
 \end{split}
\end{equation}
Thus $\ell_Yq_Y=1$.
Recall
\[
H=L^2_\eta(\R),\qquad W=H^1_\eta(\R),\qquad D=H^2_\eta(\R),
\qquad \|v\|_{H^j_\eta}=\|e^{\eta\rho}v\|_{H^j},
\]
and
\[
\mathcal X_\eta=\phi(0,\cdot)+W.
\]
We seek solutions in the form
\begin{equation}\label{lp:ansatz}
 u(t,x)=\phi(Y(t),z)+w(t,z),\qquad \ell_{Y(t)}w(t)=0.
\end{equation}
The associated chart and its tube are
\begin{equation}\label{lp:tube}
 \begin{split}
  \Theta_\epsilon(X,w)(x)&=\phi(\epsilon X,x-X)+w(x-X),\\
  \mathcal U_{\epsilon,r}
  &=\{\Theta_\epsilon(X,w):X\in\R,\
       \ell_{\epsilon X}w=0,\ \|w\|_W<r\}.
 \end{split}
\end{equation}

Define, 
on the common domain $D\subset H$,
\begin{equation}\label{lp:A}
 A_{\epsilon,Y}v
 =\partial_z\bigl(a(Y+\epsilon z)v_z\bigr)
  +(b(Y+\epsilon z)+c(Y))v_z
  +f_u(Y+\epsilon z,\phi(Y,z))v,
\end{equation}
and put
\begin{equation}\label{lp:residual}
 R_\epsilon(Y)
 =\Dcal_{\epsilon,Y}(\phi)+c(Y)q_Y-\epsilon c(Y)\phi_Y,
\end{equation}
\[
 N_{\epsilon,Y}(w)
 =f(Y+\epsilon z,\phi+w)-f(Y+\epsilon z,\phi)
  -f_u(Y+\epsilon z,\phi)w.
\]
Writing $X'=c(Y)+\sigma$,
we obtain the exact phase--normal equations
\begin{equation}\label{lp:exact}
 \begin{split}
  Y'&=\epsilon(c(Y)+\sigma),\\
  w_t&=A_{\epsilon,Y}w+R_\epsilon+N_{\epsilon,Y}(w)
       +\sigma(q_Y+w_z-\epsilon\phi_Y).
 \end{split}
\end{equation}
The derivative of the constraint is $\ell_Yw_t+\epsilon(c+\sigma)\dot\ell_Yw=0$.
Consequently, 
wherever its denominator is nonzero, 
the phase correction is
\begin{equation}\label{lp:sigma}
 \sigma_\epsilon(Y,w)
 =-\frac{\ell_Y(A_{\epsilon,Y}w+R_\epsilon+N_{\epsilon,Y}(w))
          +\epsilon c(Y)\dot\ell_Yw}
 {1-\epsilon\ell_Y\phi_Y+\ell_Yw_z+\epsilon\dot\ell_Yw}.
\end{equation}
These identities are initially derived for regular perturbations.
The estimates below extend the scalar formula and the evolution to $w\in W$.


\subsubsection{Frozen fronts, the coordinate chart, and residual estimates}

Throughout the spectral and resolvent arguments below, 
a superscript $^*$ denotes the formal adjoint with respect to the unweighted $L^2(\R)$ pairing.

For $Y\in\T$, 
realize the frozen linearization as a densely defined closed operator on $H=L^2_\eta(\R)$ with domain $D=H^2_\eta(\R)$:
\[
 \Lcal_Yv=a(Y)v_{zz}+(b(Y)+c(Y))v_z+f_u(Y,\phi(Y,z))v.
\]
Its formal adjoint is
\[
 \Lcal_Y^*v
 =a(Y)v_{zz}-(b(Y)+c(Y))v_z+f_u(Y,\phi(Y,z))v.
\]
The closed subspace and the domain of the restricted operator are
\[
 E_Y=\{v\in H:\ell_Yv=0\},
 \qquad
 \operatorname{Dom}(\Lcal_Y|_{E_Y})=D\cap E_Y.
\]
The restriction acts on $E_Y$.
Its invariance is verified below.
Spectra and resolvents are understood on the complexifications of these spaces.
\begin{lemma}[Frozen-front regularity and spectral gap]\label{lp:frozen-properties}
For every $Y$, one has, 
\[
q_Y,\psi_Y<0, \quad 
\Lcal_Yq_Y=\Lcal_Y^*\psi_Y=0, \quad 
\hbox{ and } \quad \ell_Yq_Y=1.
\]
Moreover, 
$0$ is an algebraically simple eigenvalue of $\Lcal_Y$, 
and there is $\kappa_f>0$, 
independent of $Y$, 
such that
\begin{equation}\label{lp:frozen-gap}
	\operatorname{spec}(\Lcal_Y|_{E_Y})
	\subset\{\lambda:\operatorname{Re}\lambda\leq-\kappa_f\}.
\end{equation}

The maps $Y\mapsto c(Y)$ and $Y\mapsto\phi(Y,\cdot)$ are $C^3$ and periodic.
Uniformly in $Y$, $q_Y,\psi_Y,\phi_Y$ and the derivatives used below decay exponentially as $|z|\to\infty$, 
as do the front tails $\phi(Y,z)$ as $z\to+\infty$ and $1-\phi(Y,z)$ as $z\to-\infty$.
\end{lemma}

\begin{proof}
The identities for $\psi_Y$ and $\ell_Y$ follow from the frozen equation and \eqref{lp:projection}.
Integration by parts gives $\ell_Y\Lcal_Yv=0$ for $v\in D$, 
so $\Lcal_Y(D\cap E_Y)\subset E_Y$ and the stated restriction is well defined.
The constant equilibria $0$ and $1$ are uniformly linearly stable by \eqref{intro:stable-states}.
For each fixed $Y$, 
the translation eigenfunction $q_Y$ has no zeros.
Conjugating the eigenvalue equation by $e^{(b(Y)+c(Y))z/(2a(Y))}$ removes the drift term.
For eigenvalues in a sufficiently small right half-plane,
the stable tail exponents make the conjugated eigenfunctions square integrable.
The Sturm oscillation theorem \cite[Theorem~9.40]{Teschl2009} gives zero as the simple largest eigenvalue of the resulting self-adjoint operator.
Thus $\Lcal_Y$ has no other eigenvalue with nonnegative real part.
Sandstede's criterion for the essential spectrum of a front \cite[Section~3.4.3]{Sandstede2002}, 
together with the linear stability of $0$ and $1$, 
places that spectrum strictly to the left of zero.
Since $\ell_Yq_Y=1$, zero is also algebraically simple.
For \eqref{lp:frozen} with its normalization, 
the linearized map is
\[
(v,\delta c)\longmapsto
(\Lcal_Yv+\delta c\,q_Y,v(0)),
\qquad H^2_\eta\times\R\longrightarrow L^2_\eta\times\R.
\]
Pairing with $\ell_Y$, 
then using invertibility on $E_Y$ and $q_Y(0)\ne0$, 
shows that it is an isomorphism.
The implicit function theorem gives $C^3$ dependence.
Operator continuity, 
the compactness of $\T$, 
spectral perturbation \cite[Chapter~IV, Section~3]{Kato1995}, 
and uniform sectorial bounds \cite[Sections~2.4 and~3.1]{Lunardi1995} make the gap uniform in unweighted $L^2$.
The uniform linear stability of $0$ and $1$ gives the stated exponential tails.
A sufficiently small common weight preserves the gap and simplicity, 
yielding \eqref{lp:frozen-gap}.
\end{proof}

\begin{lemma}[Large-period phase--normal decomposition]\label{lp:chart-unique}
There are $\epsilon_0,r_0>0$ such that, 
for $0<\epsilon\leq\epsilon_0$ and $0<r\leq r_0$, every $u\in\mathcal U_{\epsilon,r}$
has a unique representation $u=\Theta_\epsilon(X,w)$ with $\ell_{\epsilon X}w=0$ and $\|w\|_W<r$.
The inverse coordinate map $u\mapsto(X,w)$ is continuous.
\end{lemma}
\begin{proof}
For a fixed physical profile $u$ and a trial center $X$, 
set $Y=\epsilon X\pmod1$, $w_X(z)=u(z+X)-\phi(Y,z)$, 
and $\dot\ell_Y=\partial_Y\ell_Y$.
The scalar phase function is
\[
\mathcal P_\epsilon(X,u)=\ell_Yw_X
=\int_\R\psi_Y(z)[u(z+X)-\phi(Y,z)]\,dz.
\]
Changing variables to $y=z+X$ writes the translation term as $\int_\R\psi_Y(y-X)u(y)\,dy$.
Its $X$-derivative acts on the smooth, exponentially decaying kernel, so $\mathcal P_\epsilon(X,u)$ is differentiable even for $u\in\mathcal X_\eta$.
Using the weak derivative $u_x(\cdot+X)=q_Y+(w_X)_z$ and $\ell_Yq_Y=1$, differentiation gives
\begin{equation}\label{lp:phase-derivative}
\partial_X\mathcal P_\epsilon(X,u)
=1-\epsilon\ell_Y\phi_Y(Y)
 +\ell_Y(w_X)_z
 +\epsilon\dot\ell_Yw_X.
\end{equation}
Uniform exponential decay of $\psi_Y$ and its parameter derivative, and compactness of $\T$, imply
\begin{equation}\label{lp:phase-derivative-bound}
\bigl|\partial_X\mathcal P_\epsilon(X,u)-1\bigr|
\leq C\bigl(\epsilon+\|w_X\|_W\bigr).
\end{equation}

If $u=\Theta_\epsilon(X_0,w)$ with $\|w\|_W<r$ and $Y_0=\epsilon X_0\pmod1$, then for $|s|\leq1$,
\[
\begin{split}
w_{X_0+s}(z)
 &=w(z+s)+\phi(Y_0,z+s)-\phi(Y_0+\epsilon s,z),\\
\|w_{X_0+s}\|_W
 &\leq C(r+|s|).
\end{split}
\]
This uses bounded weighted translations for $|s|\leq1$ and the uniform $W$ bounds on $q_Y=\phi_z$ and $\phi_Y$.
Choose $\delta>0$, then $r>0$ and $\epsilon>0$ sufficiently small.
Equation \eqref{lp:phase-derivative-bound} gives $\partial_X\mathcal P_\epsilon(X_0+s,u)\geq1/2$ for $|s|\leq\delta$, 
with the same choices for every $X_0\in\R$.
The implicit function theorem gives a unique local phase that depends continuously on $u$.
The constants above are uniform in $X_0$, 
since their dependence on $X_0$ is only through $Y_0\in\T$.

To prove uniqueness in the whole tube, 
suppose $u=\Theta_\epsilon(X_1,w_1)=\Theta_\epsilon(X_2,w_2)$, 
where $\|w_i\|_W<r$, 
and put $s=X_2-X_1$ and $Y_i=\epsilon X_i\pmod1$.
Strict decrease, the normalization $\phi(Y,0)=1/2$, 
and compactness of $\T$ give $m_\delta>0$ such that $|\phi(Y,z)-1/2|\geq m_\delta$ for $|z|\geq\delta$, 
uniformly in $Y$.
If $|s|\geq\delta$, 
compare the two representations at $x=X_1$:
\[
m_\delta\leq |w_1(0)|+|w_2(-s)|\leq 2C_{\rm emb}r,
\]
which is impossible for small $r$.
If $|s|<\delta$, 
both centers are zeros of $\mathcal P_\epsilon(\cdot,u)$, 
which is strictly increasing between them by \eqref{lp:phase-derivative-bound} and the preceding bound on $\|w_{X_0+s}\|_W$.
Hence $s=0$ and then $w_1=w_2$.
Local phase continuity and strong continuity of translations on $W$ give continuity of the inverse coordinates.
\end{proof}

\begin{lemma}[Phase-derivative estimates near the stable equilibria]
\label{lp:stable-equilibrium-derivative-decay}
Assume $f\in C^4(\T\times[0,1])$ and $f(y,0)=f(y,1)=0$ for every $y\in\T$.
For $k=0,1,2,3$, one has
\[
 \partial_y^k f(y,0)=\partial_y^k f(y,1)=0.
\]
Furthermore,
there is a constant $C>0$ such that
\begin{equation}\label{lp:stable-equilibrium-derivative-bound}
 |\partial_y^k f(y,u)|\leq C\min\{u,1-u\},
 \qquad (y,u)\in\T\times[0,1],\quad 0\leq k\leq3.
\end{equation}
\end{lemma}

\begin{proof}
Differentiating the two equilibrium identities with respect to $y$ gives the stated vanishing.
The $C^4$ assumption implies that $\partial_u\partial_y^k f$ is bounded for $0\leq k\leq3$.
Integrating this derivative from either stable equilibrium gives
\[
 \partial_y^k f(y,u)
 =\int_0^u\partial_u\partial_y^k f(y,s)\,ds
 =-\int_u^1\partial_u\partial_y^k f(y,s)\,ds.
\]
Taking the smaller of the two resulting bounds proves \eqref{lp:stable-equilibrium-derivative-bound}.
\end{proof}

\begin{lemma}[Residual and nonlinear estimates]\label{lp:residual-lemma}
The residual $R_\epsilon$ belongs to $C^2(\T;H)$ and satisfies
\begin{equation}\label{lp:residual-est}
 \|R_\epsilon\|_{C^2(\T;H)}\leq C\epsilon.
\end{equation}
Moreover,
on a sufficiently small fixed $W$ ball, 
$N_{\epsilon,Y}(w)$ is jointly $C^2$ in $(Y,w)$, with uniformly bounded second derivatives, 
and
\begin{equation}\label{lp:nonlinear-est}
 \begin{split}
  \|N_{\epsilon,Y}(w)\|_H&\leq C\|w\|_W^2,\\
  \|N_{\epsilon,Y}(w)-N_{\epsilon,Y}(v)\|_H
  &\leq Cr\|w-v\|_W\qquad(\|w\|_W,\|v\|_W\leq r).
 \end{split}
\end{equation}
Furthermore, its first and second derivatives with respect to $Y$ vanish at $w=0$.
For sufficiently small $h\in C^1(\T;W)$,
one has
\begin{equation}\label{lp:nonlinear-C1-est}
 \|N_{\epsilon,\cdot}(h(\cdot))\|_{C^1(\T;H)}
 \leq C\|h\|_{C^1(\T;W)}^2.
\end{equation}
\end{lemma}

\begin{proof}
To prove \eqref{lp:residual-est}, first use the frozen equation to write
\[
\begin{aligned}
 R_\epsilon={}&[a(Y+\epsilon z)-a(Y)]\phi_{zz}
 +\epsilon a'(Y+\epsilon z)\phi_z\\
 &+[b(Y+\epsilon z)-b(Y)]\phi_z
 +\Delta_\epsilon-\epsilon c(Y)\phi_Y,\\
 \Delta_\epsilon(Y,z)={}&f(Y+\epsilon z,\phi(Y,z))-f(Y,\phi(Y,z)).
\end{aligned}
\]
The coefficient differences and their first two $Y$ derivatives are bounded by $C\epsilon|z|$, using $a,b\in C^3(\T)$.
For a derivative $g$ of $f$, abbreviate $[g]=g(Y+\epsilon z,\phi(Y,z))-g(Y,\phi(Y,z))$.
The chain rule gives
\[
 \partial_Y\Delta_\epsilon=[f_y]+[f_u]\phi_Y,
 \qquad
 \partial_Y^2\Delta_\epsilon
 =[f_{yy}]+2[f_{yu}]\phi_Y
   +[f_{uu}]\phi_Y^2+[f_u]\phi_{YY}.
\]
For $0\leq k\leq2$, Lemma~\ref{lp:stable-equilibrium-derivative-decay} and the mean-value formula yield
\[
\begin{aligned}
 \left[\partial_y^k f\right]
 &=\epsilon z\int_0^1
   \partial_y^{k+1}f(Y+\theta\epsilon z,\phi(Y,z))\,d\theta,\\
 \left|\left[\partial_y^k f\right]\right|
 &\leq C\epsilon|z|\min\{\phi(Y,z),1-\phi(Y,z)\}.
\end{aligned}
\]
The remaining differences $[f_u]$, $[f_{yu}]$, and $[f_{uu}]$ are bounded by $C\epsilon|z|$ and multiply exponentially decaying phase derivatives of $\phi$.
The uniform frozen-tail bounds in Lemma~\ref{lp:frozen-properties} therefore give,
for some $\nu>\eta$,
\[
 |\partial_Y^jR_\epsilon(Y,z)|
 \leq C\epsilon(1+|z|)e^{-\nu|z|},\qquad 0\leq j\leq2.
\]
Taking the weighted $L^2$ norm proves \eqref{lp:residual-est}.

Taylor's integral formula gives
\[
 N_{\epsilon,Y}(w)
 =w^2\int_0^1(1-s)f_{uu}(Y+\epsilon z,\phi+sw)\,ds.
\]
The embedding $W\hookrightarrow L^\infty$ and the bounded reaction derivatives on the range under consideration prove \eqref{lp:nonlinear-est}.
Differentiating the same formula proves joint $C^2$ regularity, 
the vanishing at $w=0$, 
and \eqref{lp:nonlinear-C1-est}.
\end{proof}

\subsubsection{Uniform transverse resolvents and the ambient extension}

The formal unweighted adjoint of \eqref{lp:A} is
\[
 A_{\epsilon,Y}^*v
 =\partial_z\bigl(a(Y+\epsilon z)v_z\bigr)
  -\partial_z\bigl((b(Y+\epsilon z)+c(Y))v\bigr)
  +f_u(Y+\epsilon z,\phi(Y,z))v.
\]
Define the ambient operator on $H$ with domain $D$ by
\begin{equation}\label{lp:ambient-operator}
 \widehat A_\epsilon(Y)=Q_YA_{\epsilon,Y}Q_Y-P_Y.
\end{equation}

\begin{lemma}[Uniform ambient resolvent and semigroup estimates]
\label{lp:stable}\label{common:ambient-extension}
For sufficiently small $\epsilon\geq0$, 
the operators $\widehat A_\epsilon(Y)$ are uniformly sectorial and exponentially stable on $H$, 
with domain $D$.
Furthermore, 
there are $M,\kappa>0$, 
independent of $\epsilon,Y$, 
such that, 
for $t>0$,
\begin{equation}\label{lp:semigroup}
 \begin{split}
 \|e^{t\widehat A_\epsilon(Y)}\|_{H\to H}
 +\|e^{t\widehat A_\epsilon(Y)}\|_{W\to W}
 &\leq Me^{-\kappa t},\\
 \|e^{t\widehat A_\epsilon(Y)}\|_{H\to W}
 +\|e^{t\widehat A_\epsilon(Y)}\|_{W\to D}
 &\leq M(1+t^{-1/2})e^{-\kappa t}.
 \end{split}
\end{equation}
Moreover, 
we have
\begin{equation}\label{lp:ambient-linear-estimates}
 \|\partial_Y^j\widehat A_\epsilon(Y)\|_{D\to H}\leq C_j,
 \qquad 0\leq j\leq2.
\end{equation}
\end{lemma}

\begin{proof}
Although $A_{\epsilon,Y}-\Lcal_Y$ need not be small from $D$ to $H$, 
exponential localization gives the unified estimate
\begin{equation}\label{lp:leakage}
 \|\ell_{\cdot}A_{\epsilon,\cdot}\|_{C^2(\T;H\to\mathbb R)}
 +\|A_{\epsilon,\cdot}q_{\cdot}\|_{C^2(\T;H)}
 \leq C\epsilon.
\end{equation}
Indeed, 
integration by parts and the frozen kernel identities give, 
for $v\in D$,
\[
 \ell_YA_{\epsilon,Y}v
 =\langle(A_{\epsilon,Y}^*-\Lcal_Y^*)\psi_Y,v\rangle,\qquad
 A_{\epsilon,Y}q_Y=(A_{\epsilon,Y}-\Lcal_Y)q_Y.
\]
Coefficient differences and their first two $Y$-derivatives are bounded by $C\epsilon|z|$, 
while the coefficient derivatives in $z$ arising in the adjoint are $O(\epsilon)$.
Using $a,b\in C^3$, $f\in C^4$, 
and the differentiated frozen-tail bounds in Lemma~\ref{lp:frozen-properties},
we obtain, 
for $0\leq j\leq2$ and some $\mu>\eta$,
\[
 \bigl|\partial_Y^j[(A_{\epsilon,Y}^*-\Lcal_Y^*)\psi_Y]\bigr|
 +\bigl|\partial_Y^j[(A_{\epsilon,Y}-\Lcal_Y)q_Y]\bigr|
 \leq C\epsilon(1+|z|)e^{-\mu|z|}.
\]
Taking the dual-$H$ and $H$ norms extends $\ell_YA_{\epsilon,Y}$ to $H$ and proves \eqref{lp:leakage}.

The identities $\ell_YQ_Y=0$ and $\ell_Yq_Y=1$ give, for $v\in D$,
\begin{equation}\label{lp:sum-action}
 \ell_Y\widehat A_\epsilon(Y)v=-\ell_Yv,\qquad
 \widehat A_\epsilon(Y)q_Y=-q_Y.
\end{equation}
Expanding $Q_Y=I-P_Y$ gives
\[
 \widehat A_\epsilon(Y)-A_{\epsilon,Y}
 =-P_YA_{\epsilon,Y}Q_Y-(A_{\epsilon,Y}q_Y)\ell_Y-P_Y.
\]
By \eqref{lp:leakage},
$\|\widehat A_\epsilon(Y)-A_{\epsilon,Y}+P_Y\|_{H\to H}\leq C\epsilon$,
so $\widehat A_\epsilon(Y)-A_{\epsilon,Y}$ is uniformly bounded and finite rank.
We next prove the resolvent estimate for large $|\lambda|$.
After conjugation by $Jv=e^{\eta\rho}v$, 
take the self-adjoint reference operator
	\[
	T_{\epsilon,Y}=-\partial_z(a(Y+\epsilon z)\partial_z)+1
	\quad\hbox{in }L^2(\R).
	\]
Its closed quadratic form has domain $H^1(\R)$, and its operator domain is $H^2(\R)$.
Let the graph norm of $T=T_{\epsilon,Y}$ be $\|v\|_{D(T)}:=\|v\|_2+\|Tv\|_2$.
Uniform ellipticity and the boundedness of $a$ and $a'$ give constants independent of $\epsilon,Y$ such that
	\[
	 C^{-1}\|v\|_{H^2}\leq\|v\|_{D(T)}
	 \leq C\|v\|_{H^2},\qquad v\in H^2(\R).
	\]
For a real vertex parameter $\beta$ and $0<\theta<\pi/2$, set
	\begin{equation}\label{lp:sector}
	 S_{\beta,\theta}
	 =\{-\beta-re^{i\varphi}:r\geq0,\ |\varphi|\leq\theta\},
	 \qquad \Omega_{\beta,\theta}=\mathbb C\setminus S_{\beta,\theta}.
	\end{equation}
Thus $S_{\beta,\theta}$ is the closed left-pointing sector with	vertex $-\beta$.
The spectral theorem defines the positive square root $T^{1/2}$, 
whose domain is the form domain $H^1(\R)$.
Its norm uniformly controls the $H^1$ norm by the quadratic form and ellipticity.
The same functional calculus \cite[Chapter~5]{Schmudgen2012} yields
	\[
	\|(\lambda+T)^{-1}\|_{2\to2}\leq C_\delta|\lambda|^{-1},
	\quad
	\|T^{1/2}(\lambda+T)^{-1}\|_{2\to2}
	\leq C_\delta|\lambda|^{-1/2},
	\quad
	\|T(\lambda+T)^{-1}\|_{2\to2}\leq C_\delta
	\]
for $\lambda\in\overline{\Omega_{0,\delta}}$ with $|\lambda|\geq1$.
To make the perturbation explicit, 
write $a_\epsilon=a(Y+\epsilon z)$, 
$d_\epsilon=b(Y+\epsilon z)+c(Y)$, 
$V_\epsilon=f_u(Y+\epsilon z,\phi(Y,z))$, 
and $p=\eta\rho'$.
Since $J\partial_zJ^{-1}=\partial_z-p$, 
direct expansion gives
	\[
	JA_{\epsilon,Y}J^{-1}
	 =\partial_z(a_\epsilon\partial_z)
	 +(d_\epsilon-2a_\epsilon p)\partial_z
	 +V_\epsilon-d_\epsilon p+a_\epsilon p^2-(a_\epsilon p)'.
	\]
Consequently, 
since $J:H\to L^2$ is an isometry,
	\[
	K_{\epsilon,Y}:=J(\widehat A_\epsilon(Y)-A_{\epsilon,Y})J^{-1},
	\qquad \|K_{\epsilon,Y}\|_{2\to2}\leq C.
	\]
Combining the two identities with the definition of $T$ yields
	\[
	J\widehat A_\epsilon(Y)J^{-1}+T
	=(d_\epsilon-2a_\epsilon p)\partial_z
	 +\bigl[1+V_\epsilon-d_\epsilon p
	       +a_\epsilon p^2-(a_\epsilon p)'\bigr]
	 +K_{\epsilon,Y}.
	\]
Both displayed coefficient functions are bounded uniformly in $\epsilon,Y$, 
since $a,b,f_u,\rho'$ and $\rho''$ are bounded on the relevant ranges.
The form identity gives $\|\partial_z w\|_2\leq C\|T^{1/2}w\|_2$.
Hence the spectral resolvent bounds above imply
	\[
	\|(J\widehat A_\epsilon(Y)J^{-1}+T)
	     (\lambda+T)^{-1}\|_{2\to2}
	\leq C_\delta\bigl(|\lambda|^{-1/2}+|\lambda|^{-1}\bigr).
	\]
For sufficiently large $|\lambda|$, 
the preceding norm is less than $1/2$, 
and
	\[
	\lambda-J\widehat A_\epsilon(Y)J^{-1}
	=\bigl[I-(J\widehat A_\epsilon(Y)J^{-1}+T)
	(\lambda+T)^{-1}\bigr](\lambda+T).
	\]
A Neumann series and the preceding estimates then give
	\begin{equation}\label{lp:parameter}
		|\lambda|\|v\|_H+|\lambda|^{1/2}\|v\|_W+\|v\|_D
		\leq C\|(\lambda-\widehat A_\epsilon(Y))v\|_H.
	\end{equation}
This holds for $\lambda\in\overline{\Omega_{0,\delta}}$ and $|\lambda|\geq R_\delta$, with $R_\delta$ independent of $\epsilon,Y$.

For bounded $\lambda$, 
we first obtain a coercive estimate on the spatial tails.
With $p$ and $d_\epsilon$ as above, 
direct integration by parts gives
	\begin{equation}\label{lp:energy}
		\begin{split}
			\operatorname{Re}(JA_{\epsilon,Y}J^{-1}v,v)_{L^2}
			={}&-\int_\R a_\epsilon|v_z|^2\,dz\\
			&+\int_\R\left(V_\epsilon-pd_\epsilon+a_\epsilon p^2
			-\frac{\epsilon}{2}b'(Y+\epsilon z)\right)|v|^2\,dz.
		\end{split}
	\end{equation}
By uniform frozen tails and \eqref{intro:stable-states}, 
choose $R$ independent of $\epsilon,Y$ so that $V_\epsilon\leq-\gamma_0$ on $|z|>R$.
First choose $\eta$ small and then $\epsilon$ small.
The coefficient	in parentheses is then at most $-\gamma_0/2$ on these tails.

Choose $\kappa_0<\min(\gamma_0/4,\kappa_f,1)/4$.
Fix $\chi_R\in C_c^\infty(\R)$ with $0\leq\chi_R\leq1$, 
$\chi_R\equiv1$ on $[-R,R]$, 
and $\operatorname{supp}\chi_R\subset[-R-1,R+1]$.
Subtracting a sufficiently large nonnegative smooth compactly supported potential $M_{\mathrm{coer}}\chi_R$ from $A_{\epsilon,Y}$
makes $\lambda-A_{\epsilon,Y}+M_{\mathrm{coer}}\chi_R$ coercive on $H^1_\eta$ whenever $\operatorname{Re}\lambda\geq-2\kappa_0$.
Lax--Milgram and elliptic regularity make it an isomorphism $D\to H$.
Compactness of multiplication by $\chi_R:D\to H$ and the finite-rank extension above show that $\lambda-\widehat A_\epsilon(Y)$ is Fredholm of index zero in this half-plane.

Suppose the uniform lower bound
$\|u\|_H\leq C\|(\lambda-\widehat A_\epsilon(Y))u\|_H$
fails for bounded $\lambda$ in this half-plane and all sufficiently small $\epsilon$.
There would be $\epsilon_n\to0$, $Y_n\to Y_*$, 
$\lambda_n\to\lambda_*$ and $u_n\in D$ with $\|u_n\|_H=1$ such that
\[
 g_n=(\lambda_n-\widehat A_{\epsilon_n}(Y_n))u_n\longrightarrow0
 \quad\hbox{in }H.
\]
By \eqref{lp:sum-action},
$(\lambda_n+1)\ell_{Y_n}u_n=\ell_{Y_n}g_n$.
Since $|\lambda_n+1|\geq1-2\kappa_0>0$, 
this gives $\ell_{Y_n}u_n\to0$ and $P_{Y_n}u_n\to0$ in $H$.
The bounded perturbation estimate above therefore yields
\begin{equation}\label{lp:contradiction-equation}
 (\lambda_n-A_{\epsilon_n,Y_n})u_n
 =g_n+(\widehat A_{\epsilon_n}(Y_n)-A_{\epsilon_n,Y_n})u_n
 \longrightarrow0
 \quad\hbox{in }H.
\end{equation}
Elliptic estimates give a uniform $D$ bound.
Applying \eqref{lp:energy} to $Ju_n$ gives
	\[
	a_*\|(Ju_n)_z\|_2^2+
	(\gamma_0/2-2\kappa_0)\|Ju_n\|_{L^2(|z|>R)}^2
	\leq C\|Ju_n\|_{L^2(|z|\leq R)}^2+o(1).
	\]
Consequently a positive amount of mass remains in a fixed compact interval.
A subsequence converges weakly in $D$ and strongly locally in $H^1$ to a nonzero $u_*$.
Passing to the limit in \eqref{lp:contradiction-equation} and in $\ell_{Y_n}u_n\to0$ yields
	\[
	(\lambda_*-\Lcal_{Y_*})u_*=0,\qquad \ell_{Y_*}u_*=0.
	\]
This contradicts \eqref{lp:frozen-gap}.
Thus $\lambda-\widehat A_\epsilon(Y)$ has a uniform lower bound.
Fredholm index zero gives surjectivity and a uniform $H\to H$ resolvent bound on bounded subsets of this half-plane.
Fix $\delta=\pi/4$.
Choose $0<\kappa_1<\kappa_0$ and $\theta_{\mathrm{lp}}\in(\delta,\pi/2)$ close enough to $\pi/2$.
Elementary sector geometry places the bounded part of $\overline{\Omega_{\kappa_1,\theta_{\mathrm{lp}}}}$ in the preceding half-plane and its sufficiently large part in $\overline{\Omega_{0,\delta}}$.
The resolvent bound for bounded $\lambda$ and \eqref{lp:parameter} therefore give a uniform stable resolvent on this sector.
The Dunford integral gives, 
for a possibly smaller $\kappa>0$,
\begin{equation}\label{lp:semigroup-HH,HD}
 \|e^{t\widehat A_\epsilon(Y)}\|_{H\to H}\leq Ce^{-\kappa t},\qquad
 \|e^{t\widehat A_\epsilon(Y)}\|_{H\to D}\leq C(1+t^{-1})e^{-\kappa t}.
\end{equation}
Uniform elliptic regularity and the bounded finite-rank difference $\widehat A_\epsilon(Y)-A_{\epsilon,Y}$ give
$\|v\|_D\asymp\|v\|_H+\|\widehat A_\epsilon(Y)v\|_H$.
Commutation of the semigroup with its generator on $D$ therefore yields
\begin{equation}\label{lp:semigroup-DD}
 \|e^{t\widehat A_\epsilon(Y)}\|_{D\to D}\leq Ce^{-\kappa t}.
\end{equation}
Conjugation by $J$ gives $(H,D)_{1/2,2}=W$.
Interpolating \eqref{lp:semigroup-HH,HD} and \eqref{lp:semigroup-DD} in the domain and range proves \eqref{lp:semigroup}.
Differentiating \eqref{lp:ambient-operator}, 
using the uniform coefficient and frozen-profile bounds,
proves \eqref{lp:ambient-linear-estimates}.
\end{proof}

Define $p_\epsilon(Y,w)=c(Y)+\sigma_\epsilon(Y,w)$ and
\begin{equation}\label{lp:ambient-nonlinearity}
 \begin{split}
 K_\epsilon(Y,w)
 ={}&Q_Y\{R_\epsilon+N_{\epsilon,Y}(w)
           +\sigma_\epsilon(Y,w)(w_z-\epsilon\phi_Y)\}\\
    &-\epsilon p_\epsilon(Y,w)q_Y\dot\ell_Yw.
 \end{split}
\end{equation}
We solve the following system on the fixed product $\T\times W$:
\begin{equation}\label{lp:fixed-system}
 Y'=\epsilon p_\epsilon(Y,w),\qquad
 w_t=\widehat A_\epsilon(Y)w+K_\epsilon(Y,w).
\end{equation}
\begin{lemma}[Large-period ambient estimates]\label{lp:ambient-estimates}
For sufficiently small $\epsilon$ and $r$, 
formula~\eqref{lp:sigma} defines a jointly $C^2$ function on the ambient ball $\T\times B_W(0,r)$.
Moreover, 
whenever $\|w\|_W\leq r$, 
one has
\begin{equation}\label{lp:sigma-est}
 \begin{split}
  |\sigma_\epsilon(Y,w)|&\leq C(\epsilon+\epsilon r+r^2),\\
  \|D_w\sigma_\epsilon(Y,w)\|_{W\to\R}&\leq C(\epsilon+r),
 \end{split}
\end{equation}
and
\begin{equation}\label{lp:K-est}
 \begin{split}
  \|K_\epsilon(Y,w)\|_H&\leq C(\epsilon+r^2),\qquad
  \|D_wK_\epsilon(Y,w)\|_{W\to H}\leq C(\epsilon+r),\\
  \|\partial_YK_\epsilon(Y,w)\|_H
   +\|\partial_Y^2K_\epsilon(Y,w)\|_H&\leq C(\epsilon+r), \qquad
  \|K_\epsilon(\cdot,0)\|_{C^2(\T;H)}\leq C\epsilon.
 \end{split}
\end{equation}
In addition,
all second derivatives of $p_\epsilon$ and $K_\epsilon$ are uniformly bounded, 
and $|\partial_Yp_\epsilon|\leq C$.
Here, 
all derivatives are taken in the ambient product space, 
with \(w\) held fixed when differentiating with respect to \(Y\).
\end{lemma}

\begin{proof}
The leakage bound \eqref{lp:leakage} defines $\ell_YA_{\epsilon,Y}$ on $H$, 
and integration by parts gives $\ell_Yw_z=-\int_\R(\partial_z\psi_Y)w$.
The denominator in \eqref{lp:sigma} is therefore $1+O(\epsilon+r)$ and is at least $1/2$ on a uniform ambient ball.
The residual, leakage, and nonlinear estimates give the numerator bound $C(\epsilon+\epsilon r+r^2)$ and its $w$-derivative bound $C(\epsilon+r)$.
The quotient rule proves \eqref{lp:sigma-est} and joint $C^2$ regularity.
In particular, $\sigma_\epsilon(\cdot,0)=O(\epsilon)$ in $C^2$.
The second phase derivative of the denominator uses $\partial_Y^3\ell_Y$, which is available from Lemma~\ref{lp:frozen-properties}.

Finally, insert the scalar estimates into \eqref{lp:ambient-nonlinearity} and differentiate.
The factors $\sigma_\epsilon$ in the transport term and $\epsilon$ in the constraint term give \eqref{lp:K-est} and the remaining second-derivative bounds.
\end{proof}

For regular solutions of \eqref{lp:fixed-system}, pairing the normal equation with $\ell_Y$ and using \eqref{lp:sum-action}
gives $\ell_Yw_t=-\ell_Yw-\epsilon p_\epsilon(Y,w)\dot\ell_Yw$.
Since $Y'=\epsilon p_\epsilon(Y,w)$, the chain rule yields
\begin{equation}\label{lp:constraint-transport}
 \frac{d}{dt}\bigl(\ell_{Y(t)}w(t)\bigr)=-\ell_{Y(t)}w(t).
\end{equation}
Thus the constraint $\ell_Yw=0$ is invariant.
On this constraint, $Q_Yw=w$, and \eqref{lp:sigma}
makes the $P_Y$ component of the normal equation in \eqref{lp:exact} equal to $-\epsilon p_\epsilon(Y,w)q_Y\dot\ell_Yw$.
Its $Q_Y$ component agrees with \eqref{lp:fixed-system}, proving equivalence.

\subsubsection{Parabolic evolution and the differentiable time map}

All evolution estimates below are on the fixed ambient spaces $H,W,D$.
For $T>0$ put $\mathbb E_T=H^1(0,T;H)\cap L^2(0,T;D)$.

\begin{lemma}[Parabolic estimate on fixed ambient spaces]\label{lp:MR-lemma}
The trace embedding $\mathbb E_T\hookrightarrow C([0,T];W)$ is continuous.
Moreover, 
for the frozen operator $B=\widehat A_\epsilon(Y_0)$, 
one has
\begin{equation}\label{lp:maximal-L2}
 \|w\|_{\mathbb E_T}+\|w\|_{C([0,T];W)}
 \leq C_T\bigl(\|w(0)\|_W
       +\|w_t-Bw\|_{L^2(0,T;H)}\bigr),
\end{equation}
and the initial-value problem with these data is uniquely solvable in $\mathbb E_T$.
The constant is uniform in $\epsilon,Y_0$.

For a continuous lifted phase path with $\sup_{[0,T]}|Y(t)-Y_0|\leq C_T\epsilon$, 
the same assertions hold with $\widehat A_\epsilon(Y(t))$ in place of $B$, 
for sufficiently small $\epsilon$.
Its homogeneous solution satisfies
\begin{equation}\label{lp:evolution-comparison}
 \sup_{0\leq t\leq T}\|w(t)-e^{tB}w(0)\|_W
 \leq C_T\epsilon\|w(0)\|_W.
\end{equation}
\end{lemma}

\begin{proof}
By Lemma~\ref{lp:stable}, $-B$ is an invertible sectorial operator of angle less than $\pi/2$ on the Hilbert space $H$.
The de Simon maximal $L^2$-regularity theorem, 
together with the nonzero-initial-value and trace theorems
\cite[Theorems~3.4.8, 3.5.5 and~3.5.7]{PrussSimonett2016},
gives \eqref{lp:maximal-L2}, 
unique solvability, 
and the trace embedding, since $(H,D)_{1/2,2}=W$.
The constants are uniform by the uniform stable resolvent bounds and graph-norm equivalence in Lemma~\ref{lp:stable}.

For a varying phase write
\[
 w_t-Bw=[\widehat A_\epsilon(Y(t))-B]w+F.
\]
By \eqref{lp:ambient-linear-estimates}, the additional term has $L^2_tH$ norm at most $C_T\epsilon\|w\|_{L^2_tD}$.
The Picard iteration for the frozen solution map with forcing $F+[\widehat A_\epsilon(Y(t))-B]v$ is a contraction on $\{v\in\mathbb E_T:v(0)=w_0\}$:
by \eqref{lp:ambient-linear-estimates} and \eqref{lp:maximal-L2}, 
its Lipschitz constant is at most $C_T\epsilon<1$.
Its fixed point gives unique solvability, 
and absorption in \eqref{lp:maximal-L2} gives the asserted estimate.
Subtracting the frozen homogeneous solution gives zero initial value and forcing $[\widehat A_\epsilon(Y(t))-B]w$, which proves \eqref{lp:evolution-comparison}.
\end{proof}

\begin{lemma}[Large-period ambient time-map estimates]\label{lp:time-map}
There is a fixed $T>0$ such that, for sufficiently small $\epsilon,r$, the lifted time-$T$ map of \eqref{lp:fixed-system},
\[
 \Phi_\epsilon^T(Y,w)=(\rho(Y,w),\zeta(Y,w)),
 \qquad (Y,w)\in\R\times B_W(0,r),
\]
is jointly $C^2$.
Moreover,
it satisfies $\rho(Y+1,w)=\rho(Y,w)+1$ and $\zeta(Y+1,w)=\zeta(Y,w)$.
For a fixed $\theta<1/8$,
\begin{equation}\label{lp:map-est}
	\begin{gathered}
		|\rho-Y|+|\rho_Y-1|\leq C_T\epsilon,\qquad
		\|\rho_w\|_{W\to\R}\leq C_T\epsilon(\epsilon+r),\\
		\|\zeta_w\|_{W\to W}\leq\theta+C_T(\epsilon+r),\qquad
		\|\zeta_Y\|_W+\|\zeta_{YY}\|_W\leq C_T(\epsilon+r),\\
		\|\zeta(Y,0)\|_W+\|\zeta_Y(Y,0)\|_W
		+\|\zeta_{YY}(Y,0)\|_W\leq C_T\epsilon.
	\end{gathered}
\end{equation}
Furthermore, $\|\zeta_{Yw}\|+\|\zeta_{ww}\|\leq C_T$ and $\|D^2\rho\|\leq C_T\epsilon$, with the natural linear and bilinear operator norms.
For $0\leq t\leq T$, the corresponding bounds hold with a uniform bounded constant in place of the contracting coefficient in $\zeta_w$.
All trajectories used in these maps remain in a ball of radius $C_T(r+\epsilon)$.
\end{lemma}

\begin{proof}
Choose $T$ so that $\|e^{T\widehat A_\epsilon(Y_0)}\|_{W\to W}<1/16$ uniformly, using \eqref{lp:semigroup}.
Write $Y(t)=Y_0+y(t)$ and freeze the normal operator at $Y_0$:
\begin{equation}\label{lp:time-map-system}
 \begin{split}
  y'&=\epsilon p_\epsilon(Y_0+y,w),\qquad y(0)=0,\\
  w_t-\widehat A_\epsilon(Y_0)w
  &=[\widehat A_\epsilon(Y_0+y)-\widehat A_\epsilon(Y_0)]w
     +K_\epsilon(Y_0+y,w).
 \end{split}
\end{equation}
Seek $(y,w)\in H^1(0,T;\R)\times\mathbb E_T$ with $w(0)=w_0$.
On a ball where $\|y\|_{H^1}\leq C_T\epsilon$ and $\|w\|_{\mathbb E_T}\leq C_T(r+\epsilon)$, 
the linearized normal forcing has bounds
\[
 \begin{split}
  \|\widehat A_\epsilon(Y)-\widehat A_\epsilon(Y_0)\|_{D\to H}
   &\leq C_T\epsilon,\\
  \|(\partial_Y\widehat A_\epsilon)(Y)w+(K_\epsilon)_Y\|_{L^2_tH}
   &\leq C_T(r+\epsilon),\\
  \|(K_\epsilon)_w\|_{W\to H}&\leq C_T(r+\epsilon).
 \end{split}
\]
The scalar equation has its additional factor $\epsilon$.
Time integration and Lemma~\ref{lp:MR-lemma} therefore turn \eqref{lp:time-map-system} into a contraction on the subspace with zero initial variations.
Choosing the ball constants first and then $r,\epsilon$ small makes this map a self-map and gives
\begin{equation}\label{lp:orbit-bound}
 \|w\|_{\mathbb E_T}+\|w\|_{C_tW}
 \leq C_T(\|w_0\|_W+\epsilon),\qquad
 \|y\|_{H^1}\leq C_T\epsilon.
\end{equation}
The construction takes place in a fixed ambient neighborhood where the scalar denominator is bounded away from zero.
Thus the solution exists on the whole interval $[0,T]$ for the stated data.

The residual map for \eqref{lp:fixed-system}, 
together with the initial conditions, 
is $C^2$ from $H^1_t\R\times\mathbb E_T$ into $L^2_t\R\times L^2_tH\times\R\times W$.
Its derivative in the trajectory variables is invertible by the same contraction argument.
The implicit function theorem gives a $C^2$ solution map and, 
by the trace embedding, a $C^2$ time map into $\R\times W$.

We give the variational estimates to identify the constants in \eqref{lp:map-est}.
Abbreviate $\widehat A_\epsilon(Y)$, $p_\epsilon$, and $K_\epsilon$ by $A(Y)$, $p$, and $K$ within this proof.
A first variation $(\xi,v)$ satisfies
\begin{equation}\label{lp:first-variation}
 \begin{split}
  \xi'&=\epsilon(p_Y\xi+p_wv),\\
  v_t&=A(Y)v+(A_Y(Y)w+K_Y)\xi+K_wv.
 \end{split}
\end{equation}
For initial direction $(1,0)$, the normal forcing has $L^2_tH$ norm $O_T(r+\epsilon)$.
Lemma~\ref{lp:MR-lemma}, absorption, and scalar integration give
\[
 \|\xi-1\|_{C_t}\leq C_T\epsilon,\qquad
 \|v\|_{\mathbb E_T}+\|v\|_{C_tW}\leq C_T(r+\epsilon).
\]
For initial direction $(0,k)$, 
the same argument gives
\[
 \|\xi\|_{C_t}\leq C_T\epsilon(r+\epsilon)\|k\|_W,\qquad
 \|v-e^{tA(Y_0)}k\|_{C_tW}\leq C_T(r+\epsilon)\|k\|_W.
\]
Evaluation at $T$, 
together with the choice of $T$, 
proves the first-order bounds.

For two initial directions $(\xi_i,v_i)$, 
the second variation has zero initial value and satisfies the same linear system with additional scalar forcing $\epsilon D^2p[(\xi_1,v_1),(\xi_2,v_2)]$ and normal forcing
\begin{equation}\label{lp:second-variation-forcing}
 \begin{split}
 F_2={}&A_{YY}(Y)w\,\xi_1\xi_2
       +A_Y(Y)(\xi_1v_2+\xi_2v_1)\\
      &+K_{YY}\xi_1\xi_2
       +K_{Yw}(\xi_1v_2+\xi_2v_1)+K_{ww}[v_1,v_2].
 \end{split}
\end{equation}
All terms belong to $L^2_tH$ with uniform bounds, using $v_i\in\mathbb E_T\subset C_tW\cap L^2_tD$.
When both initial directions are $(1,0)$, 
the first normal variations are $O_T(r+\epsilon)$, 
so the normal forcing is also $O_T(r+\epsilon)$.
Lemma~\ref{lp:MR-lemma} proves the asserted second-derivative estimates.
At $w_0=0$, estimate~\eqref{lp:orbit-bound} and the bound $K(\cdot,0)=O(\epsilon)$ in $C^2$ make the first two phase variations $O_T(\epsilon)$.
This proves the last line of \eqref{lp:map-est}.
Periodicity follows from the equations, 
and the same estimates at intermediate times give the remaining assertions.
\end{proof}

\subsubsection{The constrained invariant graph}

Consider the complete metric space
\[
 \mathfrak G_\epsilon
 =\{h:\T\to W:\ell_Yh(Y)=0,\
       \|h\|_\infty\leq A\epsilon,\
       \operatorname{Lip}(h)\leq K\epsilon\},
\]
with the supremum distance.
The constants $A,K$ are chosen independently of $\epsilon$.
For $h\in\mathfrak G_\epsilon$ put $\rho_h(Y)=\rho(Y,h(Y))$ and $\zeta_h(Y)=\zeta(Y,h(Y))$.
Lemma~\ref{lp:time-map} gives, for the degree-one lifts,
\begin{equation}\label{lp:base-inverse}
 |\rho_h(Y_1)-\rho_h(Y_2)-(Y_1-Y_2)|
 \leq C\epsilon|Y_1-Y_2|.
\end{equation}
Thus $\rho_h$ is an increasing homeomorphism with inverse Lipschitz constant $1+O(\epsilon)$.
Define
\begin{equation}\label{lp:graph-transform}
 (\mathfrak T_\epsilon h)(\beta)=\zeta_h(\rho_h^{-1}(\beta)).
\end{equation}
Constraint preservation in \eqref{lp:constraint-transport} gives $\ell_\beta(\mathfrak T_\epsilon h)(\beta)=0$.

\begin{proposition}[Existence, regularity, and invariance of the large-period graph]
\label{lp:graph-proposition}
For suitable $A,K$ and sufficiently small $\epsilon$, $\mathfrak T_\epsilon$ has a unique fixed point $h_\epsilon\in\mathfrak G_\epsilon$.
It satisfies
\begin{equation}\label{lp:graph-C1}
 h_\epsilon\in C^1(\T;W),\qquad
 \|h_\epsilon\|_{C^1(\T;W)}\leq C\epsilon.
\end{equation}
Moreover, 
its graph is invariant under the full semiflow of \eqref{lp:fixed-system}.
\end{proposition}

\begin{proof}
The time-map estimates give
\[
 \begin{split}
  \|\mathfrak T_\epsilon h\|_\infty
  &\leq(\theta+C\epsilon)A\epsilon+C\epsilon,\\
  \operatorname{Lip}(\mathfrak T_\epsilon h)
  &\leq(1+C\epsilon)
       \{(\theta+C\epsilon)K\epsilon+C(1+A)\epsilon\}.
 \end{split}
\]
Choose $A$ large, 
then $K$ large, 
and finally $\epsilon$ small.
These bounds make $\mathfrak T_\epsilon$ a self-map of $\mathfrak G_\epsilon$.
For $d=\|h-\widetilde h\|_\infty$, 
the bound $\rho_w=O(\epsilon^2)$ on the graph radius gives
\[
 \|\rho_h-\rho_{\widetilde h}\|_\infty
 +\|\rho_h^{-1}-\rho_{\widetilde h}^{-1}\|_\infty
 \leq C\epsilon^2d.
\]
At a common final phase $\beta$, put $Y_h=\rho_h^{-1}(\beta)$ and
$Y_{\widetilde h}=\rho_{\widetilde h}^{-1}(\beta)$.
The preceding estimate gives $|Y_h-Y_{\widetilde h}|\leq C\epsilon^2d$.
At a common initial phase, 
the normal images differ by at most $(\theta+C\epsilon)d$,
and $\operatorname{Lip}(\zeta_{\widetilde h})\leq C\epsilon$.
Thus direct comparison at $\beta$ gives
\begin{equation}\label{lp:graph-contraction}
 \|\mathfrak T_\epsilon h-\mathfrak T_\epsilon\widetilde h\|_\infty
 \leq(\theta+C\epsilon+C\epsilon^3)d<\tfrac12d.
\end{equation}
There is a unique fixed graph in $\mathfrak G_\epsilon$.
The same comparison proves uniqueness of any two fixed graphs in a larger class with sufficiently small fixed radius $r_0$ and slope $K_0$.
Indeed, 
its contraction bound is at most
\[
 \theta+C(r_0+\epsilon)
 +C\epsilon(\epsilon+r_0)(r_0+K_0)<1.
\]
To prove $C^1$ regularity, start at $h_0=0$ and set $h_{n+1}=\mathfrak T_\epsilon h_n$.
The ambient time map is $C^2$, so every iterate is $C^2$.
For $G=\mathfrak T_\epsilon h$, differentiation of $G\circ\rho_h=\zeta_h$ gives
\begin{equation}\label{lp:graph-first-derivative}
 (G'\circ\rho_h)(\rho_Y+\rho_wh')=\zeta_Y+\zeta_wh'.
\end{equation}
All time-map derivatives are evaluated at $(Y,h(Y))$.
Thus $\|h_n'\|_\infty\leq K\epsilon$ and $\rho_{h_n}'=1+O(\epsilon)$.
Differentiating once more gives
\[
 \begin{split}
 (G''\circ\rho_h)(\rho_h')^2
 ={}&\zeta_{YY}+2\zeta_{Yw}h'+\zeta_{ww}[h',h']+\zeta_wh''\\
    &-(G'\circ\rho_h)
      \{\rho_{YY}+2\rho_{Yw}h'+\rho_{ww}[h',h']+\rho_wh''\}.
 \end{split}
\]
The coefficient of $h''$, 
after division by $(\rho_h')^2$, 
has norm at most $\theta+C\epsilon$.
All remaining terms are $O_{A,K}(\epsilon)$ by Lemma~\ref{lp:time-map}.
Consequently,
\[
 \|h_{n+1}''\|_\infty
 \leq(\theta+C\epsilon)\|h_n''\|_\infty+C_{A,K}\epsilon.
\]
Hence $\|h_n''\|_\infty\leq K_2\epsilon$ for a fixed $K_2$.
The Banach-valued interpolation inequality
\[
 \|f'\|_\infty^2
 \leq C\|f\|_\infty(\|f''\|_\infty+\|f\|_\infty)
\]
applied to $h_n-h_m$, 
together with \eqref{lp:graph-contraction}, 
proves convergence in $C^1(\T;W)$.
This establishes \eqref{lp:graph-C1}.

For $0\leq s\leq T$, the time-$s$ image of the fixed graph is still a graph, 
since its lifted base map satisfies \eqref{lp:base-inverse}.
Its radius and slope are $O(\epsilon)$ by the intermediate-time bounds.
It satisfies the constraint and lies in the larger fixed uniqueness class above.
Since the time-$s$ and time-$T$ maps commute, 
its image is also fixed by the time-$T$ graph transform.
Uniqueness in that class makes it the same graph.
Iteration gives invariance for every $t\geq0$:
\begin{equation}\label{lp:all-time}
 \Phi_\epsilon^t\{(Y,h_\epsilon(Y)):Y\in\T\}
 =\{(Y,h_\epsilon(Y)):Y\in\T\},\qquad t\geq0.
\end{equation}
The induced base vector field is $C^1$ on the compact circle, so graph trajectories are complete.
\end{proof}

\subsubsection{Profile regularity, range, and normal attraction}

\begin{lemma}[Range of bounded entire solutions]\label{common:range}
Under \eqref{intro:extended-stable-states}, 
every uniformly bounded classical entire solution of \eqref{intro:pde} satisfies $0\leq u\leq1$.
If it is neither constant $0$ nor constant $1$, 
then $0<u<1$ everywhere.
\end{lemma}

\begin{proof}
By uniform boundedness, choose $M>0$ such that $-M\leq u(t,x)\leq1+M$ for all $(t,x)\in\R^2$.
For arbitrary $t_0\in\R$ and $t\geq t_0$, the spatially constant functions $-M e^{-\gamma_0(t-t_0)}$ and $1+M e^{-\gamma_0(t-t_0)}$
are respectively a subsolution and a supersolution, by \eqref{intro:extended-stable-states}.
The comparison principle for bounded solutions on $\R$ gives
 \[
 -M e^{-\gamma_0(t-t_0)}\leq u(t,x)
 \leq 1+M e^{-\gamma_0(t-t_0)}.
 \]
For each fixed $t$, letting $t_0\to-\infty$ yields $0\leq u\leq1$.
The strong parabolic maximum principle and uniqueness then make the inequalities strict unless $u$ is one of the two constant states.
\end{proof}

\begin{proof}[Completion of the large-period construction in Theorem~\ref{lp:main}]
Define
\[
 s_\epsilon(Y)=c(Y)+\sigma_\epsilon(Y,h_\epsilon(Y)),\qquad
 U_\epsilon=\phi+h_\epsilon.
\]
The scalar estimates and \eqref{lp:graph-C1} give $\|s_\epsilon-c\|_{C^1(\T)}\leq C\epsilon$.
Constraint preservation identifies every graph trajectory with a physical solution.
Full-time invariance gives, 
initially in distributions in $z$,
\begin{equation}\label{lp:invariance-equation}
 \Dcal_{\epsilon,Y}(U_\epsilon)
 +s_\epsilon(U_\epsilon)_z-\epsilon s_\epsilon(U_\epsilon)_Y=0.
\end{equation}
Equivalently,
\[
 \begin{split}
 A_{\epsilon,Y}h_\epsilon
 ={}&-R_\epsilon-N_{\epsilon,Y}(h_\epsilon)\\
    &-(s_\epsilon-c)(q_Y+(h_\epsilon)_z-\epsilon\phi_Y)
      +\epsilon s_\epsilon(h_\epsilon)_Y.
 \end{split}
\]
Every term on the right is continuous in $Y$ as an $H$-valued function and is uniformly $O(\epsilon)$.
Uniform elliptic regularity first gives $\sup_Y\|h_\epsilon(Y)\|_D\leq C\epsilon$.
Writing the equation as $A_{\epsilon,Y}h_\epsilon(Y)=F(Y)$, 
we have
\[
 \begin{split}
 A_{\epsilon,Y_1}(h_\epsilon(Y_2)-h_\epsilon(Y_1))
 ={}&F(Y_2)-F(Y_1)\\
    &+(A_{\epsilon,Y_1}-A_{\epsilon,Y_2})h_\epsilon(Y_2).
 \end{split}
\]
Continuity of $F$, 
continuity of $A_{\epsilon,Y}$ in $\mathcal L(D,H)$, 
and the uniform $D$ bound make the right side tend to zero in $H$.
Since $h_\epsilon$ is already continuous in $H$, 
the elliptic estimate gives continuity in $D$.
This proves the full estimate \eqref{lp:main-est}.

Put $Y=\epsilon X$ and $Z_{\epsilon,X}(x)=U_\epsilon(Y,x-X)$.
Strong continuity of translations in the weighted Sobolev spaces and the proved regularity make $X\mapsto Z_{\epsilon,X}$ a $C^1$ curve in $\mathcal X_\eta$.
Its tangent is $-(U_\epsilon)_z+\epsilon(U_\epsilon)_Y$ in centered coordinates, 
with $\ell_Y$ pairing $-1+O(\epsilon)$, 
so it is nonzero.
The curve lies in the tube \eqref{lp:tube}, 
and the unique continuous inverse chart in Lemma~\ref{lp:chart-unique} proves that it is embedded.
The equations give \eqref{lp:reduced}, 
and periodicity gives \eqref{lp:covariance}.
Weighted Sobolev embedding and the frozen tails yield uniform exponential limits at the two spatial infinities.
Every graph trajectory is complete and uniformly bounded, 
and its tails exclude the constant states.
Lemma~\ref{common:range} therefore gives $0<U_\epsilon<1$, proving \eqref{lp:range}.

For normal attraction, compare initial states $(Y_0,w_0)$ and $(Y_0,h_\epsilon(Y_0))$, 
and put $d_0=\|w_0-h_\epsilon(Y_0)\|_W$.
Along the first solution set $e(t)=\|w(t)-h_\epsilon(Y(t))\|_W$.
Invariance and Lemma~\ref{lp:time-map} give
\[
 \begin{split}
 &\|\zeta(Y_0,w_0)-h_\epsilon(\rho(Y_0,w_0))\|_W\\
 &\quad\leq\|\zeta(Y_0,w_0)-\zeta(Y_0,h_\epsilon(Y_0))\|_W\\
 &\qquad\quad+K\epsilon\,
   |\rho(Y_0,w_0)-\rho(Y_0,h_\epsilon(Y_0))|\\
 &\quad\leq\{\theta+C_T(\epsilon+r)
                 +C_TK\epsilon^2(\epsilon+r)\}d_0.
 \end{split}
\]
Choose the fixed tube radius and then $\epsilon$ so that the coefficient is $q<1$.
Thus
\begin{equation}\label{lp:normalcontraction}
 \|w(T)-h_\epsilon(Y(T))\|_W
 \leq q\|w(0)-h_\epsilon(Y(0))\|_W.
\end{equation}
The corresponding estimate for $0\leq t\leq T$ has a fixed bounded coefficient.
Choose $r_1<r$ so that this intermediate-time bound and the graph radius $C\epsilon$ leave a positive margin to the boundary of the coordinate tube.
Writing $t=nT+s$, $0\leq s<T$, 
iteration of \eqref{lp:normalcontraction} gives $e(t)\leq Cq^ne(0)\leq Ce^{-\nu t}e(0)$, 
where $\nu=-T^{-1}\log q$.
The margin and intermediate-time bound give global forward continuation for these initial data.
If a complete orbit remains in this tube, 
then $e(0)\leq q^ne(-nT)\to0$;
applying this at every time shows that the orbit lies on the graph.
\end{proof}

\begin{corollary}[Uniform spatial transition width at large periods]
\label{lp:uniform-transition-width}
Under the assumptions of Theorem~\ref{lp:main}, 
there exist $L_0>0$, 
$C\geq1$, 
and $\mu>0$, 
independent of $L$ and the position phase, 
such that
\begin{equation}\label{lp:uniform-tail-bound}
 Z_{1/L,X}(X+z)+1-Z_{1/L,X}(X-z)
 \leq C e^{-\mu z}
 \qquad
 (L\geq L_0,\ X\in\R,\ z\geq0).
\end{equation}
Consequently, for every $\delta\in(0,1/2)$, the number
\[
 R_\delta=\frac{1}{\mu}\log\frac{2C}{\delta}
\]
is independent of $L$, and every complete orbit
\[
 u_L(t,x)=Z_{1/L,X_L(t)}(x)
 \quad\text{on }\mathcal M_{1/L}
\]
satisfies
\begin{equation}\label{lp:uniform-transition-set}
 \{x\in\R:\delta\leq u_L(t,x)\leq1-\delta\}
 \subset
 [X_L(t)-R_\delta,X_L(t)+R_\delta]
 \qquad(t\in\R).
\end{equation}
In particular,
\begin{equation}\label{lp:uniform-transition-diameter}
 \sup_{t\in\R}
 \operatorname{diam}
 \{x\in\R:\delta\leq u_L(t,x)\leq1-\delta\}
 \leq 2R_\delta
 \qquad(L\geq L_0).
\end{equation}
These estimates hold for stationary fronts, connecting transition fronts, 
and propagating pulsating fronts on the invariant family, 
without any sign restriction on the frozen speed $c(Y)$.
\end{corollary}

\begin{proof}
By \eqref{lp:manifold},
\[
 Z_{1/L,X}(X+z)=\phi(X/L,z)+h_{1/L}(X/L,z).
\]
The weighted estimate \eqref{lp:main-est} and the one-dimensional Sobolev embedding give, after adjusting the smooth weight in the definition of $H^1_\eta$,
\[
 \sup_{Y\in\T,\ z\in\R}e^{\eta|z|}
 |h_{1/L}(Y,z)|\leq \frac{C_1}{L}.
\]
The frozen profiles have uniform exponential tails in $Y$ by Lemma~\ref{lp:frozen-properties}.
Decreasing the decay exponent if necessary and enlarging the constant therefore gives \eqref{lp:uniform-tail-bound}.
For $z\geq R_\delta$, its right-hand side is at most $\delta/2$.
Applying the estimate with $X=X_L(t)$ gives \eqref{lp:uniform-transition-set} and \eqref{lp:uniform-transition-diameter}.
\end{proof}

\begin{remark}[Comparison with earlier large-period width estimates]
\label{lp:width-comparison-DHL}
The width estimate in Corollary~\ref{lp:uniform-transition-width} follows from the weighted $O(L^{-1})$ approximation by frozen profiles and their uniform exponential tails.
It therefore applies under the present assumptions, including vanishing or sign-changing frozen speeds and variable diffusion and drift.
The result in \cite{DHL2025} uses zero-number and intersection arguments under hypotheses that keep the frozen speeds strictly of one sign and uniformly separated from zero.
Our estimate concerns the constructed invariant family;
Proposition~\ref{comp:global-counterexample} shows that stationary fronts outside this family can have several spatial transitions.
\end{remark}

\subsection{Small-period construction}\label{sp:section}

\subsubsection{Phase coordinates and the phase--normal equations}

Assume the common conditions and Assumption~\ref{intro:small-assumption}.
Use $U_\alpha,V_{L,\alpha}$ from \eqref{sp:approx} and the spaces $\mathsf X=L^2(\R)$, $V=H^1(\R)$, $V'=H^{-1}(\R)$.
With $q_{0,\alpha}=-U_\alpha'$, set
\[
 \psi_\alpha(x)=C_0e^{b_h(x-\alpha)/a_h}q_{0,\alpha}(x),
 \qquad \int_\R\psi_\alpha q_{0,\alpha}=1.
\]
Let $\zeta$ be the zero-mean periodic solution of
\begin{equation}\label{sp:zeta}
 a\zeta'=b-b_h,
\end{equation}
which is solvable by \eqref{sp:effective}.
Define
\begin{equation}\label{sp:phasefunctional}
 \begin{split}
 q_{L,\alpha}&=\partial_\alpha V_{L,\alpha},\qquad
 \psi_{L,\alpha}=\psi_\alpha+L\chi(x/L)\psi_\alpha'
                         +L\zeta(x/L)\psi_\alpha,\\
 \ell_{L,\alpha}h&=\langle h,\psi_{L,\alpha}\rangle,\qquad
 d_{L,\alpha}=\ell_{L,\alpha}q_{L,\alpha}.
 \end{split}
\end{equation}
Where $d_{L,\alpha}\ne0$, put $P_\alpha=q_{L,\alpha}\ell_{L,\alpha}/d_{L,\alpha}$, $Q_\alpha=I-P_\alpha$, and $N_\alpha=\ker\ell_{L,\alpha}$.
We seek
\begin{equation}\label{sp:ansatz}
 u(t,x)=V_{L,\alpha(t)}(x)+w(t,x),\qquad \ell_{L,\alpha(t)}w(t)=0
\end{equation}
in the tube
\begin{equation}\label{sp:chart}
 \mathcal U_{L,r}=\{V_{L,\alpha}+w:\alpha\in\R,\
               \ell_{L,\alpha}w=0,\ \|w\|_2<r\}.
\end{equation}
Define $A_{L,\alpha}:V\to V'$ and the residual and nonlinear remainder by
\begin{equation}\label{sp:linearization}
 \begin{split}
 A_{L,\alpha}w&=(a(x/L)w')'+b(x/L)w'+f_u(x/L,V_{L,\alpha})w,\\
 R_{L,\alpha}&=F_L(V_{L,\alpha}),\\
 N_{L,\alpha}(w)&=f(x/L,V_{L,\alpha}+w)-f(x/L,V_{L,\alpha})
                   -f_u(x/L,V_{L,\alpha})w.
 \end{split}
\end{equation}
For regular $w$ with $d_{L,\alpha}-\ell_{L,\alpha}'w\ne0$,
substitution into \eqref{intro:pde} and differentiation of $\ell_{L,\alpha}w=0$ give
\begin{equation}\label{sp:coordinates}
 \begin{split}
 \dot\alpha=\nu_L(\alpha,w)
 &=\frac{\ell_{L,\alpha}(R_{L,\alpha}+A_{L,\alpha}w+N_{L,\alpha}(w))}
         {d_{L,\alpha}-\ell_{L,\alpha}'w},\\
 w_t&=A_{L,\alpha}w+R_{L,\alpha}+N_{L,\alpha}(w)-q_{L,\alpha}\nu_L(\alpha,w).
 \end{split}
\end{equation}
We require uniform residual, 
transverse resolvent, 
and time-map estimates for this system.

\subsubsection{The homogenized front, coordinate chart, and residual estimates}

\begin{lemma}[Homogenized-front spectral gap]\label{sp:limit-spectrum}
The functions $U',\psi_\alpha$ and their spatial derivatives decay exponentially.
Moreover, on $\mathsf X$ with domain $H^2(\R)$, 
the operator $A_{0,\alpha}=a_h\partial_x^2+b_h\partial_x+\bar f'(U_\alpha)$ has an algebraically simple eigenvalue zero, 
with kernels generated by $q_{0,\alpha}$ and $\psi_\alpha$ for $A_{0,\alpha}$ and $A_{0,\alpha}^*$.
Its remaining spectrum lies in $\{\operatorname{Re}\lambda\leq-\sigma\}$ for a uniform $\sigma>0$.
\end{lemma}
\begin{proof}
Apply the unweighted decay and spectral-gap argument of Lemma~\ref{lp:frozen-properties} with coefficients $a_h,b_h,\bar f'(U_\alpha)$.
The adjoint formula follows by differentiation, and translation conjugacy makes the gap uniform in $\alpha$.
\end{proof}

\begin{lemma}[Corrector and residual estimates]\label{sp:corrector-estimates}\label{sp:nonlinear-lemma}
For small $L$, 
$d_{L,\alpha}=1+O(L)\geq1/2$,
and the $V$ norms of $q_{L,\alpha},\psi_{L,\alpha}$ and their first two phase derivatives are uniformly bounded.
Moreover, 
we have
	\begin{align}
		\|R_{L,\alpha}\|_{V'}+
		\|\partial_\alpha R_{L,\alpha}\|_{V'}&\leq CL,
		\label{sp:residual}\\
		\|A_{L,\alpha}^*\psi_{L,\alpha}\|_2+
		\|\partial_\alpha(A_{L,\alpha}^*\psi_{L,\alpha})\|_2&\leq C,
		\label{sp:adjointbound}
	\end{align}
and
	\begin{equation}\label{sp:multiplier}
		\|[f_u(x/L,V_{L,\alpha})-\bar f'(U_\alpha)]w\|_{V'}
		\leq CL\|w\|_V\qquad(w\in V).
	\end{equation}
Finally, 
for $\|w\|_2,\|z\|_2\leq r\leq r_*$ and all phases $\alpha,\beta$, 
one has
\begin{align}\label{sp:nonlinear}
	\|N_{L,\alpha}(w)\|_{V'}&\leq C\|w\|_2^2,\notag\\
	\|N_{L,\alpha}(w)-N_{L,\beta}(z)\|_{V'}
	&\leq Cr\|w-z\|_2+Cr^2|\alpha-\beta|.
\end{align}
\end{lemma}
\begin{proof}
The periodic functions $\chi,\zeta$ and their derivatives are bounded.
Their $L$-weighted factors in \eqref{sp:approx} and \eqref{sp:phasefunctional} are independent of $\alpha$ and have bounded spatial derivatives.
Together with the decay estimates in Lemma~\ref{sp:limit-spectrum}, this gives the stated $V$ bounds for $q_{L,\alpha},\psi_{L,\alpha}$ and their phase derivatives.
Moreover, $d_{L,\alpha}=\langle\psi_\alpha,q_{0,\alpha}\rangle+O(L)=1+O(L)$.
Write $y=x/L$.
The cell identity $a(1+\chi')=a_h$ and \eqref{sp:standing} give
	\[
	R_{L,\alpha}=p_2(y)U_\alpha''+p_1(y)U_\alpha'
	+p_0(y,U_\alpha)+L r_{L,\alpha},
	\]
where
	\[
	p_2=(a\chi)',\quad p_1=a_hb/a-b_h,\quad
	p_0(y,u)=f(y,u)-\bar f(u),
	\]
and
	\[
	\begin{aligned}
	r_{L,\alpha}(x)
	={}&a(y)\chi(y)U_\alpha'''(x)
	+b(y)\chi(y)U_\alpha''(x)\\
	&+\chi(y)U_\alpha'(x)
	\int_0^1 f_u\bigl(y,U_\alpha(x)
	+sL\chi(y)U_\alpha'(x)\bigr)\,ds.
	\end{aligned}
	\]
The bounds on the periodic coefficients and on $f_u,f_{uu}$, 
together with the exponential decay of $U_\alpha'$ and its derivatives,
give $\|r_{L,\alpha}\|_{\mathsf X} + \|\partial_\alpha r_{L,\alpha}\|_{\mathsf X}\leq C$.
All three $p_j$ have zero mean in $y$ with $p_0(y,0)=p_0(y,1)=0$.
If $P_y=p$ is a periodic primitive, 
then
	\[
	p(x/L,U_\alpha)w
	=L\partial_x[P(x/L,U_\alpha)w]
	-LP(x/L,U_\alpha)w'
	-LP_u(x/L,U_\alpha)U_\alpha'w.                 \tag{*}
	\]
For $p_0$ alone its zero-mean primitive also vanishes at $u=0,1$, so that $P(x/L,U_\alpha)$ and $P_u(x/L,U_\alpha)U_\alpha'$ lie in $\mathsf X$ uniformly.
The same primitive argument for $p_2U_\alpha''$ and $p_1U_\alpha'$ proves \eqref{sp:residual}, 
including its derivative.
Apply (*) with $p=f_u-\bar f'$ and $w\in V$.
Its three terms have $V'$ norm at most $CL\|w\|_V$.
Replacing $U_\alpha$ by $V_{L,\alpha}$ in $f_u$ adds at most $CL\|w\|_2$.
This proves \eqref{sp:multiplier}.

For any $C^3$ function $\varphi$ whose derivatives through order three decay exponentially as $x\to\pm\infty$, set $\varphi_L=\varphi+L\chi(x/L)\varphi'+L\zeta(x/L)\varphi$.
The exact flux identity
	\begin{equation}\label{sp:adjointidentity}
		a(x/L)\varphi_L'-b(x/L)\varphi_L
		=a_h\varphi'-b_h\varphi
		+L\{a\chi\varphi''+(a\zeta-b\chi)\varphi'-b\zeta\varphi\}
	\end{equation}
has all periodic factors evaluated at $x/L$.
Differentiating the right side once produces only bounded coefficients multiplying functions with exponential decay as $x\to\pm\infty$.
Taking $\varphi=\psi_\alpha$, 
and then its phase derivative, 
proves \eqref{sp:adjointbound}.

Finally, 
$L^1(\R)\hookrightarrow V'$ and Taylor's formula give the first estimate in \eqref{sp:nonlinear}.
The pointwise mean-value bound
$|N_{L,\alpha}(w)-N_{L,\beta}(z)| \leq C(|w|+|z|)|w-z|+C|z|^2|\alpha-\beta|$
follows from bounded $f_{uu},f_{uuu},q_{L,\alpha}$.
Its $L^1$ norm gives the second estimate.
\end{proof}

\begin{lemma}[Small-period phase--normal decomposition]
\label{sp:chart-unique}
There are $L_0,\rho_0>0$ such that, for $0<L<L_0$, every $u\in\mathcal U_{L,\rho_0}$ defined by \eqref{sp:chart}
has a unique representation $u=V_{L,\alpha}+w$ satisfying $\ell_{L,\alpha}w=0$ and $\|w\|_2<\rho_0$.
The inverse coordinate map $u\mapsto(\alpha,w)$ is continuous.
\end{lemma}
\begin{proof}
For $u=V_{L,\alpha_0}+w_0$ with $\|w_0\|_2<\rho_0$, Lemma~\ref{sp:corrector-estimates} gives $\|u-V_{L,\alpha}\|_2\leq\rho_0+C|\alpha-\alpha_0|$ and $\|\ell_{L,\alpha}'\|_{\mathsf X^*}\leq C$.
Consequently the phase function $\alpha\mapsto\ell_{L,\alpha}(u-V_{L,\alpha})$ satisfies
\[
\partial_\alpha[\ell_{L,\alpha}(u-V_{L,\alpha})]
=\ell_{L,\alpha}'(u-V_{L,\alpha})-d_{L,\alpha}
\leq-\tfrac12+C(\rho_0+|\alpha-\alpha_0|).
\]
Fix $r>0$ sufficiently small, then choose $\rho_0,L_0>0$ so that this derivative is at most $-1/4$ whenever $|\alpha-\alpha_0|\leq r$.
The interval and constants are independent of $\alpha_0$ and $L$.
The implicit function theorem gives continuous local coordinates.
Strict decrease of $U$ gives $\|U_\alpha-U_\beta\|_2\geq c_r>0$ when $|\alpha-\beta|\geq r$.
Since $\|V_{L,\alpha}-U_\alpha\|_2\leq CL$, two representations with normal norms below $\rho_0$ satisfy $\|U_\alpha-U_\beta\|_2\leq2\rho_0+2CL<c_r$ after reducing $\rho_0,L_0$ if needed.
Hence $|\alpha-\beta|<r$.
Local uniqueness then proves the assertion.
\end{proof}

\subsubsection{Uniform transverse resolvents and the ambient extension}

We use the sectors in \eqref{lp:sector}.

\begin{lemma}[Uniform augmented resolvent estimate]\label{sp:resolvent}
There exist $0<\eta<1$, $0<\theta<\pi/2$, 
$L_0>0$, 
and $C$ such that, 
for all $0<L\leq L_0$, $\alpha\in\R$, $\lambda\in\Omega_{\eta,\theta}$, 
and $h\in V'$, 
the augmented problem
	\begin{equation}\label{sp:augmented}
		(\lambda-A_{L,\alpha})w+\mu q_{L,\alpha}=h,
		\qquad \ell_{L,\alpha}w=0
	\end{equation}
has a unique solution in $V\times\mathbb C$, and
	\begin{equation}\label{sp:resolvent-estimate}
		(1+|\lambda|)^{1/2}\|w\|_2+\|w\|_V+|\mu|
		\leq C\|h\|_{V'}.
	\end{equation}
If, in addition, 
$h\in\mathsf X$, 
then $(1+|\lambda|)\|w\|_2\leq C\|h\|_2$.
\end{lemma}
\begin{proof}
We follow the sectoriality--compactness argument of Lemma~\ref{lp:stable}.
There is a positive $L$-periodic density $m_L$, normalized by its mean, such that
	\begin{equation}\label{sp:density}
		a(x/L)m_L'-b(x/L)m_L=J_L,\qquad
		m_L=1+O(L),\quad \|m_L'\|_\infty+|J_L|\leq C.
	\end{equation}
To see this directly, 
in the variable $y=x/L$ write $m_L(Ly)=1+Lr(y)$ and impose $\int_0^1r=0$.
The equation is $ar'-b(1+Lr)=J$.
At $L=0$ its solution has $J=-b_h$ and $r'=(b-b_h)/a$.
On zero-mean periodic functions, $T(r,J)=ar'-J$ is an isomorphism onto periodic continuous functions:
integration after division by $a$ determines $J$ and then $r$.
Writing $B(r,J)=br$, the equation becomes $(I-LT^{-1}B)(r,J)=T^{-1}b$.
For small $L$, a Neumann series gives the uniform bounds in \eqref{sp:density}.
In particular, $m_L=1+Lr>0$.

Let $W_{L,\alpha}=f_u(x/L,V_{L,\alpha})$.
There is a fixed $R$ such that $W_{L,\alpha}\leq-\gamma$ for $|x-\alpha|\geq R$.
Choose $\widetilde W_{L,\alpha}\leq-\gamma/2$ equal to $W_{L,\alpha}$ outside that interval, with uniformly bounded magnitude, and let $\widetilde A$ be the corresponding differential operator.
Choose $0<\eta<\min\{\gamma/4,1\}$.
Testing $\lambda-\widetilde A$ against $m_L\bar w$ gives, apart from $(\lambda+\eta)\int m_L|w|^2$, the form
	\[
	\mathfrak q_\eta(w,w) := \int a m_L|w'|^2-\int m_L(\widetilde W+\eta)|w|^2
	+J_L\int w'\bar w.
	\]
Since $\operatorname{Re}\int w'\bar w=0$, its real part is at least $c\|w\|_V^2$, while its imaginary part is bounded by $C\|w\|_V^2$.
Thus its values lie in a fixed sector $|\arg z|\leq\omega<\pi/2$.
Choose $\theta_*\in(\omega,\pi/2)$.
For each $\lambda\in\Omega_{\eta,\theta_*}$, sector geometry gives an angle $\beta_\lambda$ depending only on $\lambda$ such that
	\[
	\operatorname{Re}\left\{e^{-i\beta_\lambda}
	\left[(\lambda+\eta)\int_\R m_L|w|^2\,dx
	+\mathfrak q_\eta(w,w)\right]\right\}
	\ge c_{\theta_*}\bigl(|\lambda+\eta|\|w\|_2^2+\|w\|_V^2\bigr).
	\]
Complex Lax--Milgram makes $\lambda-\widetilde A:V\to V'$ uniformly invertible on $\Omega_{\eta,\theta_*}$.
The operator in \eqref{sp:augmented} has the decomposition
	\[
	\begin{aligned}
	((\lambda-A_{L,\alpha})w+\mu q_{L,\alpha},\,\ell_{L,\alpha}w)
	&=((\lambda-\widetilde A)w,\mu)\\
	&\quad+((\widetilde W-W)w+\mu q_{L,\alpha},
	\,\ell_{L,\alpha}w-\mu).
	\end{aligned}
	\]
The first map is invertible.
In the second, $\widetilde W-W:V\to V'$ is compact by local $H^1\hookrightarrow L^2$ compactness, and the remaining terms have rank at most two.
Thus the augmented operator is Fredholm of index zero, as in Lemma~\ref{lp:stable}.
For bounded $\lambda$ in this sector, the inverse of $\lambda-\widetilde A$ and
	\[
	 \|(W-\widetilde W)w\|_{V'}
	 \leq C\|w\|_{L^2(|x-\alpha|\leq R)}
	\]
give
	\begin{equation}\label{sp:tightness}
		\|w\|_V\leq C\bigl(\|(\lambda-A_{L,\alpha})w\|_{V'}
		+\|w\|_{L^2(|x-\alpha|\leq R)}\bigr).
	\end{equation}

Pairing \eqref{sp:augmented} with $\psi_{L,\alpha}$ gives the exact identity $\mu=(\ell h+\ell(A_{L,\alpha}w))/d_{L,\alpha}$.
In particular, 
when $\ell h=0$, 
$\mu=\ell(A_{L,\alpha}w)/d_{L,\alpha}$ and $|\mu|\leq C\|w\|_2$ by \eqref{sp:adjointbound}.
The projected operator $A_{L,\alpha}-q_{L,\alpha}\ell_{L,\alpha}A_{L,\alpha}/d_{L,\alpha}$ is thus a bounded finite-rank perturbation of $A_{L,\alpha}$.
The argument used in Lemma~\ref{lp:stable}, now on $V\to V'$ with the density-weighted form and a fixed shift $M>0$, gives the large-parameter sector estimate:
for some $0<\theta_0<\pi/2$ and $\lambda\in\Omega_{-M,\theta_0}$,
	\[
	|\lambda-M|\|w\|_2^2+\|w\|_V^2\leq C\|h\|_{V'}^2.
	\]
For $h\in\mathsf X$, the same form estimate gives $|\lambda-M|\|w\|_2\leq C\|h\|_2$ away from a fixed ball.
General $h$ is reduced to this case by replacing it with $h-q\ell h/d$ and replacing $\mu$ with $\mu-\ell h/d$.

By Lemma~\ref{sp:limit-spectrum} and the limiting sectorial bound, decrease $\eta>0$ if necessary and choose $0<\theta_1<\pi/2$ so that the limiting augmented operator
is invertible on the closure of $\Omega_{\eta,\theta_1}$.
Take $\max(\theta_0,\theta_1)<\theta<\pi/2$.
The same two-angle sector geometry used in Lemma~\ref{lp:stable} covers all sufficiently large parameters.
That is, for some $M_1$,
	\begin{equation}\label{sp:sector-cover}
		\Omega_{\eta,\theta}\cap\{|\lambda|>M_1\}
		\subset\Omega_{-M,\theta_0}.
	\end{equation}
The remaining parameters form a compact subset of the resolvent region of the limiting augmented operator.

Suppose the uniform bound on this compact set fails along $L_n\to0$.
The large-$M$ energy inequality permits normalization so that
	\[
	\|w_n\|_2=1,\quad \|h_n\|_{V'}\to0,\quad
	\|w_n\|_V+|\mu_n|\leq C,\quad \lambda_n\to\lambda_*.
	\]
Translate by integer multiples of $L_n$ to arrange $0\leq\alpha_n<L_n$.
Along a subsequence, $w_n\rightharpoonup w_*$ in $V$, $w_n\to w_*$ locally in $\mathsf X$, and $\mu_n\to\mu_*$.
Test the equation against $\varphi_n=\varphi+L_n\chi(x/L_n)\varphi' +L_n\zeta(x/L_n)\varphi$ for $\varphi\in C_c^\infty$.
Write $a_n=a(x/L_n)$, $b_n=b(x/L_n)$, 
and $W_n=f_u(x/L_n,V_{L_n,\alpha_n})$.
By \eqref{sp:adjointidentity}, 
\eqref{sp:multiplier}, 
$\|w_n\|_V\leq C$, 
and $\|\varphi_n-\varphi\|_2=O(L_n)$,
	\[
	\begin{aligned}
	&\int_\R\bigl(a_nw_n'\overline{\varphi_n'}
	    -b_nw_n'\overline{\varphi_n}
	    -W_nw_n\overline{\varphi_n}\bigr)\,dx\\
	&\quad=\int_\R\bigl(a_hw_*'\bar\varphi'
	    -b_hw_*'\bar\varphi
	    -\bar f'(U_0)w_*\bar\varphi\bigr)\,dx+o(1).
	\end{aligned}
	\]
Also $\psi_{L_n,\alpha_n}\to\psi_0$ and $q_{L_n,\alpha_n}\to q_{0,0}$ in $\mathsf X$.
Thus
	\[
	(\lambda_*-A_{0,0})w_*+\mu_*q_{0,0}=0,\qquad
	\langle\psi_0,w_*\rangle=0.
	\]
The inverse of the limiting augmented operator implies $w_*=0$ and $\mu_*=0$.
Since $w_*=0$, compact embedding on the fixed core gives $\|w_n\|_{L^2(|x-\alpha_n|\leq R)}\to0$.
Since also $\mu_n\to0$ and $q_{L_n,\alpha_n}$ is bounded in $V'$, \eqref{sp:tightness}, 
applied with right side $h_n-\mu_nq_{L_n,\alpha_n}$, 
yields $w_n\to0$ in $V$, 
a contradiction.
This proves the estimate on the compact remainder.
Fredholm index zero and injectivity give existence.
Combining this with \eqref{sp:sector-cover} proves all the stated estimates.
\end{proof}

On $N_\alpha$, define $B_\alpha=Q_\alpha A_{L,\alpha}$.
Since $\eta<1$, 
the scalar eigenvalue $-1$ lies in the excluded sector.
Put
\begin{equation}\label{sp:extended}
 \widehat B_\alpha=B_\alpha Q_\alpha-P_\alpha.
\end{equation}

\begin{lemma}[Small-period ambient linear estimates]
\label{sp:ambient-semigroup}\label{sp:operatorlip-lemma}
For sufficiently small $L$, 
$\widehat B_\alpha$ is uniformly sectorial on $V'$ with domain $V$.
Its realization on $\mathsf X$ is also uniformly sectorial, 
with domain
\[
D_{\mathsf X}(\widehat B_\alpha)
:=\{w\in V:\widehat B_\alpha w\in\mathsf X\}.
\]
Furthermore, 
there exist constants $C,\omega>0$ independent of $L$ and $\alpha$, such that, 
for all $t>0$,
\begin{equation}\label{sp:semigroups}
 \|e^{t\widehat B_\alpha}\|_{\mathsf X\to\mathsf X}\leq Ce^{-\omega t},\qquad
 \|e^{t\widehat B_\alpha}\|_{V'\to\mathsf X}
 \leq C(1+t^{-1/2})e^{-\omega t}.
\end{equation}
\begin{equation}\label{sp:operatorlip}
	\|\widehat B_\alpha-\widehat B_\beta\|_{\mathsf X\to V'}
	\leq C|\alpha-\beta|.
\end{equation}
\end{lemma}
\begin{proof}
As in \eqref{lp:sum-action}, 
$\ell_{L,\alpha}\widehat B_\alpha=-\ell_{L,\alpha}$ and $\widehat B_\alpha q_{L,\alpha}=-q_{L,\alpha}$.
For a source $h$, 
solve its $Q_\alpha h$ component by the augmented problem in Lemma~\ref{sp:resolvent};
the complementary component is $(\lambda+1)^{-1}P_\alpha h$.
The uniform bound $P_\alpha:V'\to V$ then gives
\[
 \begin{split}
 \|(\lambda-\widehat B_\alpha)^{-1}\|_{V'\to V}&\leq C,\\
 \|(\lambda-\widehat B_\alpha)^{-1}\|_{V'\to\mathsf X}&\leq C(1+|\lambda|)^{-1/2},\\
 \|(\lambda-\widehat B_\alpha)^{-1}\|_{\mathsf X\to\mathsf X}&\leq C(1+|\lambda|)^{-1}.
 \end{split}
\]
Since $\widehat B_\alpha:V\to V'$ is uniformly bounded, the resolvent identity also gives the sectorial $V'\to V'$ bound.
Consistent contour integration gives \eqref{sp:semigroups}.
For \eqref{sp:operatorlip}, expand
\[
 \widehat B_\alpha=A_{L,\alpha}-A_{L,\alpha}P_\alpha-P_\alpha A_{L,\alpha}
                  +P_\alpha A_{L,\alpha}P_\alpha-P_\alpha.
\]
The potential of $A_{L,\alpha}$ is Lipschitz in $\alpha$, and $P_\alpha:\mathsf X\to V$ is uniformly Lipschitz.
The identity $\ell_{L,\alpha}A_{L,\alpha}w=(w,A_{L,\alpha}^*\psi_{L,\alpha})_{\mathsf X}$ and \eqref{sp:adjointbound}
give the same bound for $P_\alpha A_{L,\alpha}:\mathsf X\to V'$.
Each term in the expansion thus has the asserted difference bound.
\end{proof}

Define
\begin{equation}\label{sp:ambient-nonlinearity}
 G_L(\alpha,w)
 =Q_\alpha\bigl(R_{L,\alpha}+N_{L,\alpha}(w)\bigr)
  -\nu_L(\alpha,w)\frac{q_{L,\alpha}}{d_{L,\alpha}}
     \ell_{L,\alpha}'w.
\end{equation}
Together with $\dot\alpha=\nu_L(\alpha,w)$, the ambient normal equation is
\begin{equation}\label{sp:normal}
 w_t=\widehat B_\alpha w+G_L(\alpha,w).
\end{equation}

\begin{lemma}[Constraint preservation and small-period ambient estimates]
\label{sp:ambient-estimates}
For sufficiently small $L_0,r_*>0$, 
the scalar formula in \eqref{sp:coordinates} is well defined and locally Lipschitz on $\R\times B_{\mathsf X}(0,r_*)$, 
with $d_{L,\alpha}-\ell_{L,\alpha}'w\geq1/4$.
Moreover, 
every regular solution of the ambient system satisfies
\begin{equation}\label{sp:constraint-transport}
 \frac{d}{dt}\bigl(\ell_{L,\alpha(t)}w(t)\bigr)
 =-\ell_{L,\alpha(t)}w(t).
\end{equation}
On $\ell_{L,\alpha}w=0$, 
the ambient system agrees with \eqref{sp:coordinates}.
Finally, 
for $0<L<L_0$, 
$0<r\leq r_*$, 
and $\|w\|_2,\|z\|_2\leq r$, 
one has
\begin{align}\label{sp:coordinatebounds}
	|\nu_L(\alpha,w)|&\leq C(L+r),\notag\\
	|\nu_L(\alpha,w)-\nu_L(\beta,z)|
	&\leq C\|w-z\|_2+C(L+r)|\alpha-\beta|,\notag\\
	\|G_L(\alpha,w)\|_{V'}&\leq C(L+r^2),\notag\\
	\|G_L(\alpha,w)-G_L(\beta,z)\|_{V'}
	&\leq C(L+r)\|w-z\|_2+C(L+r^2)|\alpha-\beta|.
\end{align}
\end{lemma}
\begin{proof}
By Lemma~\ref{sp:corrector-estimates}, $\ell_{L,\alpha}A_{L,\alpha}$ extends to $\mathsf X$ by pairing with $A_{L,\alpha}^*\psi_{L,\alpha}$.
Also $d_{L,\alpha}\geq1/2$ and $|\ell_{L,\alpha}'w|\leq Cr$, so the denominator is at least $1/4$ for small $r_*$.
The residual, 
adjoint, 
and nonlinear bounds in that lemma, 
applied to the quotient \eqref{sp:coordinates},
give the estimates for $\nu_L$.
Insertion into \eqref{sp:ambient-nonlinearity} gives the estimates for $G_L$.
The identities $\ell_{L,\alpha}\widehat B_\alpha=-\ell_{L,\alpha}$ and $\ell_{L,\alpha}(q_{L,\alpha}/d_{L,\alpha})=1$
give \eqref{sp:constraint-transport} by the calculation leading to \eqref{lp:constraint-transport}.
On the constraint, \eqref{sp:coordinates} identifies the $P_\alpha$ component, proving equivalence.
\end{proof}

\subsubsection{Parabolic evolution and time-map estimates}

\begin{lemma}[Small-period flow estimates]\label{sp:flow-estimates}
Fix $T>0$.
There are $C_T,\delta_T>0$ such that, for $0<L<\delta_T$ and $0\leq r_0<\delta_T$, the coordinate equations \eqref{sp:coordinates}
have a unique flow on $[0,T]$ for initial data in \eqref{sp:chart} with $\|w(0)\|_2\leq r_0$.
The trajectories remain in a radius-$r$ tube with $r=C_T(r_0+L)<\min\{\rho_0,r_*\}$.
Furthermore, 
the following estimates hold for a single trajectory and for a pair of trajectories, 
respectively,
with $D(t)=\|w_1(t)-w_2(t)\|_2$ and $A(t)=|\alpha_1(t)-\alpha_2(t)|$ in the latter case:
\begin{align}
	\|w(T)\|_2&\leq Ce^{-\omega T}\|w(0)\|_2+C_T(L+r^2),
	\label{sp:oneorbit}\\
	\sup_{0\leq t\leq T}(D(t)+A(t))&\leq C_T(D(0)+A(0)),\notag\\
	D(T)&\leq Ce^{-\omega T}D(0)
	+C_T(L+r)(D(0)+A(0)), \label{sp:twoorbit}\\
	|(\alpha_1(t)-\alpha_2(t))-(\alpha_1(0)-\alpha_2(0))|
	&\leq C_TD(0)+C_T(L+r)A(0),
	\quad 0\leq t\leq T. \label{sp:phaseincrement}
\end{align}
\end{lemma}
\begin{proof}
Fix a tentative radius $r>r_0$.
For each local solution, let
\[
\tau=\min\bigl\{T,\inf\{t>0:\|w(t)\|_2=r\}\bigr\},
\qquad \inf\varnothing=+\infty.
\]
We first work on $[0,\tau]$ within the solution's existence interval.
For two solutions, use the smaller of their stopping times.
Set $S_\alpha(t)=e^{t\widehat B_\alpha}$ and $k(t)=C(1+t^{-1/2})e^{-\omega t}$.
Freezing \eqref{sp:normal} at $\alpha_*=\alpha(0)$ gives, in $\mathsf X$,
\[
 w(t)=S_{\alpha_*}(t)w(0)+\int_0^t S_{\alpha_*}(t-s)
 \bigl[(\widehat B_{\alpha(s)}-\widehat B_{\alpha_*})w(s)
       +G_L(\alpha(s),w(s))\bigr]\,ds.
\]
This uses the consistent $V'\to\mathsf X$ semigroup in \eqref{sp:semigroups}.
The integrable kernel $k$ controls the $V'$ forcing.
By \eqref{sp:coordinatebounds}, $|\alpha(t)-\alpha_*|\leq CT(L+r)$, so \eqref{sp:operatorlip} bounds the first forcing term by $C_T(L+r)r$.
Consequently,
\[
 \|w(t)\|_2\leq Ce^{-\omega t}\|w(0)\|_2+C_T(L+r^2),
 \qquad 0\leq t\leq\tau.
\]

For two solutions, put $\delta w=w_1-w_2$ and freeze their difference at $\alpha_*=\alpha_1(0)$.
Its Duhamel forcing is
\[
 (\widehat B_{\alpha_1}-\widehat B_{\alpha_*})\delta w
 +(\widehat B_{\alpha_1}-\widehat B_{\alpha_2})w_2
 +G_L(\alpha_1,w_1)-G_L(\alpha_2,w_2).
\]
Here and in the next bounds the phases and normal coordinates in the forcing are evaluated at the integration time.
Equations \eqref{sp:operatorlip} and \eqref{sp:coordinatebounds} bound this forcing in $V'$ by $C_T(L+r)(D+A)$, while the phase equations give
\[
\begin{aligned}
D(t)&\leq Ce^{-\omega t}D(0)
 +C_T(L+r)\int_0^t k(t-s)(D(s)+A(s))\,ds,\\
A(t)&\leq A(0)+C\int_0^t D(s)\,ds
 +C(L+r)\int_0^t A(s)\,ds.
\end{aligned}
\]
Gronwall's inequality in the second line gives $A(t)\leq C_T(A(0)+\int_0^tD(s)\,ds)$.
Since $k\in L^1(0,T)$, taking the supremum in the first line and then absorbing $C_T(L+r)\sup_{[0,\tau]}D$ gives $\sup_{[0,\tau]}D\leq C_T(D(0)+(L+r)A(0))$ for small $L+r$.
The bound for $A$ now yields $\sup_{[0,\tau]}(D+A)\leq C_T(D(0)+A(0))$.
Substitution into the first line gives $D(t)\leq Ce^{-\omega t}D(0)+C_T(L+r)(D(0)+A(0))$.
Integrating the difference of the phase equations once more gives \eqref{sp:phaseincrement}.

Choose $r=C_T(r_0+L)$ with $C_T$ large enough to dominate the linear term.
For small $r_0,L$, the bound $C_Tr^2<r/2$ rules out a first exit before $T$.
The short-interval contraction for the integral equation gives existence, uniqueness and continuous dependence;
the first-exit bound gives continuation and \eqref{sp:oneorbit}--\eqref{sp:phaseincrement}.
Regularization and approximation as in \cite{Lunardi1995,PrussSimonett2016} identify this flow with the parabolic solution.
They also pass \eqref{sp:constraint-transport} to initial data in $\mathsf X$, proving constraint preservation.
\end{proof}

\subsubsection{The constrained invariant graph}

Consider the complete graph class
\[
	\mathcal G_{\rho,K}=\{g:\R\to\mathsf X:\ell_{L,\alpha}g(\alpha)=0,
 \ \|g\|_\infty\leq\rho,\ \operatorname{Lip}(g)\leq K\}
\]
with the supremum metric.

\begin{proposition}[Existence and invariance of the small-period graph]
\label{sp:graph-proposition}
For suitable $\rho,K,T>0$ and all sufficiently small $L$, the time-$T$ graph transform on $\mathcal G_{\rho,K}$
has a unique fixed point $g_L$.
Its graph is invariant under the full semiflow of \eqref{sp:normal} with $\dot\alpha=\nu_L(\alpha,w)$, and
\[
 \|g_L\|_\infty+\operatorname{Lip}_{\mathsf X}(g_L)\leq CL.
\]
\end{proposition}
\begin{proof}
Fix $T$ with $\vartheta=Ce^{-\omega T}<1/16$.
For the time-$T$ image $(\rho_g,z_g)$ of a graph, Lemma~\ref{sp:flow-estimates} gives
$\operatorname{Lip}(\rho_g-I)\leq C_T(K+\rho+L)$ and $\|\rho_g-I\|_\infty\leq C_T(\rho+L)$.
Choose $K$ small, then $\rho\ll K$, and finally $L\ll\rho$.
Then $\rho_g$ is an increasing homeomorphism with inverse Lipschitz constant at most $2$.
The transform $\mathcal Tg=z_g\circ\rho_g^{-1}$ preserves the constraint and satisfies
\[
 \begin{split}
 \|\mathcal Tg\|_\infty&\leq\vartheta\rho+C_T(L+\rho^2),\\
 \operatorname{Lip}(\mathcal Tg)&\leq2\vartheta K+C_T(L+\rho)(1+K).
 \end{split}
\]
Reducing $\rho,L$ makes it a self-map of $\mathcal G_{\rho,K}$.
For $d=\|g-\tilde g\|_\infty$ and a common final phase,
the inverse base maps differ by at most $C_Td$,
while the normal images at a common initial phase differ by at most
$[\vartheta+C_T(L+\rho)]d$ and have Lipschitz constant at most
$\vartheta K+C_T(L+\rho)(1+K)$.
Comparison at the common final phase gives
$\|\mathcal Tg-\mathcal T\tilde g\|_\infty
\leq[\vartheta+C_T(K+\rho+L)]d<\tfrac12d$,
so the graph transform has a unique fixed point $g_L$.

Put $R_*=\|g_L\|_\infty$ and $K_*=\operatorname{Lip}(g_L)$.
Using the actual radius and slope in the self-map estimates gives
\[
 \begin{split}
 R_*&\leq\vartheta R_*+C_T(L+R_*^2),\\
 K_*&\leq2\vartheta K_*+C_T(L+R_*)(1+K_*).
 \end{split}
\]
Since $R_*\leq\rho$ and $K_*\leq K$, absorption proves $R_*+K_*\leq CL$.
The intermediate-time bounds now make every time-$s$ image, 
$0\leq s\leq T$, 
a constrained graph with radius and slope $O_T(L)$.
For small $L$ these images belong to $\mathcal G_{\rho,K}$.
Commutation of the time maps and uniqueness of the fixed graph therefore give invariance for all $t\geq0$.
\end{proof}

\subsubsection{Profile regularity, range, and normal attraction}

\begin{proof}[Completion of the small-period construction in Theorem~\ref{sp:main}]
Set $Z_{L,\alpha}=V_{L,\alpha}+g_L(\alpha)$ using Proposition~\ref{sp:graph-proposition}.
The unique chart in Lemma~\ref{sp:chart-unique} makes this parametrization injective.
Translating a fixed graph by the simultaneous substitutions $x\mapsto x+L$, $\alpha\mapsto\alpha+L$ preserves its phase condition and graph space.
Uniqueness proves \eqref{sp:covariance}.
Put $s_L(\alpha)=\nu_L(\alpha,g_L(\alpha))$.
Equations \eqref{sp:coordinatebounds} and \eqref{sp:graphbound} prove its asserted bounds and periodicity.
Its scalar flow is global in both time directions, which proves the entire-orbit statement in \eqref{sp:reduction}.

Here $L$ is fixed, and constants in parabolic smoothing may depend on $L$.
Entire graph orbits have uniformly bounded $\mathsf X$ distance from $U_{\alpha(t)}$, and $|\dot\alpha|\leq CL$.
Set $h(t)=u(t)-U_{\alpha(0)}$ on $0\leq t\leq1$.
For fixed $L$, $F_L(U_{\alpha(0)})$ is bounded in $L^2$ uniformly in $\alpha(0)$, and the reaction difference in the equation for $h$ is Lipschitz in $h$.
Testing that equation against $h$ and absorbing the drift term gives
\[
\frac12\frac{d}{dt}\|h\|_2^2+\frac{a_*}2\|h_x\|_2^2
\leq C_L(1+\|h\|_2^2).
\]
The initial $L^2$ bound and Gronwall's inequality control $\sup_{0\leq t\leq1}\|h(t)\|_2^2+\int_0^1\|h_x(t)\|_2^2\,dt$.
Choose a time in the first half-interval at which that seminorm is bounded, differentiate the equation in $x$, and test against its derivative.
Boundedness of the fixed-$L$ coefficient derivatives and Young's inequality give an $H^1$ Gronwall bound at the final time.
This proves
\begin{equation}\label{sp:fixedLregularity}
	\sup_\alpha\|Z_{L,\alpha}-U_\alpha\|_{H^1}\leq C_L.
\end{equation}
Repeating this argument on shorter interior intervals, with the stated coefficient regularity, gives classical regularity and a uniform bound on $u(t,x)$ for all $t,x$ on the curve.
The entire-orbit property allows application of these estimates at every time.

No graph profile is constant $0$ or $1$, because $Z_{L,\alpha}-U_\alpha\in L^2(\R)$ whereas the difference between $U_\alpha$ and either constant is not in $L^2(\R)$.
Thus Lemma~\ref{common:range} gives $0<Z_{L,\alpha}<1$.

The map $\alpha\mapsto Z_{L,\alpha}-U_\alpha$ is continuous into $\mathsf X$.
Restricted to $0\leq\alpha\leq L$ its image is compact, so its $L^2$ tails vanish uniformly.
For $h\in H^1(\R)$ a cutoff and the one-dimensional Sobolev inequality give
\[
\sup_{|x|\geq R+1}|h(x)|^2
\leq C\|h\|_{L^2(|x|\geq R)}\|h\|_{H^1(\R)}.
\]
Apply this to $Z_{L,\alpha}-U_\alpha$, use \eqref{sp:fixedLregularity}, and then use covariance for all other phases.
The exponential tails of $U$ prove \eqref{sp:tails}.
The same interpolation argument shows that the parametrization is continuous in the uniform norm for this fixed $L$.

For attraction in a sufficiently small radius-$r$ tube, compare with the graph orbit starting at the same phase.
Set $e(t)=\|w(t)-g_L(\alpha(t))\|_2$.
For the $T,\vartheta$ chosen above and $K_*=\operatorname{Lip}_{\mathsf X}(g_L)\leq CL$,
estimates \eqref{sp:twoorbit}--\eqref{sp:phaseincrement} give
\begin{equation}\label{sp:normalcontraction}
	\|w(T)-g_L(\alpha(T))\|_2\leq q\|w(0)-g_L(\alpha(0))\|_2,
	\qquad q\leq\vartheta+C_T(L+r+K_*)<1.
\end{equation}
The intermediate-time comparison gives $e(t)\leq C_Te(0)$.
Choose the initial distance small enough that this bound keeps the solution inside the chart.
Iteration of \eqref{sp:normalcontraction} then gives global forward continuation and
$e(t)\leq Ce^{-\nu t}e(0)$, with $\nu=-T^{-1}\log q$.
For a complete orbit remaining in this neighborhood,
$e(0)\leq q^ne(-nT)\to0$; hence it lies on the graph.
\end{proof}

\section{Dynamics on the invariant front families}\label{sec:dynamics}

\subsection{The common phase equation}

The two constructions provide different estimates for the phase velocity, 
but the passage from a periodic phase equation to front dynamics is identical. 
We state that passage once. 
The front-tail condition below is uniform in the phase for each fixed $L$.

\begin{theorem}[Dynamics of a periodic invariant front family]
\label{phase:classification}
 Fix $L>0$. 
 Let $\{Z_p:p\in\R\}$ be an injectively parametrized invariant family of classical front-like profiles, 
 continuous in $p$ in the locally uniform topology, 
 with $0<Z_p<1$ and
 \begin{equation}\label{phase:abstract-hypotheses}
  Z_{p+L}(x+L)=Z_p(x),\qquad
  \lim_{R\to\infty}\sup_p\left\{
   \sup_{x-p\geq R}|Z_p(x)|+
   \sup_{x-p\leq-R}|1-Z_p(x)|\right\}=0.
 \end{equation}
 Suppose the flow on this family is $u(t,x)=Z_{p(t)}(x)$, 
 where $\dot p=v_L(p)$ and $v_L$ is $L$-periodic and Lipschitz.
 Then the complete trajectories on the family have the following classification.
 \begin{enumerate}
  \item If $v_L$ has no zero, 
  all states on the family lie on one pulsating front orbit up to time translation. 
  Its signed speed is given by \eqref{intro:passage}.
  \item Each zero $p_*$ of $v_L$ gives a stationary front $Z_{p_*}$.
  If the zero set is nonempty, 
  each component $(p_-,p_+)$ of its complement gives an entire orbit converging locally uniformly to the stationary fronts $Z_{p_-}$ and $Z_{p_+}$,
  in the direction specified by the sign of $v_L$ there.
  Consequently, 
  no orbit traverses all spatial phases in this case. 
  In particular, 
  if $v_L\equiv0$,
  every member of the family is stationary.
 \end{enumerate}
\end{theorem}

\begin{remark}
The zeros need not be nondegenerate or isolated.
\end{remark}

\begin{proof}
 If $v_L$ has no zero, 
 it has one sign and every lifted phase traverses $\R$. 
 Separation of variables gives the signed passage time
 \[
  \tau_L=\int_0^L\frac{dp}{v_L(p)},\qquad
  p(t+\tau_L)=p(t)+L.
 \]
 Thus $c_L=L/\tau_L$ is \eqref{intro:passage}. 
 The covariance in \eqref{phase:abstract-hypotheses} yields $u(t+\tau_L,x+L)=u(t,x)$, 
 also when $\tau_L<0$.
 The difference $p(t)-c_Lt$ is bounded and periodic. 
 Define
 \[
  \Phi_L(\xi,y)=u\big((Ly-\xi)/c_L,Ly\big).
 \]
 The space--time covariance makes $\Phi_L$ periodic in $y$ and gives $u(t,x)=\Phi_L(x-c_Lt,x/L)$. 
 Since $Ly-p((Ly-\xi)/c_L)=\xi+O(1)$ uniformly in $y$, 
 the tails in \eqref{phase:abstract-hypotheses} imply that $\Phi_L(\xi,y)$ tends uniformly in $y$ to $1$ 
 as $\xi\to-\infty$ and to $0$ as $\xi\to+\infty$. 
 Every initial phase occurs on this scalar orbit,
 so the front is unique on the family up to time translation.

 At a zero, 
 scalar uniqueness makes $p$ constant. 
 Injectivity also shows that every stationary state on the family arises this way.
 When zeros exist, 
 their periodic lift has bounded complementary components. 
 On any one of them, 
 $p(t)$ is strictly monotone and cannot reach a boundary phase in finite time because $v_L$ is Lipschitz.
 It therefore tends to the two boundary phases as $t\to\pm\infty$ in the order determined by the sign of $v_L$. 
 Continuity of $p\mapsto Z_p$ gives convergence to the corresponding stationary fronts.
\end{proof}

Covariance \eqref{lp:covariance} and the tails \eqref{lp:range} verify these hypotheses for Theorem~\ref{lp:main}, 
with $p=X$ and $v_L(p)=s_{1/L}(p/L)$. 
Likewise,
\eqref{sp:covariance} and \eqref{sp:tails} verify them for Theorem~\ref{sp:main}, 
with $p=\alpha$ and $v_L=s_L$. 
Thus the same classification applies to both limits. 

\begin{lemma}[Strict ordering by phase]\label{phase:ordering}
 After increasing the large-period threshold and decreasing the small-period threshold if necessary, 
 both invariant families are strictly ordered by their lifted position phase:
 \begin{equation}\label{phase:strict-order}
  p_1<p_2\quad\Longrightarrow\quad Z_{p_1}(x)<Z_{p_2}(x)
  \qquad(x\in\R).
 \end{equation}
 Here $Z_p=Z_{\epsilon,p}$ in the large-period regime and $Z_p=Z_{L,p}$ in the small-period regime.
 In the large-period regime, moreover,
 \begin{equation}\label{phase:positive-tangent}
  \partial_X Z_{\epsilon,X}(x)>0\qquad(X,x\in\R).
 \end{equation}
\end{lemma}

\begin{proof}
 For large periods, write $Y=\epsilon X$, $z=x-X$, and
 \[
  Q_X(x):=\partial_X Z_{\epsilon,X}(x)
  =-(U_\epsilon)_z(Y,z)+\epsilon(U_\epsilon)_Y(Y,z).
 \]
 The estimates in \eqref{lp:main-est}, 
 Sobolev embedding, 
 and Lemma~\ref{lp:frozen-properties} give
 $Q_X(x)=-\phi_z(Y,z)+O(\epsilon)$ uniformly.
 For every fixed $R$, the minimum of $-\phi_z$ on $\T\times[-R,R]$ is positive.
 Thus $Q_X>0$ when $|x-X|\leq R$, 
 for sufficiently small $\epsilon$, 
 and $\sup_X\|Q_X\|_\infty\leq M$.
 Choose $R$ using the uniform tails in \eqref{lp:range}, 
 so that the profiles outside this interval lie in the linear stability neighborhoods of $0$ and $1$.

 Along any complete graph orbit $u(t,x)=Z_{\epsilon,X(t)}(x)$, 
 put $q(t,x)=Q_{X(t)}(x)$.
 Differentiating the flow with respect to its initial phase gives the linearized equation for $Jq$, 
 where $J=\partial_{X(0)}X(t)$ and $J'=v_L'(X(t))J$.
 Hence
 \begin{equation}\label{phase:tangent-equation}
  q_t=\partial_x(a(x/L)q_x)+b(x/L)q_x+
  \bigl[f_u(x/L,u)-v_L'(X(t))\bigr]q.
 \end{equation}
 This differentiation is justified by the $C^1$ parametrization and the differentiable parabolic time maps. 
 Positive-time regularization gives the classical linear equation.
 Since $\|v_L'\|_\infty=\epsilon\|s_\epsilon'\|_\infty=O(\epsilon)$,
 the zeroth-order coefficient in \eqref{phase:tangent-equation} is at most $-\gamma_0/2$ on the two moving tails.
 Comparison from time $-T$ with the spatially constant subsolution
 $-M e^{-\gamma_0(t+T)/2}$, using the positive values on $x=X(t)\pm R$, 
 gives
 \[
  q(0,x)\geq-M e^{-\gamma_0T/2}\qquad(|x-X(0)|\geq R).
 \]
 Letting $T\to\infty$ proves $q\geq0$ everywhere.
 The strong maximum principle and positivity in the core give \eqref{phase:positive-tangent}.
 Integration in $X$ proves \eqref{phase:strict-order} for this family.

 For small periods, 
 set $e_L(\alpha)=Z_{L,\alpha}-U_\alpha$.
 Equations \eqref{sp:approx} and \eqref{sp:graphbound} imply
 \begin{equation}\label{phase:small-error}
  \sup_\alpha\|e_L(\alpha)\|_2+
  \operatorname{Lip}_{L^2}(e_L)\leq CL.
 \end{equation}
 Bounded entire graph solutions have a uniform local spatial $C^\beta$ bound, 
 for some $\beta\in(0,1)$, 
 by interior H\"older estimates for uniformly parabolic divergence-form equations with bounded drift and bounded forcing
 \cite[Section~1, Theorem~C]{Aronson1968}.
 These constants depend on ellipticity and coefficient bounds, 
 and are independent of $L$.
 Interpolating this bound with \eqref{phase:small-error} gives
 \[
  \sup_\alpha\|e_L(\alpha)\|_\infty
  \leq C L^{2\beta/(2\beta+1)}\longrightarrow0.
 \]
 Thus we can choose $R$ so that $U(z)\leq\delta_0/2$ for $z\geq R$ and $U(z)\geq1-\delta_0/2$ for $z\leq-R$.
 Decreasing $L$ so that $\sup_\alpha\|e_L(\alpha)\|_\infty\leq\delta_0/2$ then gives, uniformly in $\alpha$,
 \[
  Z_{L,\alpha}(x)\leq\delta_0\quad(x\geq\alpha+R),
  \qquad
  Z_{L,\alpha}(x)\geq1-\delta_0\quad(x\leq\alpha-R).
 \]

 Take two complete phase trajectories $\alpha_1(t)<\alpha_2(t)$, and put
 \[
  u_i(t,x)=Z_{L,\alpha_i(t)}(x),\qquad
  d(t)=\alpha_2(t)-\alpha_1(t)>0,\qquad D=u_2-u_1.
 \]
 Then
 \begin{equation}\label{phase:small-difference}
  \begin{aligned}
   D_t&=\partial_x(a(x/L)D_x)+b(x/L)D_x+c(t,x)D,\\
   c(t,x)&=\int_0^1 f_u(x/L,(1-\theta)u_1+\theta u_2)\,d\theta.
  \end{aligned}
 \end{equation}
 The bound $\|s_L\|_\infty\leq CL$ gives $|\alpha_i(t)-\alpha_i(t-1)|\leq CL$.
 Let $K_t(x,y)$ be the fundamental solution of \eqref{phase:small-difference} from $t-1$ to $t$.
 The two-sided Gaussian bounds \cite[Sections~7--8]{Aronson1968} give, 
 uniformly in $L,t$ and the two orbits,
 \begin{equation}\label{phase:unit-kernel}
  C^{-1}e^{-C|x-y|^2}\leq K_t(x,y)\leq Ce^{-|x-y|^2/C}.
 \end{equation}
 In particular, 
 $\sup_x\|K_t(x,\cdot)\|_2\leq C$.
 At time $t-1$, 
 write $A=\alpha_1(t-1)$ and $B=\alpha_2(t-1)$.
 By \eqref{phase:small-error},
 \[
  D(t-1,y)=U(y-B)-U(y-A)+E(y),\qquad
  \|E\|_2\leq CL\min\{B-A,1\}.
 \]
 The stationary-front equation and the strong maximum principle give $-U'>0$.
 The first term equals $\int_A^B(-U'(y-p))\,dp$ and is positive.
 For $x\in[\alpha_1(t)-R,\alpha_2(t)+R]$, restrict the kernel integral of the positive reference term to $|y-x|\leq1$.
 For small $L$, the unit-time phase displacements are at most $1$, 
 so $A-R-2\leq y\leq B+R+2$.
 If $B-A\leq1$, then $|y-p|\leq R+3$ for every $p\in[A,B]$.
 If $B-A\geq1$, choose a subinterval of length $1$ in $[A,B]$ containing the point of $[A,B]$ closest to $y$.
 Every $p$ in this subinterval also satisfies $|y-p|\leq R+3$.
 With $m_R:=\min_{|z|\leq R+3}(-U'(z))>0$, 
 both cases give
 \[
  \int_A^B(-U'(y-p))\,dp\geq m_R\min\{B-A,1\}.
 \]
 Equation~\eqref{phase:unit-kernel} therefore gives a lower bound $c_R\min\{B-A,1\}$ for the positive term,
 whereas the contribution of $E$ is bounded by $CL\min\{B-A,1\}$.
 Hence, for small $L$,
 \begin{equation}\label{phase:small-core-order}
  D(t,x)\geq c_R'\min\{d(t-1),1\}>0
  \quad\bigl(\alpha_1(t)-R\leq x\leq\alpha_2(t)+R\bigr).
 \end{equation}

 Since $0<u_i<1$, we have $|D|\leq1$ for all $t,x$.
 On $x\leq\alpha_1(t)-R$ both profiles are near $1$, and on $x\geq\alpha_2(t)+R$ both are near $0$.
 Thus $c\leq-\gamma_0$ on these two moving tails.
 On $[-T,0]$, 
 compare $D$ with the spatially constant subsolution $-e^{-\gamma_0(t+T)}$.
 At $t=-T$ it is at most $D$, and \eqref{phase:small-core-order} gives positive values of $D$ on the moving boundaries.
 Comparison therefore yields
 \[
  D(0,x)\geq-e^{-\gamma_0T}
  \qquad\bigl(x\leq\alpha_1(0)-R\ \text{or}\ x\geq\alpha_2(0)+R\bigr).
 \]
 Letting $T\to\infty$ and repeating the argument at any time gives $D\geq0$ everywhere.
 Positivity in the core and the strong maximum principle yield $D(0,x)>0$ for every $x$.
 Since the two initial phases were arbitrary, this proves \eqref{phase:strict-order} in the small-period regime.
\end{proof}

Consequently, 
every propagating orbit on either invariant family is strictly monotone in time. 
It is increasing when $v_L>0$ and decreasing when $v_L<0$ along its phase interval.
In particular, the connecting fronts in Theorem~\ref{phase:classification} are time-monotone connections between strictly ordered stationary fronts.

It remains to estimate the exact velocity $v_L$.
The frozen-front construction gives a first correction at large periods, 
while the stationary-front construction yields a first-order small-period velocity coefficient. 
These are the regime-specific estimates used in the common theorem.

\subsection{Large-period dynamics}

\subsubsection{Passage time in frozen coordinates}

In the large-period parametrization, 
$p=X=LY$ and $v_L(p)=s_\epsilon(Y)$ with $\epsilon=L^{-1}$. 
Changing variables in the common passage-time formula \eqref{intro:passage} gives
\begin{equation}\label{lp:harmonic}
 c_L=\left(\int_0^1\frac{dY}{s_\epsilon(Y)}\right)^{-1}
\end{equation}
whenever $s_\epsilon$ has no zero. 

\subsubsection{Consequences for every sign pattern of the frozen speed}

The exact criterion uses $s_\epsilon$, 
not just its leading term $c$.
We calculate the first correction, 
including the term caused by variation of the frozen profile. 
Taylor expansion of \eqref{lp:residual}, 
with its exponentially localized factors, 
gives
\begin{equation}\label{lp:R1}
	\begin{split}
		R_\epsilon(Y)&=\epsilon R_1(Y)+O(\epsilon^2),\\
		R_1(Y,z)&=a'(Y)(\phi_z+z\phi_{zz})+zb'(Y)\phi_z
		+z f_Y(Y,\phi)-c(Y)\phi_Y.
	\end{split}
\end{equation}
More precisely, 
the remainder satisfies
\begin{equation}\label{lp:R1-remainder}
 \left\|R_\epsilon-\epsilon R_1\right\|_{C^1(\T;H)}
 \leq C\epsilon^2.
\end{equation}
Here $f_Y$ is the partial derivative at fixed $u$. 
For the reaction contribution, 
the integral Taylor formula gives
\[
\begin{aligned}
&f(Y+\epsilon z,\phi)-f(Y,\phi)-\epsilon zf_Y(Y,\phi)\\
&\qquad=\epsilon^2z^2\int_0^1(1-\theta)
 f_{YY}(Y+\theta\epsilon z,\phi)\,d\theta.
\end{aligned}
\]
Lemma~\ref{lp:stable-equilibrium-derivative-decay} and the uniform frozen-tail
estimates bound these terms by $C\epsilon^2z^2e^{-\nu|z|}$ in the weighted region. 
The coefficient remainders satisfy the same type of estimate, 
which proves \eqref{lp:R1-remainder}.

\begin{proposition}[Large-period speed expansion]\label{lp:speed-expansion}
	With
	\begin{equation}\label{lp:d}
		d(Y)=-\ell_YR_1(Y),
	\end{equation}
	the reduced speed satisfies
	\begin{equation}\label{lp:speed-exp}
		s_\epsilon=c+\epsilon d+O(\epsilon^2)
		\quad\hbox{in }C^1(\T).
	\end{equation}
\end{proposition}

\begin{proof}
Since $\ell_Yh_\epsilon(Y)=0$, \eqref{lp:exact}--\eqref{lp:sigma} give $s_\epsilon-c=\sigma_\epsilon(Y,h_\epsilon(Y))$.
Set $h=h_\epsilon$ and introduce the numerator and denominator in \eqref{lp:sigma}:
	\[
	\begin{aligned}
	 \mathcal N_\epsilon(Y)&=
	 \ell_Y\bigl(A_{\epsilon,Y}h(Y)+R_\epsilon(Y)
	 +N_{\epsilon,Y}(h(Y))\bigr)
	 +\epsilon c(Y)\dot\ell_Yh(Y),\\
	 \mathcal D_\epsilon(Y)&=
	 1-\epsilon\ell_Y\phi_Y(Y)+\ell_Yh_z(Y)
	 +\epsilon\dot\ell_Yh(Y).
	\end{aligned}
	\]
	Estimates \eqref{lp:leakage} and \eqref{lp:graph-C1} give
	\[
	 \|\ell_YA_{\epsilon,Y}h(Y)\|_{C^1(\T)}
	 \leq C\epsilon^2.
	\]
	By \eqref{lp:nonlinear-C1-est} and the $C^1$ bound on $h$,
	we have
	\[
	 \|N_{\epsilon,\cdot}(h(\cdot))\|_{C^1(\T;H)}
	 +\|\epsilon c\dot\ell_Yh(Y)\|_{C^1(\T)}
	 \leq C\epsilon^2.
	\]
	By \eqref{lp:R1-remainder},
	\[
	 \|\ell_YR_\epsilon-\epsilon\ell_YR_1\|_{C^1(\T)}
	 \leq C\epsilon^2.
	\]
	Consequently,
	\begin{equation}\label{lp:speed-numerator-estimate}
	 \|\mathcal N_\epsilon-\epsilon\ell_YR_1\|_{C^1(\T)}
	 \leq C\epsilon^2.
	\end{equation}
	The same $C^1$ graph bound gives
	\begin{equation}\label{lp:speed-denominator-estimate}
	 \|\mathcal D_\epsilon-1\|_{C^1(\T)}\leq C\epsilon.
	\end{equation}
	For small $\epsilon$, $\mathcal D_\epsilon$ is bounded away from zero and
	\[
	 \|\mathcal D_\epsilon^{-1}-1\|_{C^1(\T)}\leq C\epsilon.
	\]
	Since $\sigma_\epsilon=-\mathcal N_\epsilon/\mathcal D_\epsilon$,
	\eqref{lp:speed-numerator-estimate} and \eqref{lp:speed-denominator-estimate} imply
	\[
	 \left\|\sigma_\epsilon+\epsilon\ell_YR_1\right\|_{C^1(\T)}
	 \leq C\epsilon^2.
	\]
	Since $d=-\ell_YR_1$, adding $c(Y)$ proves \eqref{lp:speed-exp} in $C^1(\T)$.
\end{proof}

\begin{corollary}\label{lp:all-cases}
	For all sufficiently large $L$ the following conclusions hold.
	\begin{enumerate}
		\item If $c$ is strictly positive or strictly negative on $\T$,
		there is a pulsating front with the same sign of speed, 
		and
		\[
		c_L=c_*+\epsilon c_*^2\int_0^1\frac{d(Y)}{c(Y)^2}\,dY
		+O(\epsilon^2),\qquad
		c_*=\left(\int_0^1\frac{dY}{c(Y)}\right)^{-1}.
		\]
		\item If $c$ takes both positive and negative values, 
		$s_\epsilon$ also takes both signs. 
		In particular the manifold contains stationary fronts and connecting orbits, 
		and no orbit traverses all spatial phases.
		\item Suppose $c$ has a zero but does not change sign. 
		If $c\geq0$ and $c\not\equiv0$, 
		positivity of $d$ on $\mathcal Z=\{Y:c(Y)=0\}$ gives a right-moving pulsating front,
		whereas $d(Y_0)<0$ at some $Y_0\in\mathcal Z$ gives stationary fronts. 
		Reverse the signs when $c\leq0$ and $c\not\equiv0$.
		If $c\equiv0$, 
		a strict constant sign of $d$ gives a pulsating front with
		\[
		c_L=\epsilon\left(\int_0^1\frac{dY}{d(Y)}\right)^{-1}
		+O(\epsilon^2).
		\]
		A sign change of $d$ then gives stationary fronts. 
		In every propagating case in this item, 
		$c_L\to0$.
	\end{enumerate}
	When these sign tests are inconclusive,
	Theorem~\ref{phase:classification} still classifies the dynamics on the invariant family in terms of the zeros and sign of the exact phase velocity $s_\epsilon$.
\end{corollary}

\begin{proof}
	The first assertion follows by uniformly expanding the reciprocal of $c+\epsilon d+O(\epsilon^2)$. 
	The second follows by evaluating at one positive and one negative point of $c$ and applying continuity.
	For the third item with $c\geq0$ and $c\not\equiv0$, 
	compactness makes $d$ uniformly positive on a neighborhood of $\mathcal Z$ when it is positive on $\mathcal Z$.
	There $c+\epsilon d+O(\epsilon^2)>0$, 
	while on the complement $c$ is bounded below by a positive number. 
	A negative value of $d$ at a zero gives a negative value of $s_\epsilon$ there, 
	while a fixed point with $c>0$ gives a positive value. 
	If propagation persists,
	the Lipschitz bound near a zero $Y_0$ gives $c(Y)\leq C|Y-Y_0|$, 
	so $\int c(Y)^{-1}\,dY=+\infty$.
	Fatou's lemma applied to the positive reciprocals of $s_\epsilon$ then shows that \eqref{lp:harmonic} tends to zero.
	The negative case is identical after changing signs. 
	When $c\equiv0$,
	the claims follow by factoring $\epsilon$ from \eqref{lp:speed-exp}.
	The displayed speed formula also gives $c_L\to0$.
\end{proof}

When the frozen speed touches zero without changing sign, propagation is decided by the correction term. 
The following explicit pair of examples shows why this correction is needed.

\begin{example}[Same frozen speed, different large-period dynamics]\label{lp:touching-example}
	Set
	\[
	 c(Y)=1-\cos(2\pi Y),
	 \qquad b(Y)=-c(Y).
	\]
	First take
	\begin{equation*}
	 u_t=\partial_x\!\left(a\!\left(\frac{x}{L}\right)u_x\right)
	 +b\!\left(\frac{x}{L}\right)u_x
	 +f\!\left(\frac{x}{L},u\right),
	\end{equation*}
	where 
	\[
	a(Y)=2+\sin(2\pi Y),
	\qquad f(Y,u)=a(Y)u(1-u)(2u-1).
	\]
	A direct computation yields that $d(Y)=-a'(Y)=-2\pi\cos(2\pi Y)$.  
	Hence Proposition~\ref{lp:speed-expansion} gives
	\[
	 s_{\epsilon}(Y)
	 =1-\cos(2\pi Y)-2\pi\epsilon\cos(2\pi Y)+O(\epsilon^2)
	 \quad\text{in }C^1(\mathbb T).
	\]
	Since $s_{\epsilon}(0)<0$ and $s_{\epsilon}(1/2)>0$ for small $\epsilon$,
	the large-period reduced speed changes sign.  
	Thus this equation has stationary fronts and connecting transition fronts in the invariant family,
	but no pulsating front traversing all spatial phases.

	Now replace $a,f$ by
	\[
	 \widetilde a(Y)=2-\sin(2\pi Y),
	 \qquad
	 \widetilde f(Y,u)=\widetilde a(Y)u(1-u)(2u-1).
	\]
	The frozen equation again has the same profile $\phi$ and speed $c(Y)$.
	Its first-order correction is
	$\widetilde d(Y)=-\widetilde a'(Y)=2\pi\cos(2\pi Y)$, so
	\[
	 \widetilde s_{\epsilon}(Y)
	 =1-\cos(2\pi Y)+2\pi\epsilon\cos(2\pi Y)+O(\epsilon^2)
	 \quad\text{in }C^1(\mathbb T).
	\]
	Because $c\geq0$, its zero set is $\{0\}$ and $\widetilde d(0)=2\pi>0$,
	Corollary~\ref{lp:all-cases} implies that $\widetilde s_{\epsilon}>0$ for all sufficiently small $\epsilon$. 
	Thus the modified equation has a right-moving pulsating front.
	Its speed is
	\[
	 \widetilde c_L=\left(\int_0^1\frac{dY}{\widetilde s_{\epsilon}(Y)}\right)^{-1},
	 \qquad \epsilon=L^{-1},
	\]
	and the quadratic-touching asymptotics give
	\[
	 \widetilde c_L\sim 2\sqrt{\pi}\,L^{-1/2}\qquad(L\to\infty).
	\]
	Replacing $(a,f)$ by $(\widetilde a,\widetilde f)$ leaves the frozen profile and speed unchanged,
	but reverses the first-order correction at the touching phase and changes the large-period conclusion from stationary/connecting dynamics to a propagating pulsating front.
\end{example}

\subsubsection{Quantitative propagation near degenerate zero speeds}\label{bn:section}
We use the expansion in Proposition~\ref{lp:speed-expansion} to quantify
propagation near finite-order zeros of the frozen speed.
All statements in this subsection refer to the invariant family constructed in Theorem~\ref{lp:main}.

\begin{theorem}[Propagation through finite-order zero-speed phases]\label{bn:running}
	Suppose $c\geq0$ and its zero set on $\T$ consists of finitely many points $Y_1,\ldots,Y_J$, 
	with $J\geq1$. 
	Assume that, in a local coordinate $r$ at $Y_j$,
	\[
	c(Y_j+r)=A_j r^{2m_j}+o(r^{2m_j}),\qquad
	A_j>0,\quad m_j\in\N,\quad m_j\geq1,\qquad d_j:=d(Y_j)>0.
	\]
	Let $m_*:=\max_jm_j$ and set
	\[
	I_m:=\int_{\R}\frac{dt}{1+t^{2m}}
	=\frac{\pi}{m}\csc\frac{\pi}{2m},\qquad
	K_*:=I_{m_*}\sum_{j:m_j=m_*}
	A_j^{-1/(2m_*)}d_j^{1/(2m_*)-1}.
	\]
	Then, 
	for all sufficiently large $L$ there is a propagating front in the invariant family, 
	and its speed satisfies
	\begin{equation}\label{bn:speed}
		c_L\sim K_*^{-1} L^{-(1-1/(2m_*))}.
	\end{equation}
	In particular, if $J=1$ and $m_1=1$, then, writing $A=A_1$ and $d_0=d_1$,
	\begin{equation}\label{bn:quadratic}
		c_L\sim\frac{\sqrt{A d_0}}{\pi\sqrt L}.
	\end{equation}
\end{theorem}

\begin{proof}
	Write $s_\epsilon=c+\epsilon d+\epsilon^2r_\epsilon$, 
	where $\sup_{0<\epsilon\leq\epsilon_0}\|r_\epsilon\|_{C^1}<\infty$.
	Choose disjoint coordinate neighborhoods $|r|<\delta$ of the zeros,
	small enough that, 
	in the $j$th neighborhood,
	\[
	\tfrac12A_jr^{2m_j}\leq c(Y_j+r)\leq\tfrac32A_jr^{2m_j},
	\qquad \tfrac12d_j\leq d(Y_j+r)\leq\tfrac32d_j.
	\]
	After reducing $\epsilon_0$,
	\begin{equation}\label{bn:lower}
		s_\epsilon(Y_j+r)\geq\tfrac12A_jr^{2m_j}
		+\tfrac14\epsilon d_j.
	\end{equation}
	On the compact complement $c$ is bounded below by a positive constant,
	so $s_\epsilon>0$ everywhere. Theorem~\ref{phase:classification} applies.
	
	For one neighborhood put $m=m_j$, $A=A_j$, $d_0=d_j$ and $\lambda_\epsilon=(\epsilon d_0/A)^{1/(2m)}$. 
	Substitution yields
	\[
	\int_{-\delta}^{\delta}\frac{dr}{s_\epsilon(Y_j+r)}
	=\frac{\lambda_\epsilon}{\epsilon d_0}
	\int_{\R}
	\frac{\epsilon d_0\,\mathbf1_{\{|t|<\delta/\lambda_\epsilon\}}}
	{s_\epsilon(Y_j+\lambda_\epsilon t)}\,dt.
	\]
	For fixed $t$, 
	the integrand converges to $(1+t^{2m})^{-1}$.
	Estimate~\eqref{bn:lower} bounds it by $4/(1+t^{2m})$, 
	which is integrable.
	 Dominated convergence therefore gives
	\[
	\int_{-\delta}^{\delta}\frac{dr}{s_\epsilon(Y_j+r)}
	\sim I_{m_j}A_j^{-1/(2m_j)}d_j^{1/(2m_j)-1}
	\epsilon^{1/(2m_j)-1}.
	\]
	The integral over the complement is $O(1)$.
	When these estimates are summed, 
	exactly the indices with largest $m_j$ contribute to the leading term. 
	Inverting the exact formula~\eqref{lp:harmonic} proves \eqref{bn:speed}. 
	Finally, 
	setting $v=t^{2m}$ on the positive half-line and using the Euler beta integral gives the stated value of $I_m$.
\end{proof}

\begin{remark}
	The passage-time integral and square-root speed law near the propagation threshold are established mechanisms.
	See, for example, \cite{ST2017}. 
	An infinite zero set, 
	a zero of infinite order, 
	or $d=0$ at a frozen zero requires information beyond this finite-order argument.
\end{remark}

\subsection{Small-period dynamics}\label{sp:dynamics}

We now estimate the variation of $s_L$ over one period and compute its first-order coefficient.

\subsubsection{Phase velocity and its mean}

\begin{proposition}[Small-period phase velocity and its first-order mean]
\label{sp:phase-average}\label{sp:first-order-speed}
	Let $s_L$ be the velocity in Theorem~\ref{sp:main} and set
$\bar s_L=L^{-1}\int_0^L s_L(\alpha)\,d\alpha$.
	Let 
	\[
	\langle h\rangle :=\int_0^1h(y)\,dy,
	\] 
	and let $\zeta$ and $\psi=\psi_0$ be defined by \eqref{sp:zeta} and Lemma~\ref{sp:limit-spectrum} respectively.
	Then the constant
	\begin{equation}\label{sp:kappa}
		\kappa=\int_\R\psi(z)
		\big\langle\zeta(y)f(y,U(z))\big\rangle\,dz
		-\frac{a_hb_h}{2}\left\langle\frac{\zeta}{a}\right\rangle
	\end{equation}
	satisfies, 
	uniformly in $\alpha$,
	\begin{equation}\label{sp:first-order-expansion}
		s_L(\alpha)=\kappa L+O(L^2),\qquad
		\bar s_L=\kappa L+O(L^2).
	\end{equation}
	Moreover,
	\begin{equation}\label{sp:phase-oscillation}
		\|s_L-\bar s_L\|_\infty\leq C_0L^2.
	\end{equation}
	If $\kappa\ne0$, 
	the invariant family consists of the time slices of a single propagating pulsating front for all sufficiently small $L$,
	and its signed speed satisfies $c_L=\kappa L+O(L^2)$.
\end{proposition}
\begin{proof}
	First we improve the bound on the invariant graph from $\mathsf X$ to $V$.
	Along a complete graph orbit write $w(t)=g_L(\alpha(t))$.
	By \eqref{sp:graphbound} and $|s_L|\leq CL$, 
	this is an $\mathsf X$-valued Lipschitz function with $\|w_t\|_2\leq CL^2$ at almost every time.
	At such times the parabolic equation gives, 
	in $V'$,
	\begin{equation}\label{sp:graph-regularity-equation}
		-A_{L,\alpha}g_L(\alpha)+s_L(\alpha)q_{L,\alpha}
		=R_{L,\alpha}+N_{L,\alpha}(g_L(\alpha))-w_t,
		\qquad \ell_{L,\alpha}g_L(\alpha)=0.
	\end{equation}
	The right side is $O(L)$ in $V'$ by \eqref{sp:residual},
	\eqref{sp:nonlinear}, 
	and \eqref{sp:graphbound}.
	Apply Lemma~\ref{sp:resolvent} at $\lambda=0$ with $\mu=s_L(\alpha)$ to obtain $\|g_L(\alpha)\|_V\leq CL$.
	This holds at every phase: at a zero of $s_L$ the graph orbit is stationary and $w_t=0$; 
	elsewhere, 
	approach the phase along times at which $w_t$ exists and use $\mathsf X$ continuity and weak lower semicontinuity of the $V$ norm.

	The flux identity \eqref{sp:adjointidentity} with $\varphi=\psi_\alpha$ gives
	\[
	 A_{L,\alpha}^*\psi_{L,\alpha}
	 =L\partial_x\bigl[a\chi\psi_\alpha''
	 +(a\zeta-b\chi)\psi_\alpha'-b\zeta\psi_\alpha\bigr]
	 +f_u(x/L,V_{L,\alpha})\psi_{L,\alpha}
	 -\bar f'(U_\alpha)\psi_\alpha,
	\]
	where the periodic factors are evaluated at $x/L$.
	The bracket is uniformly bounded in $\mathsf X$.
	Estimate~\eqref{sp:multiplier}, 
	applied to $\psi_\alpha$,
	and $\|\psi_{L,\alpha}-\psi_\alpha\|_2\leq CL$ therefore give $\|A_{L,\alpha}^*\psi_{L,\alpha}\|_{V'}\leq CL$.
	Consequently,
	$|\ell_{L,\alpha}A_{L,\alpha}g_L(\alpha)|\leq CL^2$.
	In \eqref{sp:coordinates}, the nonlinear projection is $O(L^2)$ and the denominator is $1+O(L)$.
	Hence
	\begin{equation}\label{sp:velocity-projection}
	 s_L(\alpha)=\ell_{L,\alpha}R_{L,\alpha}+O(L^2).
	\end{equation}

	It remains to average the explicit residual. 
	Put $z=x-\alpha$ and $y=x/L$, 
	and use primes on $U,\psi$ for $z$ derivatives and on periodic coefficients for $y$ derivatives. 
	Taylor expansion of $R_{L,\alpha}=F_L(V_{L,\alpha})$ gives
	\[
	 R_{L,\alpha}=R_0(y,z)+LR_1(y,z)+O_{\mathsf X}(L^2),
	\]
	where
	\begin{align*}
	 R_0&=(a\chi)'U''+(a_hb/a-b_h)U'+f(y,U)-\bar f(U),\\
	 R_1&=a\chi U'''+b\chi U''+\chi U'f_u(y,U).
	\end{align*}
	Each periodic factor in $R_0$ has zero mean. 
	Twice integrating by parts with periodic primitives, 
	using the exponential decay of $U',\psi$ and their derivatives,
	yields $\int_\R\psi(z)R_0((z+\alpha)/L,z)\,dz=O(L^2)$ uniformly in $\alpha$.
	A single periodic averaging of the terms already multiplied by $L$ then gives
	\begin{equation}\label{sp:kappa-unsimplified}
	 \ell_{L,\alpha}R_{L,\alpha}
	 =L\int_\R\Big\langle
	 \psi R_1+(\chi\psi'+\zeta\psi)R_0\Big\rangle\,dz+O(L^2).
	\end{equation}
	Periodic integration by parts gives $\langle\chi(a\chi)'\rangle=\langle a\chi\rangle$ 
	and $\langle\zeta(a\chi)'\rangle=-\langle b\chi\rangle$.
	The corresponding $\langle a\chi\rangle$, $\langle b\chi\rangle$,
	and $\langle\chi f(y,U)\rangle$ terms in \eqref{sp:kappa-unsimplified} are total $z$ derivatives or cancel.
	Its remaining coefficient is
	\[
	 a_h\left\langle\frac{\chi b}{a}\right\rangle
	 \int_\R\psi'U'\,dz
	 +a_h\left\langle\frac{\zeta b}{a}\right\rangle
	 \int_\R\psi U'\,dz
	 +\int_\R\psi\langle\zeta f(y,U)\rangle\,dz.
	\]
	The cell equations and their zero-mean normalizations imply
	\[
	 \left\langle\frac{\chi}{a}\right\rangle=0,
	 \quad\left\langle\frac{\chi b}{a}\right\rangle
	 =-a_h\left\langle\frac{\zeta}{a}\right\rangle,
	 \quad\left\langle\frac{\zeta b}{a}\right\rangle
	 =b_h\left\langle\frac{\zeta}{a}\right\rangle.
	\]
	Finally, 
	$\int_\R\psi U'=-1$ and the adjoint formula in Lemma~\ref{sp:limit-spectrum} gives $\int_\R\psi'U'=-b_h/(2a_h)$.
	Thus \eqref{sp:kappa-unsimplified} reduces to $\ell_{L,\alpha}R_{L,\alpha}=\kappa L+O(L^2)$ with \eqref{sp:kappa}, 
	and \eqref{sp:velocity-projection} proves the expansion for $s_L$ in \eqref{sp:first-order-expansion}.
	Averaging gives the expansion for $\bar s_L$ and, after increasing $C_0$, \eqref{sp:phase-oscillation}.
	If $\kappa\ne0$, $s_L$ has its sign for small $L$.
	The exact passage-time formula \eqref{intro:passage} then yields
	$c_L=\kappa L+O(L^2)$.
\end{proof}


\section{Stability and global attraction}\label{sec:stability-uniqueness}

Section~\ref{sec:existence} established local exponential attraction and local completeness of both invariant front families.
The normal errors decay in weighted $H^1$ for large periods and in $L^2$ for small periods,
with constants uniform in the respective regimes.
Stability of an individual stationary front also depends on the phase dynamics:
if the phase velocity changes from negative to positive at its phase,
nearby profiles move away along the family.

When $s_L\equiv0$, the small-period family consists entirely of stationary fronts,
and nearby solutions converge exponentially to individual members.
Indeed, let $T>0$ and $q<1$ be as in \eqref{sp:normalcontraction},
and put $e_0=\|w(0)-g_L(\alpha(0))\|_2$.
Comparison with the stationary graph orbit at phase $\alpha(nT)$,
using \eqref{sp:phaseincrement} and normal contraction, gives
$|\alpha((n+1)T)-\alpha(nT)|\leq Cq^ne_0$.
Summation yields a limit $\alpha_\infty$ with
$|\alpha(nT)-\alpha_\infty|\leq Cq^ne_0$.
The intermediate-time estimates and the Lipschitz graph parametrization then give
$\|u(t)-Z_{L,\alpha_\infty}\|_2\leq Ce^{-\nu t}e_0$,
where $\nu=-T^{-1}\log q>0$.
Since $|\alpha_\infty-\alpha(0)|\leq Ce_0$, for every fixed $\alpha_*\in\R$ these estimates give
\[
 \sup_{t\geq0}\|u(t)-Z_{L,\alpha_*}\|_2
 \leq C\bigl(e_0+|\alpha(0)-\alpha_*|\bigr).
\]
By the inverse-chart continuity in Lemma~\ref{sp:chart-unique} and the Lipschitz continuity of $g_L$,
the right-hand side tends to zero as $\|u(0)-Z_{L,\alpha_*}\|_2\to0$.
Thus every $Z_{L,\alpha_*}$ is Lyapunov stable.

The remainder of this section establishes global exponential orbital stability of propagating pulsating fronts.
It then constructs initial data with limits $1$ and $0$ whose stationary solutions remain a positive uniform distance from the large-period invariant family.

\subsection{Global exponential orbital stability of propagating fronts}

In the following,
we denote $u(t,\cdot;g)$ the solution of \eqref{intro:pde} with initial datum $g$.

\begin{theorem}[Global exponential orbital stability of a propagating front]
\label{comp:global-running-attraction}
 Under the common assumptions of Subsection~\ref{intro:assumptions},
 fix any $L>0$ and suppose \eqref{intro:pde} has a pulsating front $U_L(t,x)$ with speed $c_L\ne 0$. 
 Then there exists $\delta\in(0,\delta_0]$ and $\mu_L>0$ such that every $g\in L^\infty(\R,[0,1])$ satisfying
 \[
  \liminf_{x\to-\infty}g(x)>1-\delta,\qquad
  \limsup_{x\to+\infty}g(x)<\delta
 \]
 has a solution obeying
 \begin{equation}\label{comp:global-attraction}
  \|u(t,\cdot;g)-U_L(t+\tau_g,\cdot)\|_\infty
  \leq C_g e^{-\mu_Lt},\qquad t\geq0,
 \end{equation}
 for some $\tau_g\in\R$ and $C_g>0$.
 Here, 
 the rate $\mu_L$ is independent of $g$ at this fixed period.
\end{theorem}

\begin{proof}
 Reflection $u(t,x)\mapsto1-u(t,-x)$ reduces $c_L<0$ to the case $c_L>0$ and preserves the common assumptions. 
 Write $U=U_L$ and $c=c_L>0$. The monotonicity theorem \cite[Theorem~1.10]{BH2012} applies to this invasion front:
 the constant equilibria $0$ and $1$ are uniformly stable and \eqref{intro:extended-stable-states} 
 supplies the required monotonicity of the reaction near them. 
 Hence $U_t>0$. 
 Fix $0<\delta_*<\delta_0/4$ and define the transition region by
 \[
  \mathcal C_{\delta_*}
   =\{(t,x)\in\R^2:\delta_*\leq U(t,x)\leq1-\delta_*\}.
 \]
 Writing $U(t,x)=\Phi(x-ct,x/L)$, 
 the uniform limits of $\Phi$ as its first argument tends to $\pm\infty$ place $\mathcal C_{\delta_*}$ within $|x-ct|\leq R_{\delta_*}$.
 Since $(\xi,y)\mapsto-c\Phi_\xi(\xi,y)$ is positive and continuous,
 its restriction to the compact set $[-R_{\delta_*},R_{\delta_*}]\times\T$ has a positive minimum. 
 In particular,
 \begin{equation}\label{comp:core-monotonicity}
  \beta:=\inf_{(t,x)\in\mathcal C_{\delta_*}}U_t(t,x)>0.
 \end{equation}
 Use $\delta=\delta_*/2$ in the theorem. For a given $g$, 
 choose $q_0\in(\delta,\delta_*)$ larger than both of its tail deficits.
 Put $\mu=\gamma_0/2$, $q(t)=q_0e^{-\mu t}$, 
 and choose $K>(M+\mu)/\beta$, 
 where
 $M=\sup_{y\in\T,\,-q_0\leq s\leq1+q_0}|f_u(y,s)|$.
 For constants $\tau_-<\tau_+$ to be fixed, 
 set
 \[
  \sigma_\pm(t)=\tau_\pm\pm K\int_0^tq(s)\,ds,\qquad
  U^\pm(t,x)=U(t+\sigma_\pm(t),x)\pm q(t).
 \]
 With $\mathcal Pw=w_t-\partial_x(a(x/L)w_x)-b(x/L)w_x -f(x/L,w)$ and $V_\pm=U(t+\sigma_\pm(t),x)$, 
 direct substitution gives
 \[
  \mathcal P U^\pm
   =\pm Kq\,U_t(t+\sigma_\pm(t),x)\mp\mu q
     +f(x/L,V_\pm)-f(x/L,V_\pm\pm q).
 \]
 Wherever $\delta_*\leq V_\pm\leq1-\delta_*$,
 \eqref{comp:core-monotonicity} and the choice of $K$ give $\mathcal P U^-\leq0\leq\mathcal P U^+$.
 If $V_\pm<\delta_*$ or $V_\pm>1-\delta_*$, 
 then $q<\delta_*$ and $2\delta_*<\delta_0$ keep the intervals between $V_\pm$ and $V_\pm\pm q$ in a region where \eqref{intro:extended-stable-states} applies.
 Hence,
 in the respective tail regions,
 \[
  f(x/L,V_+)-f(x/L,V_++q)\geq\gamma_0q,\qquad
  f(x/L,V_-)-f(x/L,V_--q)\leq-\gamma_0q.
 \]
 Since $U_t>0$ and $\mu<\gamma_0$, 
 the preceding formula gives $\mathcal P U^+\geq(\gamma_0-\mu)q\geq0$ and $\mathcal P U^-\leq-(\gamma_0-\mu)q\leq0$.
 The added terms $\pm q(t)$ have no spatial derivatives, 
 so neither the diffusion nor the periodic drift changes this calculation.

 On the far left and right, 
 the tail condition on $g$ and the choice of $q_0$ give the desired initial ordering. 
 Moving the front sufficiently far in time gives the same ordering on the remaining bounded interval. 
 Thus we may choose $\tau_-\ll0\ll\tau_+$ such that $U^-(0,\cdot)\leq g\leq U^+(0,\cdot)$.
 Comparison then gives
 \begin{equation}\label{comp:global-sandwich}
  U(t+\sigma_-(t),x)-q(t)\leq u(t,x;g)
  \leq U(t+\sigma_+(t),x)+q(t),\qquad t\geq0.
 \end{equation}
 Both shifts have finite limits. 
 This yields uniformly controlled tails and a bounded transition region in the moving coordinate $\xi=x-ct$, 
 while parabolic regularity makes the orbit precompact there in the uniform topology.

 For phase selection, 
 set $v(t,\xi)=u(t,\xi+ct)$,
 $T=L/c$, 
 and $V^\tau(t,\xi)=U(t+\tau,\xi+ct)$.
 The equation for $v$ is $T$-periodic.  
 Let $\mathsf X_\infty=BUC(\R)$, equipped with the supremum norm, 
 and let $\mathscr P_T:\mathsf X_\infty\to\mathsf X_\infty$ denote its time-$T$ map.
 Although the initial datum $g$ is only assumed to be bounded measurable,
 parabolic smoothing gives $v(t,\cdot)\in\mathsf X_\infty$ for every $t>0$.
 The ordered profiles $V^\tau(0,\cdot)$ are fixed points of $\mathscr P_T$.
 Repeating the barrier construction for initial data near any $V^\tau(0,\cdot)$, 
 with arbitrarily small $q_0>0$, 
 shows that these fixed points are Lyapunov stable. 
 Moreover, 
 \eqref{comp:global-sandwich} and parabolic smoothing give a precompact discrete orbit whose omega limit lies between two such profiles. 
 The order-interval argument in \cite[proof of Proposition~4.4]{DHZ} then makes the omega limit a single profile $V^{\tau_g}(0,\cdot)$: 
 strong comparison and the cutoff barriers rule out a persistent gap between its limiting phases.
 The additional drift $b((\xi+ct)/L)v_\xi$ enters the barrier calculation through a bounded multiple of the cutoff derivative, 
 which is absorbed by the cutoff speed. 
 The argument uses the uniform stability of $0$ and $1$ and the positive lower bound for $U_t$ in \eqref{comp:core-monotonicity}.
 Continuity of the semiflow over one period therefore gives
 \[
  \|u(t,\cdot;g)-U(t+\tau_g,\cdot)\|_\infty\longrightarrow0.
 \]

 We finally recall the exponential step of \cite[Section~4.2]{DHZ}, 
 including its drift modification. 
 The derivative
 $D\mathscr P_T(V^\tau(0,\cdot))\in\mathcal L(\mathsf X_\infty,\mathsf X_\infty)$ has the positive neutral eigenfunction $\partial_\tau V^\tau(0,\cdot)$ with multiplier $1$.
 Its linear equation is
 \[
  z_t=\partial_\xi\!\bigl(a((\xi+ct)/L)z_\xi\bigr)
    +\bigl(b((\xi+ct)/L)+c\bigr)z_\xi
    +f_u((\xi+ct)/L,V^\tau)z.
 \]
 The positive neutral eigenfunction and parabolic comparison give an algebraically simple multiplier $1$
 and put every other eigenvalue inside the unit disk, 
 by the argument of \cite[Lemma~4.5]{DHZ}. 
 The negativity of $f_u$ near $0$ and $1$ places the essential spectral radius strictly below $1$: 
 compare with the stable tail evolution and use compactness of the core perturbation as in \cite[Lemma~4.6]{DHZ}. 
 The bounded drift adds a first-order term to these linear equations; 
 its contribution to the tail barriers is bounded by a constant times their spatial derivative. 
 Local exponential stability with asymptotic phase for this fixed-point curve now applies.
 The covariance identifies phases modulo $T$ by spatial translation.
 Compactness of this phase interval gives a common local exponential rate $\mu_L>0$.
 The uniform convergence just proved brings the orbit into that neighborhood, 
 giving \eqref{comp:global-attraction} after enlarging $C_g$ over the initial finite-time interval.
\end{proof}

\subsection{Failure of global attraction to the invariant family}

The global attraction result above assumes the existence of a propagating front.
Under the hypotheses of the invariant-family construction alone,
even smooth initial data $0<g<1$ with $g(-\infty)=1$ and $g(+\infty)=0$ need not approach that family.
We construct such data as a stationary solution with three separated transitions.

\begin{proposition}[Failure of global attraction for front-like data]
\label{comp:global-counterexample}
 There are smooth coefficients satisfying the common hypotheses and both regime-specific assumptions in Subsection~\ref{intro:assumptions} such that, 
 for every sufficiently large $L$,
 \eqref{intro:pde} has a stationary solution $0<u<1$ with $u(-\infty)=1$ and $u(+\infty)=0$ satisfying
 \begin{equation}\label{comp:no-global-attraction}
 	\inf_{X\in\R}\|u-Z_{1/L,X}\|_\infty>0.
 \end{equation}
 Consequently, 
 the solution with initial datum $g=u$ is stationary and therefore does not approach $\Mcal_{1/L}$ in the uniform norm.
\end{proposition}

\begin{proof}
 Take the periodic variational Allen--Cahn subclass
 \begin{equation}\label{comp:multitransition-example}
  a\equiv1,\qquad b\equiv0,\qquad
  f(y,u)=-A(y)G'(u),\qquad G(u)=\sin^2(\pi u),
 \end{equation}
 with $A(y)=1+\rho\cos(2\pi y)$ and fixed small $\rho>0$.
 This has the stated bistable properties on $[0,1]$ and a decreasing stationary front for its spatial average. 
 As in Subsection~\ref{intro:assumptions}, 
 extend $f$ outside $[0,1]$ with stable exterior tails. 
 The solutions below remain in $(0,1)$.
 To apply the multi-transition existence theorem \cite[Theorem~3.27]{MR2016}, 
 use the original periodic-in-$u$ potential, 
 which agrees with this extension on $[0,1]$.
 Its action in the original spatial coordinate and the rescaled action used by that theorem are
 \begin{equation}\label{comp:multitransition-energy}
  \begin{aligned}
   \mathcal E_L(h)
    &=\int_\R\left(\tfrac12|h_x|^2+A(x/L)G(h)\right)\,dx,\\
   \mathcal J_L(v)
    &=\int_\R\left(\tfrac12|v_y|^2+L^2A(y)G(v)\right)\,dy,
       \qquad v(y)=h(Ly).
  \end{aligned}
 \end{equation}
Thus $\mathcal J_L(v)=L\mathcal E_L(h)$. 
The periodic ground states are exactly the integer constant functions,
and their energy is zero.
Thus $\mathcal J_L$ is the renormalized action of \cite{MR2016} on finite-action connections. 
From this point we use the notation of \cite{MR2016}. 
$M_0$ denotes the set of periodic zero-energy ground states. 
We suppress the fixed parameter $L$ and
write the levels directly in terms of the rescaled action $\mathcal J_L$.
For $v,w\in M_0$, 
let $\Gamma_1(v,w)$ be the class of finite-action connections whose traces converge to $v$ as $y\to-\infty$ and to $w$ as $y\to+\infty$,
and set
\[
 C_1(v,w):=\inf_{q\in\Gamma_1(v,w)}\mathcal J_L(q).
\]
The corresponding free-endpoint class allows the terminal ground state to vary over $M_0\setminus\{v\}$. 
Its level is
\[
 c_1(v):=\inf_{w\in M_0\setminus\{v\}}C_1(v,w).
\]
Thus $C_1(1,0)$ is the least action among connections with prescribed limits $1$ and $0$, 
whereas $c_1(1)$ is the least one-transition action leaving $1$ when the terminal ground state is not prescribed.  
In the present example we verify that these levels coincide.  
These are the definitions in \cite[(2.6), (2.9)]{MR2016}. 
The infimum need not in general be attained by a single transition and may instead split into a heteroclinic chain.

A basic action minimizer means a minimizer of $\mathcal J_L$ among finite-action connections with $v(-\infty)=1$, $v(+\infty)=0$,
or with these limits reversed.
 The large-period calculation gives $c(Y)\equiv0$ and
 \[
  d(Y)=-A'(Y)
  \frac{\int_\R G(\phi(Y,z))\,dz}
       {\int_\R\phi_z(Y,z)^2\,dz}
       =-\frac{A'(Y)}{2A(Y)},
 \]
 by \eqref{lp:R1}--\eqref{lp:d} and $\phi_z^2=2A(Y)G(\phi)$. 
 Thus $s_{1/L}$ has exactly two isolated zeros per cell for large $L$, 
 near the minimum and maximum of $A$.

 We check the cited theorem's hypotheses. 
 The rescaled potential $L^2A(y)G(u)$ is smooth and periodic in $y,u$, 
 even in $y$, 
 and has exactly the integer constants as periodic zero-energy minimizers. 
 Thus $v=1$ and $w=0$ belong to $M_0$, 
 and the potential satisfies (F1), (F2), and (N0) of \cite{MR2016}.
 Clipping a connection from $1$ to $0$ to $[0,1]$ preserves its limits at spatial infinity 
 and cannot increase either part of the nonnegative action. 
 Hence every basic action minimizer takes values in $[0,1]$. 
Put $A_{\min}=1-\rho$, $A_{\max}=1+\rho$, and
$E_0=\int_0^1\sqrt{2G(u)}\,du$. 
 For any descending connection $h$ with values in $[0,1]$, 
 square completion yields
 \[
  \begin{split}
   \mathcal E_L(h)
      -\sqrt{A_{\min}}E_0
   =\tfrac12\int_\R
      \left(h_x+\sqrt{2A_{\min}G(h)}\right)^2\,dx
      +\int_\R(A(x/L)-A_{\min})G(h)\,dx.
  \end{split}
 \]
 A frozen-front profile centered at a minimum of $A$ gives an upper bound $\sqrt{A_{\min}}E_0+O(L^{-2})$ for the basic transition energy.
 Both terms on the right are therefore $O(L^{-2})$ for any basic minimizer. 
 The first excludes an additional excursion of fixed height. 
 The stationary ODE and $0\leq h\leq1$ give $\|h_x\|_\infty\leq C$ uniformly in $L$. 
 Thus, 
 around any half-level point $x_*$, 
 there is an $L$-independent interval on which $1/4\leq h\leq3/4$ and $G(h)\geq1/2$.
 Since $A(x/L)$ varies by $O(L^{-1})$ on this interval, 
 the $O(L^{-2})$ bound on the second integral implies $A(x_*/L)\to A_{\min}$.
 ODE compactness after centering at a half-level point and the vanishing square term give local $C^1$ convergence to the decreasing frozen-front profile at $A_{\min}$.
 A return from a stable tail neighborhood would require an increasing excursion of fixed height,
 contradicting the same square-term bound.
 Thus both tails stay in fixed stable neighborhoods beyond a uniformly bounded distance from the center.
 Comparison in $h_{xx}=A(x/L)G'(h)$ then gives uniform exponential tails,
 including those of $h_x$.
 These bounds upgrade local convergence to weighted $H^1$ convergence,
 placing each basic minimizer in the single-interface tube \eqref{lp:tube}.
 Since it is stationary,
 the complete-orbit assertion of Theorem~\ref{lp:main} puts it on $\Mcal_{1/L}$, 
 so its phase is a zero of $s_{1/L}$.
 The other zero per period lies near a maximum of $A$ and has energy $\sqrt{A_{\max}}E_0+o(1)$, 
 so the upper bound above excludes it as a minimizer. 
 Hence all basic minimizers are spatial-period translates of the front at the minimum-phase zero.
 In the rescaled coordinate, 
 their traces on a fixed unit cell form a discrete sequence with only the constant states $0$ and $1$ as accumulation points. 
 The symmetry $u\mapsto1-u$ gives the same conclusion for the reverse transition. 
 Hence alternative (ii) of \cite[Proposition~2.43]{MR2016} holds in both orientations.
 A transition to a nonadjacent integer well crosses an adjacent one and has larger positive action.
 The two adjacent costs agree by $u\mapsto2-u$.
 Hence $c_1(1)=C_1(1,0)$.

 For $k=3$ and $(v,w)=(1,0)$, use the constrained class $\mathcal A_3$ of \cite[(3.1), (3.26)]{MR2016}.
 Choose six ordered blocks with centers $m_1,\ldots,m_6\in\mathbb Z$ and width $l>0$ such that
 \[
  l\geq m_0,\qquad m_{i+1}-m_i-6l\geq m_0
 \]
 for the adjacent centers specified in \cite[(3.28)]{MR2016},
 with successive trace constraints near $1,0,0,1,1,0$.
 The equal-cost and isolation conditions verified above allow us to apply \cite[Theorem~3.27]{MR2016}.
 It yields a constrained action minimizer that is a classical heteroclinic from $1$ to $0$
 with three separated transitions $1\to0\to1\to0$.
 By \cite[Proposition~3.29]{MR2016}, 
 sufficiently large separation makes the solution $H^1$-close to the respective constants on intermediate blocks in the rescaled coordinate $y=x/L$. 
 In one dimension it is also uniformly close there.
 In the original $x$ coordinate, 
 the three prescribed transitions have action at least $3\sqrt{A_{\min}}E_0-o(1)$, 
 while gluing three basic minimizers far apart gives an upper bound $3\sqrt{A_{\min}}E_0+O(L^{-2})$ for the three-transition constrained minimizer. 
 Any excursion outside $[0,1]$ must reach $3/2$ or $-1/2$: a maximum in $(1,3/2)$ or a minimum in $(-1/2,0)$ contradicts $u_{xx}=A(x/L)G'(u)$. 
 By square completion,
 reaching either level and returning costs at least another $2\sqrt{A_{\min}}\int_1^{3/2}\sqrt{2G(u)}\,du-o(1) =\sqrt{A_{\min}}E_0-o(1)$, contradicting the upper bound for large $L$. 
 Thus $0\leq u\leq1$, 
 and the strong maximum principle gives $0<u<1$. 
 Its three transitions place it outside the single-interface tube and hence outside $\Mcal_{1/L}$.
 To prove \eqref{comp:no-global-attraction},
 suppose that $\|u-Z_{1/L,X_n}\|_\infty\to0$ for a sequence of phases $X_n$.
 Fix $x_0$ with $u(x_0)=1/2$.
 The uniform tails in \eqref{lp:uniform-tail-bound} force $(X_n)$ to be bounded.
 After passing to a subsequence, $X_n\to X_*$,
 and continuity of the parametrization in the uniform norm gives $u=Z_{1/L,X_*}$,
 a contradiction.
 Taking this stationary profile as initial datum proves the failure of global attraction.
\end{proof}

\section{Realization of arbitrary smooth circle flows}\label{sec:realization}

Sections~\ref{sec:existence} and \ref{sec:dynamics} describe front dynamics through a scalar periodic flow.
We now solve the inverse problem posed in Theorem~\ref{real:circle-flow}. 
Starting from a prescribed smooth circle flow, 
we construct a locally bistable reaction and an invariant front family whose quotient dynamics are conjugate to that flow.
The construction uses the large-period family when the prescribed flow has singularities, 
and a homogeneous traveling-front family when it is rotational.

\subsection{Construction and time-preserving conjugacy}

\begin{proof}[Proof of Theorem~\ref{real:circle-flow}]
Fix $q\in C^\infty(\T)$. 
We use the reference reaction and front
\begin{equation}\label{real:logistic-front}
 g(u)=u(1-u)(2u-1),\qquad
 U(z)=\frac{1}{1+e^z},\qquad U_p(x)=U(x-p),
\end{equation}
and write $P(u)=u(1-u)$, 
so that $g=-PP'$ and $U'=-P(U)$.
Choose a nonnegative $\chi\in C^\infty(\R)$ with $\supp\chi\subset(1/3,2/3)$ and $\chi(1/2)>0$.
All reactions below are given a smooth stable extension outside
$[0,1]$ as in Subsection~\ref{intro:assumptions}.

\medskip
\noindent\emph{Step 1: homogeneous realizations of the stationary and rotational flows.}
If $q\equiv0$, 
take $f=g$ and $L=1$.
The family of translates $U_p$ is stationary, 
and $\mathcal H(\theta)=[U_\theta]$ gives \eqref{real:conjugacy}.

If $q$ is nowhere zero, 
let $\sigma$ be its sign and put
\[
 T_q=\int_0^1\frac{d\theta}{|q(\theta)|}.
\]
Choose a sufficiently small $\delta>0$ and the homogeneous bistable reaction $f=g+\delta\sigma\chi$.
Let $(\varphi_\sigma,c_\sigma)$ be its normalized decreasing traveling-front profile--speed pair. 
The front equation, 
multiplied by $\varphi_\sigma'$ and integrated, 
gives
\[
 c_\sigma\int_\R|\varphi_\sigma'|^2\,dx
 =\int_0^1f(u)\,du
 =\delta\sigma\int_0^1\chi(u)\,du,
\]
so $c_\sigma$ has sign $\sigma$.
Since a homogeneous equation is periodic with any positive period,
choose
\[
 L=|c_\sigma|T_q,\qquad
 Z_p(x)=\varphi_\sigma(x-p).
\]
The induced circle flow on this family has constant velocity $c_\sigma/L=\sigma/T_q$, 
and hence the same minimal period as $\Phi_t^q$. 
The increasing lift
\[
 \beta(\theta)=\frac{\sigma}{T_q}
              \int_0^\theta\frac{d\zeta}{q(\zeta)}
 \quad\text{satisfies}\quad
 \beta(\theta+1)=\beta(\theta)+1,\qquad
 \beta'(\theta)q(\theta)=\frac{\sigma}{T_q}.
\]
It follows that $\mathcal H(\theta)=[Z_{L\beta(\theta)}]$ is a time-preserving conjugacy.
Both homogeneous translation families are strictly ordered and normally exponentially attracting: 
the translation-mode spectral gap and the phase--normal contraction in the homogeneous case of Section~\ref{sec:large-complete} give these assertions.

It remains to treat a nontrivial flow with singularities. 
For the rest of the proof, 
write
\[
 K=q^{-1}(0),\qquad \emptyset \ne K\ne\T,
\]
and let $\sigma_I$ be the sign of $q$ on a component $I$ of $\T\setminus K$.

\medskip
\noindent\emph{Step 2: a smooth locally bistable reaction encoding the prescribed flow.}
Define
\begin{equation}\label{real:flat-encoding}
 H(y)=
 \begin{cases}
  \exp\bigl(-1/q(y)^2\bigr),&q(y)\ne0,\\
  0,&q(y)=0,
 \end{cases}
 \qquad
 Q(y)=
 \begin{cases}
  \operatorname{sgn}(q(y))H(y),&q(y)\ne0,\\
  0,&q(y)=0.
 \end{cases}
\end{equation}
Both functions are smooth and flat on $K$.
Indeed, 
they are compositions of $q$ with smooth functions on $\R$ that are flat at the origin.
Moreover, 
$H>0$ off $K$ and $Q$ has the sign of $q$ there.
Put
\[
 r(u)=\log\frac{1-u}{u}\qquad(0<u<1),
\]
and, 
for $0\leq\lambda\leq1$, 
set
\begin{equation}\label{real:reaction-family}
 f_\lambda(y,u)
 =g(u)+\delta\chi(u)H(y)Q(y-\lambda r(u)).
\end{equation}
The support of $\chi$ makes this a smooth family on $\T\times[0,1]$, 
with uniform $C^4$ bounds.
Fix $\delta>0$ sufficiently small. 
The strict signs of $g$ away from $1/2$, 
together with $g'(1/2)>0$, 
imply that each $f_\lambda(y,\cdot)$ has a unique intermediate zero and the strict bistable signs in Assumption~\ref{intro:large-assumption}.
The reaction agrees with $g$ near $0$ and $1$, 
so the stability constants at $0$ and $1$ can be chosen uniformly in $\lambda$.

The construction of Theorem~\ref{lp:main} and the ordering in Lemma~\ref{phase:ordering} apply uniformly to this family.
Indeed, 
the frozen-front implicit-function,
spectral,
time-map,
and ordering estimates are locally uniform in $\lambda$.
Compactness of $[0,1]$ and the uniform bistability and linear stability at $0$ and $1$ just verified give a common threshold $\epsilon_0>0$.
Choose
\begin{equation}\label{real:parameter-choice}
 L>\max\{1,\epsilon_0^{-1}\},\qquad
 \lambda_*=L^{-1},\qquad f=f_{\lambda_*}.
\end{equation}
The large-period invariant family for this reaction will be denoted by
\[
 Z_p=Z_{1/L,p},\qquad \mathcal M=\mathcal M_{1/L}.
\]
Its exact position equation is $\dot p=s_{1/L}(p/L)$.
The normalized circle velocity is
\begin{equation}\label{real:normalized-velocity}
 \omega(Y)=\frac1L s_{1/L}(Y),\qquad Y\in\T.
\end{equation}
Theorem~\ref{lp:main} already supplies normal exponential attraction,
uniform front tails,
and complete trajectories on this family.

\medskip
\noindent\emph{Step 3: identification of the orbit intervals.}
Since $U''+g(U)=0$ and $r(U_p(x))=x-p$, direct substitution gives
\begin{equation}\label{real:exact-profile-residual}
 U_p''+f(x/L,U_p)
 =\delta\chi(U_p)H(x/L)Q(p/L).
\end{equation}
Thus $U_p$ is stationary if and only if $p/L\pmod1\in K$.
Sufficiency is immediate; 
necessity follows by evaluating the residual at $x=p$.
At a phase in $K$ the frozen reaction is $g$, 
so $U_p$ has phase $p$ and zero normal coordinate in the tube \eqref{lp:tube}.
Local completeness in Theorem~\ref{lp:main} and uniqueness of the phase chart in Lemma~\ref{lp:chart-unique} therefore give $Z_p=U_p$.

To exclude extra singularities of the induced flow, 
let $(p_-,p_+)$ be a component of the complement of the lifted set
\[
 \widehat K_L=\{p\in\R:p/L\pmod1\in K\},
\]
and denote its assigned sign by $\sigma_I$.
If some $X\in(p_-,p_+)$ were stationary, 
strict ordering would give
\[
 U_{p_-}(x)<w(x):=Z_X(x)<U_{p_+}(x).
\]
Hence, with
\[
 \vartheta(x)=x-r(w(x)),\qquad
 \mathcal R(x)=\delta\chi(w(x))H(x/L)Q(\vartheta(x)/L),
\]
we have $p_-<\vartheta(x)<p_+$ and $\sigma_I\mathcal R(x)\geq0$.
At any point $x_0$ where $w(x_0)=1/2$, 
one has $\vartheta(x_0)=x_0\in(p_-,p_+)$, 
so $\sigma_I\mathcal R(x_0)>0$.
The stationary equation is $w''+g(w)+\mathcal R=0$.
For $B=w'+P(w)$ it gives
\[
 B'=P'(w)B-\mathcal R.
\]
The spatial limits and exponential tails yield $B(\pm\infty)=0$,
while $P'(w)\to-1$ at $-\infty$ and $P'(w)\to1$ at $+\infty$.
Variation of constants from both ends gives
\[
 \begin{aligned}
 B(x)&=-\int_{-\infty}^x
  \exp\left(\int_s^xP'(w(\tau))\,d\tau\right)\mathcal R(s)\,ds,\\
 B(x)&=\int_x^\infty
  \exp\left(-\int_x^sP'(w(\tau))\,d\tau\right)\mathcal R(s)\,ds.
 \end{aligned}
\]
The first formula implies $\sigma_I B\leq0$, 
and the second implies $\sigma_I B\geq0$.
Thus $B\equiv0$ and then $\mathcal R\equiv0$, 
a contradiction.
Consequently, 
$\omega^{-1}(0)=K$, 
and the orbit intervals of
$\omega$ and $q$ agree.

\medskip
\noindent\emph{Step 4: orientation of the induced flow.}
Set $\mathcal F_L(x,u)=\int_0^u f(x/L,s)\,ds$ and define
\begin{equation}\label{real:renormalized-energy}
 \mathcal E_L(w)=\int_\R
 \left\{\frac12|w_x|^2-\mathcal F_L(x,w)
       +\mathcal F_L(x,1)\mathbf1_{(-\infty,0)}(x)\right\}\,dx.
\end{equation}
The primitive of $g$ is $-\tfrac12P^2$, 
the perturbation is supported in the interior of the $u$ interval, 
and the profiles have exponential tails. 
Thus this energy is finite on both $\mathcal M$ and the comparison family $U_p$.
Along a classical orbit on $\mathcal M$, 
integration by parts gives
\begin{equation}\label{real:energy-dissipation}
 \frac{d}{dt}\mathcal E_L(u(t))
 =-\int_\R|u_t|^2\,dx\leq0.
\end{equation}
Along the comparison family, 
\eqref{real:exact-profile-residual} gives
\begin{equation}\label{real:comparison-energy-slope}
 \begin{aligned}
 \frac{d}{dp}\mathcal E_L(U_p)&=-\delta Q(p/L)A_L(p),\\
 A_L(p)&=\int_\R\chi(U_p)H(x/L)(-U_p')\,dx.
 \end{aligned}
\end{equation}
If $p/L\pmod1\notin K$, then $A_L(p)>0$, 
as can be seen near $x=p$.
For a complementary lifted interval $(p_-,p_+)$ it follows that
\begin{equation}\label{real:ordered-stationary-energies}
 \sigma_I\bigl(\mathcal E_L(U_{p_+}) - \mathcal E_L(U_{p_-})\bigr)<0.
\end{equation}
By Step 3 the exact velocity has a constant strict sign in this interval. Theorem~\ref{phase:classification} gives a complete orbit connecting the two stationary phases $p_-$ and $p_+$ in the direction of that sign.
Weighted $H^1$ continuity of the profiles and boundedness of the phase on the interval give convergence of the energies at both ends.
The dissipation identity and \eqref{real:ordered-stationary-energies} force this direction to have sign $\sigma_I$.
We have proved
\begin{equation}\label{real:matching-orbits}
 \omega^{-1}(0)=q^{-1}(0),\qquad
 \omega(Y)q(Y)>0\quad\text{whenever }q(Y)\ne0.
\end{equation}

\medskip
\noindent\emph{Step 5: a time-preserving conjugacy of the complete flows.}
We now construct the time-preserving conjugacy.
Use the periodic lifts of $q$ and $\omega$ to $\R$.
Let
\[
 \widetilde K=\{\theta\in\R:\theta\pmod1\in K\}
\]
be the lifted zero set. 
On each component $I=(a,b)$ of $\R\setminus\widetilde K$, 
choose the midpoint $\theta_I$, 
and define
\[
 A_I(\theta)=\int_{\theta_I}^{\theta}\frac{d\zeta}{q(\zeta)},
 \qquad
 B_I(Y)=\int_{\theta_I}^{Y}\frac{d\zeta}{\omega(\zeta)}.
\]
Since $q$ is smooth and $\omega$ is $C^1$, 
neither flow can reach a boundary point in finite time.
Both maps are monotone bijections from $I$ onto $\R$, 
with the same orientation. 
Therefore
\begin{equation}\label{real:phase-conjugacy}
 \beta_I=B_I^{-1}\circ A_I
\end{equation}
 is an increasing homeomorphism of $I$ onto itself and fixes its boundary points by continuous extension.
Set $\beta=\beta_I$ on each such interval and $\beta(\theta)=\theta$ on
$\widetilde K$.
This defines an increasing homeomorphism of $\R$ with $\beta(\theta+1)=\beta(\theta)+1$.
Continuity at an accumulation of intervals follows 
because $\beta$ maps every interval into itself 
and the lengths of distinct intervals in a compact set tend to zero.
The same argument applies to the inverse.
Thus $\beta$ induces an orientation-preserving homeomorphism of $\T$.

Let $\Theta_t^\omega$ be the flow of $\dot Y=\omega(Y)$.
The integral coordinates increase at unit rate along their respective flows, 
so
\[
 B_I\bigl(\Theta_t^\omega(\beta_I(\theta))\bigr)
 =A_I(\theta)+t
 =A_I\bigl(\Phi_t^q(\theta)\bigr).
\]
Hence
\[
 \beta\circ\Phi_t^q=\Theta_t^\omega\circ\beta
 \qquad(t\in\R),
\]
including points in $K$, 
where both flows are stationary.

Finally, 
covariance and the embedded parametrization identify $\widehat{\mathcal M}$ with the phase circle by $Y\mapsto[Z_{LY}]$.
The map
\[
 \mathcal H(\theta)=[Z_{L\beta(\theta)}]
\]
is an orientation-preserving homeomorphism and satisfies \eqref{real:conjugacy}.
The front-tail estimates and Theorem~\ref{phase:classification} show that every complete trajectory on $\mathcal M$ is a transition front. 
This completes the realization.
\end{proof}

\subsection{Sensitivity to the spatial period}\label{real:sensitivity-section}

\begin{proof}[Proof of Proposition~\ref{real:period-sensitivity}]
Since $q\geq0$, the encoding in \eqref{real:flat-encoding} has $Q=H\geq0$.
Keep $f=f_{1/L_0}$ from \eqref{real:reaction-family} fixed.
The uniform threshold in the proof of Theorem~\ref{real:circle-flow}
gives its normally attracting invariant family for every period $\Lambda$ sufficiently close to $L_0$.
At $L_0$, \eqref{real:matching-orbits} identifies the stationary phases with $K$.
If a profile $w$ on this family were stationary, its equation would be
\[
 w''+g(w)+\mathcal R_\Lambda=0,\qquad
 \mathcal R_\Lambda(x)
 =\delta\chi(w)H(x/\Lambda)
 Q\bigl(x/\Lambda-r(w)/L_0\bigr)\geq0.
\]
The two variation-of-constants formulas in Step 3 imply $\mathcal R_\Lambda\equiv0$.
Choose an open interval $J$ around a half-level point where $\chi(w)>0$.
Since $K$ is nowhere dense, the points with $H(x/\Lambda)>0$ are dense in $J$.
Hence continuity gives $x/\Lambda-r(w(x))/L_0\in\widetilde K$ throughout $J$,
where $\widetilde K=\{\theta\in\R:\theta\pmod1\in K\}$.
The set $\widetilde K$ is totally disconnected, so this expression is constant on $J$.
Consequently $w(x)=U((L_0/\Lambda)x-C)$ there.
Substitution into $w''+g(w)=0$ gives
$[(L_0/\Lambda)^2-1]U''=0$ on $J$,
which is impossible if $\Lambda\ne L_0$.
Thus the exact phase velocity has no zeros for all sufficiently close $\Lambda\ne L_0$,
and Theorem~\ref{phase:classification} gives the asserted rotation flow and propagating front.
\end{proof}






\end{document}